\documentclass[reqno]{amsproc}
\usepackage{amssymb}
\usepackage{amsmath}
\usepackage{mathdots}
\usepackage{amsbsy}
\usepackage{amscd}
\usepackage{amsthm}

\usepackage{url}
\usepackage{xcolor}
\usepackage{amsmath,amssymb,amsthm,amsfonts,longtable}
\usepackage[english]{babel}
\usepackage{tikz-cd}
\usetikzlibrary{cd}
\usetikzlibrary{decorations.markings}
\tikzset{negated/.style={
decoration={markings,
mark= at position 0.5 with {
\node[transform shape] (tempnode) {$\times$};
}
},
postaction={decorate}
}
}
\usepackage{array}
\usepackage[colorlinks,citecolor=red,urlcolor=blue,bookmarks=false,hypertexnames=true]{hyperref}

\newtheorem{definition}{Definition}[section]
\newtheorem{lemma}[definition]{Lemma}

\newtheorem{theorem}[definition]{Theorem}

\newtheorem{remark}[definition]{Remark}

\newtheorem{example}[definition]{Example}
\newcommand{\restr}{\mathord\downarrow} 
\newcommand{\ind}{\mathord\uparrow} 
\newcommand{\Irr}{\textnormal{Irr}}
\newcommand{\cd}{\textnormal{cd}}
\newcommand{\nl}{\textnormal{nl}}
\newcommand{\lin}{\textnormal{lin}}
\newcommand{\Core}{\textnormal{Core}}
\newcommand{\inertiagroup}{\textnormal{I}}

\renewcommand{\SS}{P}
\newcommand{\TT}{Q}
\newcommand{\NAB}{H}
\newcommand{\AB}{K}
\newcommand{\SG}{R}
\newcommand{\SGP}{S}
\newcommand{\SGM}{T}
\newcommand{\MTG}{U}

\title[Faithful permutation representations]
{Minimal degrees  for faithful permutation representations \\ 
of groups of order $p^6$ where $p$ is an odd prime} 

\author{E.A.\ O'Brien}
\address{Department of Mathematics, University of Auckland, Auckland, New Zealand.}
\email{e.obrien@auckland.ac.nz}
\author{Sunil Kumar Prajapati$^*$}
\address{Indian Institute of Technology Bhubaneswar, Arugul Campus, Jatni, 
Khurda-752050, India.}
\email{skprajapati@iitbbs.ac.in}
\author{Ayush Udeep}
\address{Indian Institute of Technology Bhubaneswar, Arugul Campus, 
	Jatni, Khurda-752050, India.}
\email{udeepayush@gmail.com}
\thanks{$^{\textbf{*}}$Corresponding author.
}

\begin{document}
\maketitle

\begin{abstract}
We determine the minimal degree of a faithful permutation 
representation for each group of order $p^6$ where $p$ is an odd prime. 
{We also record how to obtain such a representation.}
\end{abstract}
	
\section{Introduction}
The minimal faithful permutation degree $\mu(G)$ of a finite group
$G$ is the least positive integer $n$ such that $G$ is isomorphic to 
a subgroup of the symmetric group $S_n$.
Of course, $\mu(G)$ is the minimal degree 
of a faithful representation of $G$ by permutation matrices. 
A {\it quasi-permutation matrix} is a square matrix over the complex field 
$\mathbb{C}$ with non-negative integral trace. 
Wong \cite{W} defined a
{\it quasi-permutation representation} of $G$ as  
one consisting of quasi-permutation matrices.
Let $c(G)$ denote the minimal degree of 
a faithful quasi-permutation representation of $G$. 
It is easy to deduce that $c(G)\leq \mu(G)$, so $c(G)$ is often
important in computing $\mu(G)$.

Both $c(G)$ and $\mu(G)$ have been 
studied extensively: see, for example, 
\cite{AB, HB, HB1997, BGHS, Britnell2017, EasdownSaunders2016, 
Elias2010, GA, 
SaundersThesis, 
DW}.

As is well known (see, for example, \cite[Theorem 2]{DLJ}), 
if $G \cong C_{p^{n_1}} \times \ldots \times C_{p^{n_k}}$, then  
$\mu(G) = \sum_{i=1}^k p^{n_i}$.
Behravesh and Ghaffarzadeh \cite{BGgroup64} computed 
$\mu(G)$ for the groups of order dividing $2^6$ using {\sf GAP}.
In \cite{BD, S}, the values of $\mu(G)$ are determined for the 
groups of order dividing $p^{5}$ where $p$ is an odd prime.

We compute $\mu(G)$ for 
the non-abelian groups of order $p^{6}$ where $p$ is an odd prime. 
Newman, O'Brien and Vaughan-Lee 
\cite{NO'bV2004} establish that there are  
\[ 3p^2 + 39 p + 344 + 24(p-1,3) + 11 (p-1,4) + 2 (p-1,5) \]
groups of order $p^{6}$ where $p \geq 5$; there are
504 groups of order $3^6$.  They provide a 
database of parameterized presentations for the groups of order 
$p^6$; it is available publicly in {\sf GAP} \cite{GAP} and 
{\sc Magma} \cite{MAGMA}. 
We reference extensively the list of presentations 
and other data available at \cite{Arxiv}, and use its group identifiers.

Our main result is the following.
\begin{theorem}\label{main-theorem}
Let $G$ be a non-abelian group of order $p^{6}$ where $p\geq 5$ is prime. 
\begin{enumerate}
\item If $|Z(G)|= p^4$, then $c(G) = p^4+p, p^3+p^2+p, p^3+p^2, p^3+2p, 
3p^2, 2p^2+2p, 2p^2+p, 3p^2+p, p^2+3p, p^5, p^4+p^2, 2p^3$, or  $p^3+2p^2$.
\item If $|Z(G)| = p^3$, then $c(G) = p^5$, $p^4 + p^3$, $p^4+p^2$, 
$p^4+p$, $p^4$, $2p^3+p^2$, $2p^3+p$, $2p^3$, $p^3+2p^2$, $p^3+p^2+p$, 
$p^3 +p^2$, $p^3+2p$, $p^3+p$, $3p^2$, $2p^2+p$, $2p^2$, or $p^2+2p$.
\item If $|Z(G)| = p^2$, then $c(G) =p^4$, $2p^3$, $p^3 +p^2$, $p^3+p$, 
$p^3$, $2p^2$, or $p^2+p$.
\item If $|Z(G)| = p$, then $c(G) =p^4$, $p^3$, or $p^2$.
\end{enumerate}
\end{theorem}

Behravesh and Ghaffarzadeh 
\cite[Theorem 3.2]{BG} proved that 
$c(G) = \mu(G)$ for a $p$-group $G$ of odd order, {so the 
theorem also establishes $\mu(G)$.}

{We not only determine $\mu(G)$ for each non-abelian
group $G$ of order $p^6$, but we 
also record how to realize a minimal degree faithful permutation 
representation for $G$.  Our results -- and related code -- are available 
publicly in {\sc Magma} via a GitHub repository \cite{github}, 
and can be used to construct explicitly this representation of $G$ for 
a given prime $p$.} 

The groups of order $p^6$ are classified into 43 isoclinism
families labelled $\Phi_{i}$ for $1 \leq i \leq 43$. All abelian
groups are in $\Phi_1$.
The classifications of \cite{Easterfield}, 
\cite{RJ}, and \cite{NO'bV2004} exploit
(and agree on the details of) this partition. 
Invariants for these families are tabulated in \cite{RJ}; 
those used in our study include the structures of the derived group
and the central quotient, and the character degrees. 

In Section \ref{section:BR}, we summarize notation and 
background results. Prajapati and Udeep \cite[Theorem 26]{SA}
determine $c(G)$, and so $\mu(G)$, 
for every group of odd order $p^6$ having center of order $p^4$.
We complete this task for the remaining non-abelian groups of order $p^6$. 
To prove Theorem \ref{main-theorem}, we exploit extensively the 
equality of $\mu(G)$ and $c(G)$. For some groups
we determine $\mu(G)$ directly; for the remainder
we determine $c(G)$. Our choice is dictated by 
the ease of the associated computations.
In Section \ref{Slessthanp4} we calculate 
$\mu(G)$ for some groups of order $p^6$ and center of
order dividing $p^3$. 
In Sections \ref{SS3}--\ref{SS5} we obtain all possible 
values of $c(G)$ for each isoclinism family
containing groups with center of order dividing $p^{3}$. 
By refining these values, we then compute $c(G)$ 
for all outstanding groups in Section \ref{subsection:tables}. 
We summarize our results in Tables \ref{t:VZ}--\ref{t:7}.
For completeness, we compute $\mu(G)$ for the groups of order 
$3^{6}$ using {\sf GAP}; we summarize the results in Table \ref{t:729}.
We discuss the GitHub repository in Section \ref{section:access}.
	
\section{Notation and background}\label{section:BR}

\subsection{Notation}
For a finite group $G$, we use the following notation throughout the paper. 

\begin{tabular}{cl}
	$d(G)$ & the minimal number of generators of $G$\\
	$o(x)$ & the order of $x\in G$\\
	$G'$ & the commutator subgroup of $G$\\
	$\exp(G)$ & the exponent of $G$\\
	$\Core_{G}(H)$ & the core in $G$ of $H \leq G$\\
	$\Irr(G)$ & the set of irreducible complex characters of $G$\\
	$\lin(G)$ & the set of linear characters of $G$\\
	$\nl(G)$ & the set of non-linear irreducible characters of $G$\\
	$\cd(G)$ & $\{ \chi(1) ~|~ \chi \in \Irr(G) \}$\\
	$\inertiagroup_{G}(\chi)$ & the inertia group of $\chi \in \Irr(G)$ 
in $G$ \\
	$\mathbb{Z}_{\geq 0}$ & the set of all non-negative integers\\
	$C_{p}^{k}$ & 
the elementary abelian $p$-group of rank $k \geq 3$\\
	$G_{(r,s)}$ & 
the group with number $s$ in isoclinism family $\Phi_{r}$ \\
	$\omega$ & the smallest positive integer which is a primitive root modulo $p$ \\
	$\nu$ & the smallest positive integer which is a quadratic non-residue modulo $p$\\
	$\phi(n)$ & the Euler phi function\\
	$(a, b)$ & the greatest common divisor of the positive integers $a$ and $b$
\end{tabular}

\subsection{Background}
We summarize some results regarding $\mu(G)$ and $c(G)$, which 
we use throughout. In this section, unless otherwise 
indicated, $p$ is an arbitrary prime.

It is well known that every permutation representation of $G$ is a direct sum
of transitive permutation representations, and every transitive
representation is equivalent to a representation
of the set of cosets of some subgroup of $G$.
Hence we identify a faithful representation of $G$ with 
the set of 
subgroups $\{H_1, \ldots, H_m\}$ determining 
its {\it transitive constituents}; 
the degree is $\sum_{i=1}^m |G:H_i|$, and $\cap_{i=1}^m \Core_G (H_i) = 1.$
Now $\mu(G)$ is the minimum degree over
all such sets of subgroups. 

\begin{lemma} \textnormal{\cite[Theorem 3]{DLJ}} \label{L2} 
Let $G$ be a $p$-group whose center is minimally generated by $d$ elements, 
and let $\{H_{1},\ldots,H_{m} \}$ determine a minimal 
degree faithful permutation representation of $G$. 
If $p$ is odd, then $m=d$; if $p=2$, then $d/2\leq m\leq d$, 
and the upper bound is achieved.
\end{lemma}
	
\begin{lemma}
\textnormal{\cite[Corollary 2.2]{DW}}\label{thm:nilpotent} 
If $H$ and $K$ are non-trivial nilpotent groups, then 
$\mu(H \times K)= \mu(H)+ \mu(K)$.
\end{lemma}

Suppose $F$ is a subfield of a splitting field $E$ for $G$. 
Let $\chi, ~\psi \in \Irr_E(G)=\Irr(G)$. 
Let $F(\chi)$ be the field obtained by adjoining to $F$ the 
values $\chi(g)$ for all $g \in G$, and 
let $\Gamma(F(\chi)/F)$ denote the associated Galois group. 
Now $\chi$ and $\psi$ are Galois conjugate 
over $F$ if $F(\chi)=F(\psi)$ and there 
exists $\sigma \in \Gamma(F(\chi)/F)$ such that $\chi^{\sigma}= \psi$.
Clearly, Galois conjugacy defines an 
equivalence relation on $\Irr(G)$; if $\mathcal{C}$ is 
the equivalence class of $\chi$ with respect to Galois conjugacy over $F$, 
then $|\mathcal{C}|=| F(\chi) : F |$ (see \cite[Lemma 9.17]{I}). 

\begin{lemma}\textnormal{\cite[Theorem 1]{FORD}} \label{thm:ford}
Let $G$ be a $p$-group and let $\chi$ be an irreducible complex character 
of $G$.  One of the following holds:
\begin{enumerate}
\item [\rmfamily(i)]
There exists a linear character $\lambda$ on a subgroup 
$H$ of $G$ which induces $\chi$ and generates the same field 
as $\chi$: namely, 
$\lambda\ind_{H}^{G} = \chi$ and $\mathbb{Q}(\lambda) = \mathbb{Q}(\chi)$.
\item [\rmfamily(ii)] $p=2$ and there exist subgroups $H < K$ of $G$ with 
$|K/H| = 2$ and a linear character $\lambda$ of $H$ such that 
$\lambda\ind_{H}^{K} = \eta \in \Irr(K)$, 
$[\mathbb{Q}(\lambda): \mathbb{Q}(\eta)] = 2$, 
$\eta\ind_{K}^{G} = \chi$, and $\mathbb{Q}(\eta) = \mathbb{Q}(\chi)$.
\end{enumerate}
\end{lemma}

\begin{definition} \label{D1}
Let $G$ be a finite group and let $\chi \in \Irr(G)$.
Define $\Gamma(\chi)$ to be the Galois group of $\mathbb{Q}(\chi)$ 
over $\mathbb{Q}$ and $d(\chi) := |\Gamma(\chi)| \cdot \chi(1)$.
\end{definition}

\begin{lemma} \textnormal{\cite[Lemma 2.2]{BG}} \label{lemma:c(G)Algorithm}
Let $G$ be a finite group.  Let $X \subset 
\Irr(G)$ be such that $\cap_{\chi \in X} \ker (\chi)= 1$ 
and $\cap_{\chi \in Y} \ker (\chi) \neq 1$ 
for every proper subset $Y$ of $X$. 
Let $\xi_X = \sum_{\chi \in X} \left[ \sum_{\sigma \in \Gamma(\chi)}
\chi^{\sigma}  \right]$ and let $m(\xi_X)$ be the absolute value of 
the minimum value that $\xi_X$ takes over $G$.
Then $$c(G) = \min \{\xi_X(1) + m(\xi_X) \; | \; X \subset \Irr(G) 
\text{\ satisfying\ the\ above\ property} \}.$$
\end{lemma}

\noindent 
We identify $X_G\subset \Irr(G)$ with a minimal degree 
faithful quasi-permutation representation of $G$ if 
\begin{equation}\label{eq:X_G} 
\bigcap_{\chi \in X_G} \ker (\chi) = 1 \text{ and } 	
\bigcap_{\chi \in Y} \ker (\chi) \neq 1 
\text{ for every } Y \subset X_G
\end{equation}
and $c(G) = \xi_{X_G}(1) + m(\xi_{X_G})$.

\begin{lemma}\textnormal{\cite[Theorem 2.3]{GA}} \label{L1}
Let $G$ be a $p$-group whose center is minimally
generated by $d$ elements and let $X_{G}$ be a minimal degree faithful
quasi-permutation representation of $G$ as defined in Equation 
$(\ref{eq:X_G})$.
\begin{enumerate}
\item [\rmfamily(i)] $|X_G| = d$.
\item [\rmfamily(ii)] $m(\xi_{X_G}) = \frac{1}{p-1} \sum_{\chi \in X_{G}}
 \left[ \sum_{\sigma \in \Gamma(\chi)} 
\chi^{\sigma}(1)  \right].$
\end{enumerate}
\end{lemma}

\begin{lemma} \textnormal{\cite[Theorem 2.11]{HB1997}} \label{L3}
Let $G \cong \prod_{i = 1}^k C_{m_{i}}$ 
where $C_{m_i}$ is cyclic of prime power order $m_i$.
If $n$ is maximal such that $G$ has a direct factor isomorphic to 
$C_{6}^{n}$, then $c(G)=T(G)-n$ where $T(G)=\sum_{i=1}^k m_{i}.$
\end{lemma}
\begin{lemma} \label{thm:mu(G)=c(G)}
\textnormal{\cite[Theorem 3.2]{BG}}
\label{thm:pgroup}
If $G$ is a finite $p$-group of odd order, then $c(G) = \mu(G)$.
\end{lemma}
Further, if $G$ is an abelian $p$-group of odd order, 
then $c(G)=\mu(G)=T(G)$ (see \cite{BGHS}). 

Lemma \ref{lemma:c(G)Algorithm} shows that 
the character degrees of $G$ play an important role in computing $c(G)$. 
We mention results which assist in finding $\cd(G)$.
\begin{lemma}\textnormal{\cite[Theorem 12.11]{I}} \label{thm:cdG1}
Let $G$ be a non-abelian group and let $p$ be a prime. 
Then $\cd(G) = \{ 1, p \}$ if and only if one of the following holds:
\begin{enumerate}
\item $G$ has an abelian normal subgroup of index $p$; 
\item $|G:Z(G)| = p^3$.
\end{enumerate}
\end{lemma}
\begin{lemma}\textnormal{\cite[Theorem B]{AM}} \label{thm:cdG2}
Let $G$ be a non-abelian $p$-group with minimal degree $d$. Then 
$G'$ has a normal subgroup $N$ of index $p$ and the 
character degrees of $G/N$ are $1$ and $d$.
\end{lemma}

\begin{lemma} \textnormal{\cite[{Theorem 2}]{SAcyclic}} 
\label{lemmalinearcharacter}
Let $G$ be a non-abelian $p$-group. 
If $d(Z(G)  \cap G') = d(Z(G))$, 
then $X_{G} \cap \lin(G) = \emptyset$ and 
$p^{s+1}~|~ c(G)$ where $p^s = \min\{ \chi(1)~|~ \chi \in \nl(G) \}$.
\end{lemma}

\begin{lemma} 
\textnormal{\cite[{Remark 20}]{SAcyclic}} 
\label{P2}
Let $G$ be a non-abelian group of order $p^{6}$.
Assume $d(Z(G)) = m$ where $1 \leq m \leq 4$, $\exp(G) = p^{b}$ where 
$1 \leq b \leq 5$, and $\max\cd(G) = p^{e}$ where $1 \leq i \leq 2$. 
Then 
\[ \sum_{k=1}^{b}a_{k}p^{k} \leq c(G) \leq \sum_{l=1}^{b+e}c_{l}p^{l}, 
\text{ where } a_{k}, c_{l}\in \mathbb{Z}_{\geq 0} \text{ with } 
a_{b}\neq 0, c_{b+e}\neq 0 \text{ and } 
\sum_{k=1}^{b}a_{k} = \sum_{l=1}^{b+e}c_{l} = m. \]
\end{lemma}
		
\begin{lemma} 
\textnormal{\cite[{Lemma 14}]{SAcyclic}} 
\label{remark:newrange}
Let $G$ be a non-abelian $p$-group. 
If $G$ is not a direct product of an abelian and a non-abelian subgroup, 
then $p^2 ~|~ c(G)$.
\end{lemma}

\begin{lemma} \label{prop:CpXCp}	
\textnormal{\cite[{Theorem 10}]{SAcyclic}} 
Let $G$ be a non-abelian $p$-group where $p$ is an odd prime, 
and let $G$ have exponent $p^r$ where $r\leq 2$.
Assume $\cd(G) = \{ 1, p, p^2 \}$, $d(Z(G)) = m$ where $m\geq 2$,
and $G$ is not a direct product of an abelian and
a non-abelian subgroup. If $G$ has an 
abelian normal subgroup of index $p^2$, 
then $c(G) = ap^2 + bp^3$, for some $a,b \in \mathbb{Z}_{\geq 0}$
 where $a+b = m$.
\end{lemma}
\begin{remark} \label{remark:CpXCp}
\textnormal{Let $G$ be a non-abelian group of order 
$p^6$ where $p$ is an odd prime, $\exp(G) \leq p^3$, 
$\cd(G) = \{ 1, p \}$, $d(Z(G)) = 2$, and $G$ is not a 
direct product of an abelian and a non-abelian subgroup. 
Suppose $G$ has an abelian normal subgroup $A$ of index $p$ 
and $\exp(A) = p^r$ where $r\leq 2$. Proceeding 
as in the proof of Lemma \ref{prop:CpXCp}, 
we deduce that  $c(G) = 2p^2$, $p^3+p^2$, or $2p^3$. }
\end{remark}

\begin{lemma}  \textnormal{\cite[{Corollary 6}]{SAcyclic}} 
	\label{Z(G)isCyclic}
Let $G$ be a non-abelian $p$-group
with cyclic center $Z(G)$. Suppose $p^{e}
= \max \cd(G)$ and $p^{\alpha} = \min\{ 
\chi(1) ~|~ \chi \in \nl(G) \text{ and } \ker(\chi) = 1 \}$. Then
\[ p^{\alpha} \cdot |Z(G)| \leq c(G) \leq p^{e}\cdot \exp(G). \]
\end{lemma}
Recall that a finite group is {\it normally monomial} if all 
of its irreducible characters 
are induced from linear characters of normal subgroups. 

\begin{lemma} \textnormal{\cite[{Theorem 7}]{SAcyclic}} 
\label{lemma:normallymonomial}
Let $G$ be a normally monomial $p$-group 
 with cyclic center $Z(G)$.
If $G$ has a normal abelian subgroup $A$ of maximal
order among all abelian subgroups of $G$, then 
\[ \max \cd(G) \cdot |Z(G)| \leq c(G) \leq \max \cd(G) \cdot \exp(A). \]
\end{lemma}

\begin{lemma} \textnormal{\cite[{Lemma 21}]{SAcyclic}} \label{thm:cdG3}
Let $G$ be a non-abelian $p$-group 
and let $\cd(G) = \{ 1, p, p^a \}$ for some integer $a>1$. 
If $\chi$ is a faithful irreducible character of $G$, then $\chi(1) = p^a$.
\end{lemma}
			
\begin{lemma} \textnormal{\cite[{Corollary 23}]{SAcyclic}} 
\label{cor:normallymonomial}
Let $G$ be a metabelian $p$-group 
with cyclic center $Z(G)$. 
If $G$ has a 
normal abelian subgroup $A$ of maximal order among all abelian subgroups 
of $G$ and $A$ has exponent equal to $|Z(G)|$,
then $c(G) = \max \cd(G) \cdot |Z(G)|$.
\end{lemma}

\subsection{VZ groups}
A group $G$ is {\it VZ} if $\chi(g)=0$ for all 
$g\in G\setminus Z(G)$ and all $\chi \in \nl(G)$.
\begin{theorem}\textnormal{\cite[Theorem 26]{SA}}  \label{thm:VZpgroupp6}
Let $G$ be a VZ group of order $p^6$ where $p$ is an odd prime. 
\begin{enumerate}
\item If $|Z(G)| = p^4$, then $c(G) = p^4+p, p^3+p^2+p, p^3+p^2, p^3+2p, 3p^2, 2p^2+2p, 2p^2+p, 3p^2+p, p^2+3p, p^5, p^4+p^2, 2p^3$ or  $p^3+2p^2$.
\item  If $|Z(G)| = p^2$, then $c(G) = 2p^3, p^3+p, p^4$ or $p^3+p^2$.
\end{enumerate}
\end{theorem}

A $p$-group $G$ is VZ if and only 
if $\cd(G) = \{1, |G:Z(G)|^{1/2} \}$
(see \cite[Theorem A]{FAM}). 
Hence the VZ groups of order $p^6$ are in 
$\Phi_2, \Phi_{5}$ and $\Phi_{15}$ (see \cite{RJ}).
For completeness, we record in Table \ref{t:VZ}
the value of $\mu(G)$ for each such group $G$ where $p \geq 5$. 
The column labelled ``Group $G$" identifies 
$G$ by its identifier in \cite{Arxiv} and that 
labelled ``${\mathcal{H}}_G$" identifies a minimal 
degree faithful permutation representation of~$G$.

\begin{tiny}
	
\begin{longtable}[c]{|l|l|l|}
\caption{$\mu(G)$ 
for VZ groups 
of order $p^6$ where $p \geq 5$ \label{t:VZ}} \\
		
		\hline
		Group $G$ & $\mathcal{H}_{G}$ &  $\mu(G)$\\
		\hline
		\endfirsthead
		
		\hline
		\multicolumn{3}{|c|}{Continuation of Table \ref{t:VZ}}\\
		\hline
		Group $G$ & $\mathcal{H}_{G}$ &  $\mu(G)$\\
		\hline
		\endhead
		
		

		\vtop{\hbox{\strut $G_{(2,1)}$, $G_{(2,2)}$ } } &  \vtop{\hbox{\strut $\{ \langle \alpha_{3} \rangle \}$ }}  & $p^{5}$\\
		\hline
			\vtop{\hbox{\strut $G_{(2,3)}$, $G_{(2,5)}$ } } &  \vtop{\hbox{\strut $\{ \langle \alpha_{3}, \beta_{2} \rangle, \langle \alpha_{1}, \alpha_{2}, \alpha_{3}, \beta_{1} \rangle \}$ }}  & $p^{4}+p$\\
		\hline
			\vtop{\hbox{\strut $G_{(2,4)}$, $G_{(2,10)}$ } } &  \vtop{\hbox{\strut $\{ \langle \alpha_{2} \rangle, \langle \alpha_{3}, \beta_{1} \rangle \}$ }}  & $p^{4}+p^{2}$\\
		\hline
			\vtop{\hbox{\strut $G_{(2,6)}$ } } &  \vtop{\hbox{\strut $\{ \langle \alpha_{2} \rangle, \langle \alpha_{3} \rangle \}$ }}  & $p^{4}+p^{2}$\\
		\hline
		\vtop{\hbox{\strut $G_{(2,7)}$, $G_{(2,8)}$, $G_{(2,11)}$, $G_{(2,12)}$ } } &  \vtop{\hbox{\strut $\{ \langle \alpha_{2}, \alpha_{3}, \beta_{2} \rangle, \langle \alpha_{3}, \beta_{1} \rangle \}$ }}  & $p^{3}+p^{2}$\\
		\hline
			\vtop{\hbox{\strut $G_{(2,9)}$ } } &  \vtop{\hbox{\strut $\{ \langle \alpha_{2} \rangle, \langle \alpha_{3}, \beta_{2} \rangle \}$ }}  & $p^{4}+p^{2}$\\
		\hline
			\vtop{\hbox{\strut $G_{(2,13)}$, $G_{(2,14)}$ } } &  \vtop{\hbox{\strut $\{ \langle \alpha_{2},  \beta_{1} \rangle, \langle \alpha_{3}, \beta_{2} \rangle \}$ }}  & $2p^{3}$\\
		\hline
		\vtop{\hbox{\strut $G_{(2,15)}$, $G_{(2,16)}$ } } &   \vtop{\hbox{\strut $ \{ \langle \alpha_{2}, \alpha_{3}, \beta_{1}, \beta_{2} \rangle, \langle \alpha_{2}, \alpha_{3}, \beta_{1}, \beta_{3} \rangle, \langle \alpha_{3}, \beta_{2}, \beta_{3} \rangle \}$ } }  & $p^{3}+2p$\\
		\hline
		\vtop{\hbox{\strut $G_{(2,17)}$, $G_{(2,18)}$ } } &   \vtop{\hbox{\strut $ \{ \langle  \alpha_{2}, \alpha_{3}, \beta_{1}, \beta_{2} \rangle, \langle \alpha_{3}, \beta_{1}, \beta_{3} \rangle, \langle \alpha_{2}, \beta_{2}, \beta_{3} \rangle \}$ } }  & $p^{3}+p^2 +p$\\
		\hline
		\vtop{\hbox{\strut $G_{(2,19)}$ } } &   \vtop{\hbox{\strut $ \{ \langle  \alpha_{2}, \beta_{1}, \beta_{3} \rangle, \langle \alpha_{3}, \beta_{1}, \beta_{2} \rangle, \langle \alpha_{3}, \beta_{2}, \beta_{3} \rangle \}$ } }  & $p^{3}+2p^2$\\
		\hline
		\vtop{\hbox{\strut $G_{(2,20)}$, $G_{(2,22)}$ } } &   \vtop{\hbox{\strut $ \{ \langle  \alpha_{2}, \alpha_{3}, \beta_{1}, \beta_{2} \rangle, \langle \alpha_{2}, \alpha_{3}, \beta_{2}, \beta_{3} \rangle, \langle \alpha_{3}, \beta_{1}, \beta_{3} \rangle \}$ } }  & $2p^2 + p$\\
		\hline
			\vtop{\hbox{\strut $G_{(2,21)}$, $G_{(2,24)}$ } } &   \vtop{\hbox{\strut $ \{ \langle  \alpha_{2}, \alpha_{3}, \beta_{1}, \beta_{2} \rangle, \langle \alpha_{2}, \beta_{1}, \beta_{3} \rangle, \langle \alpha_{3}, \beta_{2}, \beta_{3} \rangle \}$ } }  & $p^{3}+p^2 +p$\\
		\hline
			\vtop{\hbox{\strut $G_{(2,23)}$ } } &   \vtop{\hbox{\strut $ \{ \langle  \alpha_{2}, \alpha_{3}, \beta_{2}, \beta_{3} \rangle, \langle \alpha_{3}, \beta_{1}, \beta_{2} \rangle, \langle \alpha_{2}, \beta_{1}, \beta_{3} \rangle \}$ } } & $3p^2$\\
		\hline
		\vtop{\hbox{\strut $G_{(2,25)}$ } } &   \vtop{\hbox{\strut $ \{ \langle  \alpha_{2}, \beta_{1}, \beta_{2} \rangle, \langle \alpha_{3}, \beta_{1}, \beta_{3} \rangle, \langle \alpha_{3}, \beta_{2} \rangle \}$ } }  & $p^{3}+2p^2$\\
		\hline
		\vtop{\hbox{\strut $G_{(2,26)}$ } } &   \vtop{\hbox{\strut $ \{ \langle  \alpha_{2}, \alpha_{3}, \beta_{2}, \beta_{3} \rangle, \langle \alpha_{2}, \beta_{1}, \beta_{2} \rangle, \langle \alpha_{3}, \beta_{1}, \beta_{3} \rangle \}$ } }  & $3p^2$\\
		\hline
		\vtop{\hbox{\strut $G_{(2,27)}$, $G_{(2,28)}$ } } &   \vtop{\hbox{\strut $ \{ \langle  \alpha_{2}, \alpha_{3}, \beta_{1}, \beta_{2}, \beta_{3} \rangle, \langle \alpha_{2}, \alpha_{3}, \beta_{1}, \beta_{2}, \beta_{4} \rangle,$ } \hbox{\strut $ \langle \alpha_{2}, \alpha_{3}, \beta_{1}, \beta_{3}, \beta_{4} \rangle, \langle  \alpha_{3}, \beta_{2}, \beta_{3}, \beta_{4} \rangle \}$ } }  & $p^2 + 3p$\\
		\hline
			\vtop{\hbox{\strut $G_{(2,29)}$, $G_{(2,30)}$ } } &   \vtop{\hbox{\strut $ \{ \langle  \alpha_{2}, \alpha_{3}, \beta_{1}, \beta_{2}, \beta_{3} \rangle, \langle \alpha_{2}, \alpha_{3}, \beta_{1}, \beta_{2}, \beta_{4} \rangle,$ } \hbox{\strut $ \langle \alpha_{3}, \beta_{1}, \beta_{3}, \beta_{4} \rangle, \langle  \alpha_{2}, \beta_{2}, \beta_{3}, \beta_{4} \rangle \}$ } }  & $2p^2 + 2p$\\
		\hline
			\vtop{\hbox{\strut $G_{(2,31)}$ } } &   \vtop{\hbox{\strut $ \{ \langle  \alpha_{2}, \alpha_{3}, \beta_{1}, \beta_{2}, \beta_{3} \rangle, \langle \alpha_{3}, \beta_{1}, \beta_{2}, \beta_{4} \rangle,$ } \hbox{\strut $ \langle \alpha_{2}, \beta_{1}, \beta_{3}, \beta_{4} \rangle, \langle  \alpha_{2}, \beta_{2}, \beta_{3}, \beta_{4} \rangle \}$ } }  & $3p^2 + p$\\
		\hline
			\vtop{\hbox{\strut $G_{(5,1)}$, $G_{(5,2)}$ } } &  \vtop{\hbox{\strut $\{ \langle \alpha_{2}, \alpha_{3} \rangle \}$ }}  & $p^{4}$\\
		\hline
		\vtop{\hbox{\strut $G_{(5,3)}$, $G_{(5,4)}$ } } &  \vtop{\hbox{\strut $\{ \langle \alpha_{2}, \alpha_{3}, \alpha_{4}, \alpha_{5}, \beta_{1} \rangle, \langle \alpha_{2}, \alpha_{3}, \beta_{2} \rangle \}$ }}  & $p^{3}+p$\\
		\hline
		\vtop{\hbox{\strut $G_{(5,5)}$ } } &  \vtop{\hbox{\strut $\{ \langle \alpha_{2}, \alpha_{3}, \beta_{1} \rangle, \langle \alpha_{2}, \alpha_{3}, \alpha_{5} \rangle \}$ }}  & $p^{3}+p^2$\\
		\hline
		\vtop{\hbox{\strut $G_{(5,6)}$ } } &  \vtop{\hbox{\strut $\{ \langle \alpha_{2}, \alpha_{4} \rangle, \langle \alpha_{2}, \alpha_{3}, \alpha_{5} \rangle \}$ }}  & $p^{3}+p^2$\\
		\hline
			\vtop{\hbox{\strut $G_{(5,7)}$ } } &  \vtop{\hbox{\strut $\{ \langle \alpha_{2}, \alpha_{3} \rangle, \langle \alpha_{3}, \alpha_{4}, \alpha_{5} \rangle \}$ }}  & $p^{3}+p^2$\\
		\hline
			\vtop{\hbox{\strut $G_{(15,1)}$, $G_{(15,2)}$, $G_{(15,3)}$, $G_{(15,4)}$ } } &  \vtop{\hbox{\strut $\{ \langle \alpha_{1}, \alpha_{5}, \alpha_{6} \rangle, \langle \alpha_{2}, \alpha_{5}, \alpha_{6} \rangle \}$ }}  & $2p^{3}$\\
		\hline
		\vtop{\hbox{\strut $G_{(15,5k)}$, for $k = 2,3,\ldots, p-1$ } } &  \vtop{\hbox{\strut $\{ \langle \alpha_{1}, \alpha_{5}, \alpha_{6} \rangle, \langle \alpha_{2}, \alpha_{5}, \alpha_{6} \rangle \}$ }}  & $2p^{3}$\\
		\hline
			\vtop{\hbox{\strut $G_{(15,6)}$ } } &  \vtop{\hbox{\strut $\{ \langle \alpha_{1}, \alpha_{3}, \alpha_{5} \rangle, \langle \alpha_{2}, \alpha_{3}, \alpha_{6} \rangle \}$ }}  & $2p^{3}$\\
		\hline
\end{longtable}
\end{tiny}

\section{Compute $\mu(G)$ for some groups of 
order $p^{6}$ and center of order dividing $p^{3}$} \label{Slessthanp4}
For the remainder of the paper, 
unless otherwise stated, $G$ is a finite non-abelian $p$-group where
$p \geq 5$ is prime. 
{We use the following strategy to compute $\mu(G)$. 
\begin{itemize}
	\item We produce a set 
	$\mathcal{H} := \{ H_1, \ldots, H_m \}$ of
	subgroups of $G$ which satisfy 
	$ \bigcap_{\substack{i=1}}^{m} \Core_G(H_{i})= 1.$ 
	Hence $\mu(G) \leq \sum_{i=1}^{m}|G:H_{i}|$.
	\item We exhibit $S \leq G$ which, by exploiting 
	the inequality $\mu(S)\leq \mu(G)$, implies
	that
	\[  {\mu(G)~\geq~\sum_{i=1}^{m}|G:H_{i}|.} \]
	
\end{itemize}
Hence we deduce that $\mu(G)=\sum_{i=1}^{m}|G:H_{i}|$.}

We partition the groups into two types. 
\begin{enumerate}
\item Type 1: a group that is a direct product of an 
abelian and a non-abelian subgroup; these are listed in Table \ref{t:2}.
\item Type 2: all others; these are listed in Table \ref{t:3}.
\end{enumerate}

The column labelled ``Group $G$" identifies the 
group $G$ of order $p^6$ by its identifier in \cite{Arxiv}.
In Table \ref{t:2}, the column labelled ``Group $\NAB$" lists 
a non-abelian direct factor $\NAB$ of $G$. 
The column labelled ``Argument for lower bound ($\NAB$)" 
gives (information about) a lower bound
for $\mu(\NAB)$; that labelled ``${\mathcal{H}}_\NAB$" identifies a 
minimal degree faithful permutation representation 
of $\NAB$; and that labelled ``$\mu(\NAB)$" lists the value of $\mu(\NAB)$. 
In Table \ref{t:3}, the column labelled ``Argument for lower bound" 
gives a lower bound for $\mu(G)$ and that 
labelled ``${\mathcal{H}}_G$" identifies a minimal degree faithful permutation 
representation of $G$. In both cases, the column labelled 
``$\mu(G)$" lists the value of $\mu(G)$.

If $G = \NAB \times \AB$ where $\NAB$ 
is non-abelian and $\AB$ is abelian, then $\AB$ is generated
 by those generators of $G$ which are omitted from the 
 list for $\NAB$ in Table \ref{t:2}.
If  $\{ \NAB_1, \NAB_2, \ldots, \NAB_n \}$ and 
 $\{ \AB_1, \AB_2, \ldots, \AB_m \}$ are minimal degree faithful permutation 
representations of $\NAB$ and $\AB$ respectively, then 
\[ \{  \NAB_1 \times \AB, \ldots, \NAB_n \times \AB,  \NAB \times \AB_1,\ldots, \NAB \times\AB_m \} \]
determines one for $G$ 
(see \cite[Proposition 2]{DLJ}). 

\begin{example} \label{example1}
{\rm
Consider the group of Type 1 listed in \cite{Arxiv}:
\begin{displaymath}
\begin{array}{cclll}
G_{(3,23)} & = & \langle 
\alpha_{1},\alpha_{2},\alpha_{3}, \alpha_{4}, \beta_{1}, \beta_{2}, \beta_{3} 
& | &  [\alpha_{3}, \alpha_{4}] = \alpha_{2}, [\alpha_{2}, \alpha_{4}] = 
\alpha_{1} = \beta_{1}, \alpha_{3}^{p} = \beta_{2}, 
\alpha_{4}^{p} = \beta_{1}, \\
& & & & \alpha_{2}^{p} =  
\beta_{1}^{p} = \beta_{2}^{p} = \beta_{3}^{p} = 1 \rangle \\
& = & \langle \alpha_{1}, \alpha_{2}, \alpha_{3}, \alpha_{4}, 
\beta_{1}, \beta_{2} \rangle  \times  \langle \beta_{3} \rangle & & \\
& = & \NAB \times C_p.
\end{array}
\end{displaymath}
(All trivial commutators between generators are omitted.)
Observe 
$Z(\NAB) = \langle \alpha_{4}^{p}, \alpha_{3}^p \rangle 
\cong C_{p} \times C_p$. Lemma \ref{L2} implies that we search 
for subgroups $\NAB_1$ and $\NAB_2$ of $\NAB$ whose cores in $\NAB$
intersect trivially.
Since $\NAB_1=\langle \alpha_{4}, \alpha_2 \rangle$ 
and $\NAB_2=\langle \alpha_{3}, \alpha_2\rangle$ have this property,
we deduce that $\mu(\NAB) \leq |\NAB:\NAB_1| + |\NAB:\NAB_{2}| \leq 2p^2$. 
On the other hand,  
$\NAB \geq S = \langle \alpha_{3}, \alpha_{2} \rangle \cong 
C_{p^{2}}\times C_{p}$, so $\mu(S) = p^2+p \leq \mu(\NAB)$. 
Therefore, $p^2 + p \leq \mu(\NAB) \leq 2p^2$.
We deduce from its structure that $\NAB$ is not 
a direct product of an abelian and a non-abelian subgroup, and so, 
by Lemmas \ref{thm:pgroup} and \ref{remark:newrange}, $\mu(\NAB) = 2p^2$.
Lemma \ref{thm:nilpotent} now implies that $ \mu(G)=2p^2 + p$. 
Since $\{ \langle \alpha_{4}, \alpha_{2} \rangle,
\langle \alpha_{3}, \alpha_{2} \rangle \}$ is a minimal degree
faithful permutation representation of $\NAB$, the set $\{ \langle 
\alpha_{4}, \alpha_{2}, \beta_{3} \rangle, \langle \alpha_{3}, 
\alpha_{2}, \beta_{3} \rangle, \NAB \}$ 
determines one for $G_{(3,23)}$.
}
\end{example}

\begin{example} \label{example2}
\rm {
Consider the group of Type 2 listed in \cite{Arxiv}:
\begin{equation*}
\begin{split}
G_{(3, 3)} = &	\; \langle \alpha_{1}, \alpha_{2}, \alpha_{3}, \alpha_{4}, \beta_{1} ~|~ [\alpha_{3}, \alpha_{4}] = \alpha_{2}, [\alpha_{2}, \alpha_{4}] = \alpha_{1} = \beta_{1}^{p^2}, \alpha_{4}^{p} = \beta_{1}, \beta_{1}^{p^3} = \alpha_{2}^{p} =  \alpha_{3}^{p} = 1 \rangle.
\end{split}
\end{equation*}
Observe $Z(G)=\langle \alpha_{4}^{p} \rangle \cong C_{p^{3}}$, so
Lemma \ref{L2} implies that $|{\mathcal{H}}_G| = 1$.
Let $S = \langle \alpha_{3}, \alpha_{2} \rangle$.
Since $\Core_{G}(S) = 1$ we deduce that $\mu(G)\leq p^{4}$. }
On the other hand, $G \geq \langle \alpha_{4} \rangle \cong C_{p^{4}} $ 
so $\mu(G)\geq p^{4}$. Therefore, $\mu(G)= p^{4}$.
\end{example}

\begin{remark}
While the presentations recorded in \cite{Arxiv} 
apply generally to $p \geq 7$, they also hold for $p = 5$ except for 
some groups in $\Phi_{35}, \Phi_{36}, \Phi_{38}$ and $\Phi_{39}$.
For convenience we record the few changes
needed to obtain valid presentations for the groups 
of order $5^6$ in these families.

\begin{itemize}
\item 
$G_{(35, 3)}$: If $p = 5$ then $\alpha^5 = \alpha^{-1}$.

\item 
$G_{(36,3r)}$:
If $p=5$, then $\alpha_{5}^{5} = \alpha_{1}^{-1}$  and 
$\alpha_{6}^{5} = \alpha_{1}^r$ for $r=1,2$.

\item 
$G_{(39,4r)}$:
If $p=5$, then $\alpha_{5}^{5} = \alpha_{1}$ and 
$\alpha_{6}^{5} = \alpha_{1}^r$ for $r=1,2,3,4$.

\item 
Here we list presentations for the groups of order $5^6$ in $\Phi_{38}$. 

\begin{scriptsize}
\begin{longtable}[c]{|l|l|} 

\hline
Group & Presentation \\
\hline
\endfirsthead

\hline
\multicolumn{2}{|c|}{Continuation of Table}\\
\hline
Group & Presentation \\
\hline
\endhead

\hline
\endfoot

\hline
\endlastfoot

\hline
$G_{(38,1r)}$   &  \vtop{\hbox{\strut $ \langle \alpha_{1}, \alpha_{2}, \alpha_{3}, \alpha_{4}, \alpha_{5}, \alpha_{6} ~|~ [\alpha_{5}, \alpha_{6}] = \alpha_{4}, [\alpha_{4}, \alpha_{6}] = \alpha_{3}, [\alpha_{3}, \alpha_{6}] = [\alpha_{4}, \alpha_{5}] = \alpha_{2}, [\alpha_{2}, \alpha_{6}] = [\alpha_{3}, \alpha_{5}] =  \alpha_{1},  $} 
\hbox{\strut $\alpha_{1}^{5} = \alpha_{2}^{5} = \alpha_{3}^{5} =  \alpha_{4}^{5} = 1, \alpha_{5}^{5}= \alpha_1^{-r}, 
\alpha_{6}^{5} = \alpha_{1}^{-2(r-1)} \rangle \text{ for } r=0,1,2,3,4 $ } } \\
\hline
$G_{(38,2r)}$ &   \vtop{\hbox{\strut $ \langle \alpha_{1}, \alpha_{2}, \alpha_{3}, \alpha_{4}, \alpha_{5}, \alpha_{6} ~|~ [\alpha_{5}, \alpha_{6}] = \alpha_{4}, [\alpha_{4}, \alpha_{6}] = \alpha_{3}, [\alpha_{3}, \alpha_{6}] = [\alpha_{4}, \alpha_{5}] = \alpha_{2}, [\alpha_{2}, \alpha_{6}] = [\alpha_{3}, \alpha_{5}] =  \alpha_{1}, \alpha_{6}^{5} = 1, $} 
\hbox{\strut $\alpha_{1}^{5} = \alpha_{2}^{5} = \alpha_{3}^{5} =  \alpha_{4}^{5} = 1, \alpha_{5}^{5}=  \alpha_1^{-r} \rangle$ for $r= 0,2,3,4$
} } \\
\hline
$G_{(38,3)}$ &    \vtop{\hbox{\strut $ \langle \alpha_{1}, \alpha_{2}, \alpha_{3}, \alpha_{4}, \alpha_{5}, \alpha_{6} ~|~ [\alpha_{5}, \alpha_{6}] = \alpha_{4}, [\alpha_{4}, \alpha_{6}] = \alpha_{3}, [\alpha_{3}, \alpha_{6}] = [\alpha_{4}, \alpha_{5}] = \alpha_{2}, [\alpha_{2}, \alpha_{6}] = [\alpha_{3}, \alpha_{5}] =  \alpha_{1}, \alpha_{5}^{5} = \alpha_{1}^{-1}, $} 
\hbox{\strut $\alpha_{1}^{5} = \alpha_{2}^{5} = \alpha_{3}^{5} = \alpha_{4}^{5} = 1, \alpha_{6}^{5}=  \alpha_1 \rangle$ } } \\
\end{longtable}
\end{scriptsize}

\end{itemize}
\end{remark}

\setcounter{table}{1}
\noindent	
	
	\begin{tiny}
		
	\begin{longtable}[c]{| l| l| l|l|l |l|}
	\caption{$\mu(G)$ for groups of Type 1 \label{t:2}}\\
	
	\hline
	Group $G$ & Group $\NAB$ & Argument for lower bound ($\NAB$) 
& ${\mathcal{H}}_\NAB$  &$\mu(\NAB)$  
& $\mu(G)$\\ 
	\hline
	\endfirsthead
	
	\hline
	\multicolumn{6}{|c|}{Continuation of Table \ref{t:2}}\\
	\hline
Group $G$ & Group $\NAB$ & Argument for lower bound ($\NAB$) & ${\mathcal{H}}_\NAB$  &$\mu(\NAB)$  & $\mu(G)$\\  
	\hline
	\endhead
	
	\hline
	\endfoot
	
	\hline
	\endlastfoot
	
\vtop{\hbox{\strut $G_{(3,4)}$ }  } &\vtop{\hbox{\strut $\langle \alpha_{1}, \alpha_{2}, \alpha_{3}, \alpha_{4}, \beta_{1} \rangle$ } }  & $\mu(\langle \beta_{1}
, \alpha_{2}\rangle)=p^{2}+p $ & $ \{ \langle \alpha_{3}, \alpha_{2} \rangle \} $  & $p^{3}$ & $p^{3}+p$ \\
\hline

\vtop{\hbox{\strut $G_{(3,5r)}$, for $r= 1$ or $\nu$ } } & \vtop{\hbox{\strut $\langle \alpha_{1}, \alpha_{2}, \alpha_{3}, \alpha_{4}, \beta_{1} \rangle$ } }  & $\mu(\langle \alpha_{3}, \alpha_{2}\rangle)=p^{3}+p $ & $ \{ \langle \alpha_{2} \rangle \} $  & $p^{4}$ & $p^{4}+p$ \\
\hline
\vtop{\hbox{\strut $G_{(3,7)}$ }  } &\vtop{\hbox{\strut $\langle \alpha_{1}, \alpha_{2}, \alpha_{3}, \alpha_{4}, \beta_{1} \rangle$ } }  & $\mu(\langle \alpha_{4}\rangle)=p^{3} $ & $ \{ \langle \alpha_{3}, \alpha_{2} \rangle \} $  & $p^{3}$ & $p^{3}+p$ \\
\hline
\vtop{\hbox{\strut $G_{(3,11)}$ }  } &\vtop{\hbox{\strut $\langle \alpha_{1}, \alpha_{2}, \alpha_{3}, \alpha_{4}, \beta_{2} \rangle$ } }  & \vtop{ \hbox{\strut $\mu(\NAB) = p^2$ since } \hbox{\strut  $\NAB \cong G_{29} $  \cite[p. 26]{S} }  }  & $ \{ \langle \alpha_{3}, \alpha_{2} \rangle \} $  & $p^{2}$ & $2p^{2}$ \\
\hline

\vtop{\hbox{\strut $G_{(3,13r)}$, for $r= 1$ or $\nu$ } } & \vtop{\hbox{\strut $\langle \alpha_{1}, \alpha_{2}, \alpha_{3}, \alpha_{4}, \beta_{2} \rangle$ } }  & $\mu(\langle \alpha_{3}, \alpha_{2}\rangle)=p^{2}+p $ & $ \{ \langle  \alpha_{2} \rangle \} $  & $p^{3}$ & $p^{3}+p^{2} $ \\
\hline
\vtop{\hbox{\strut $G_{(3,15)}$ }  } & \vtop{\hbox{\strut $\langle \alpha_{1}, \alpha_{2}, \alpha_{3}, \alpha_{4}, \beta_{2} \rangle$ } }  & \vtop{ \hbox{\strut $\mu(\NAB) = p^2$ since } \hbox{\strut  $\NAB \cong G_{24} $  \cite[p. 26]{S} }  } & $ \{ \langle \alpha_{3}, \alpha_{2} \rangle \} $  & $p^{2}$ & $2p^{2}$ \\
\end{longtable}
\end{tiny}
	 
\begin{tiny}
	
	\begin{longtable}[c]{|l|l|l|l|}
		\caption{$\mu(G)$ for groups of Type 2 \label{t:3}}\\
		
		\hline
Group $G$ & Argument for lower bound & ${\mathcal{H}}_G$ & $\mu(G)$\\
		\hline
		\endfirsthead
		
		\hline
		\multicolumn{4}{|c|}{Continuation of Table \ref{t:3}}\\
		\hline
Group $G$ & Argument for lower bound & ${\mathcal{H}}_G$ & $\mu(G)$\\
		\hline
		\endhead
		
		\hline
		\endfoot
		
		\hline
		\endlastfoot
		
		\vtop{\hbox{\strut $G_{(3,1)}$ } }  & $\mu(\langle  \alpha_{3}, \beta_{1} \rangle )=p^{3}+p $ & $ \{ \langle \alpha_{3}, \alpha_{2} \rangle \} $  &  $p^{4}$ \\
	\hline
	\vtop{\hbox{\strut $G_{(3,2r)}$, for $r= 1$ or $\nu$ } }  & $\mu(\langle \alpha_{3}, \alpha_{2} \rangle)=p^{4}+p $ & $ \{ \langle \alpha_{2} \rangle \} $  &  $p^{5}$ \\
	\hline
	\vtop{\hbox{\strut $G_{(3,3)}$ } }  & $\mu(\langle \alpha_{4} \rangle)=p^{4} $ & $ \{ \langle \alpha_{3}, \alpha_{2} \rangle \} $  &  $p^{4}$ \\
	\hline
	\vtop{\hbox{\strut $G_{(3,6)}$ } }  & $\mu(\langle \alpha_{3},  \alpha_{2}, \beta_{1} \rangle )=2p^{2}+p $ & $ \{ \langle \alpha_{3}, \alpha_{2} \rangle, \langle \alpha_{4}, \alpha_{2}, \beta_{1} \rangle \} $  &  $p^{3}+p^{2}$ \\
	\hline
	\vtop{\hbox{\strut $G_{(3,8)}$ } }  & \vtop{\hbox{\strut  $\mu(\langle \alpha_{4},  \alpha_{2}, \beta_{1} \rangle )=p^{3}+p^{2} $ since } \hbox{\strut $\langle \alpha_{4},  \alpha_{2}, \beta_{1} \rangle \cong G_{4} $ \cite[Theorem 3.5]{BD}}  } & $ \{ \langle \alpha_{4}^{p}, \alpha_{3}, \alpha_{2} \rangle, \langle \alpha_{3}, \alpha_{2}, \beta_{1} \rangle \} $  &  $p^{3}+p^{2}$ \\
	\hline
	\vtop{\hbox{\strut $G_{(3,9)}$ } }  & $\mu(\langle \alpha_{4}, \alpha_{3}^{p}, \alpha_{2} \rangle )=p^{3}+p $ & $ \{ \langle \alpha_{4}, \alpha_{2} \rangle, \langle \alpha_{3}, \alpha_{2} \rangle \} $  &  $p^{3}+p^{2}$ \\
	\hline
	\vtop{\hbox{\strut $G_{(3,12)}$ } }  & $\mu(\langle \alpha_{3},  \alpha_{2} \rangle )=p^{3}+p $ & $ \{ \langle \alpha_{3}, \alpha_{2} \rangle, \langle \alpha_{4}, \alpha_{2}, \alpha_{1} \rangle \} $  &  $p^{3}+p^{2}$ \\
	\hline
	\vtop{\hbox{\strut $G_{(3,14)}$ } }  & $\mu(\langle \alpha_{4},  \alpha_{1} \rangle )=p^{3}+p $ & $ \{ \langle \alpha_{4}^{p}, \alpha_{3}, \alpha_{2} \rangle, \langle \alpha_{3}, \alpha_{2}, \alpha_{1} \rangle \} $  &  $p^{3}+p^{2}$ \\
\end{longtable}
	
\end{tiny}

\section{Compute $c(G)$ for some groups of order 
$p^6$ and center of order dividing $p^3$} \label{S1}

We now compute $c(G)$ for some groups $G$  
of order $p^6$ with center of order dividing $p^3$. 
In Sections \ref{SS3}--\ref{SS5}
we first determine the values that $c(G)$ can take 
for each of the relevant isoclinism families; 
the approach is dictated by the order of their centers.
In Section \ref{subsection:tables} 
we exploit these values to list $X_{G}$, and so compute $c(G)$, 
for each group.  We present our results in Tables \ref{t:4}--\ref{t:7}. 

We first present results which assist our computations.
Let $A$ be a fixed maximal abelian subgroup of $G$ containing $G'$. 
Consider the set $T$ of all subgroups $D$ of $G$ with 
$D\leq A$ and $A/D$ cyclic.  Let $T_{G}$ be a set of representatives 
of the equivalence classes of $T$ under conjugacy in $G$. 
For $D \in T$, let $K_{D}$ be a 
fixed element of maximal order in $\{ \MTG \leq G ~|~ A\leq \MTG \text{ and } \MTG' \leq D \}$.
Clearly, if $D$ is core-free, then $K_{D} = A$. 
Define $R(D)$ to be the set of those complex linear representations 
of $K_{D}$ whose restriction to $A$ has kernel $D$. 
Let $R_{C}(D)$ be a complete set of those representations 
in $R(D)$ which are not mutually $G$-conjugate.

\begin{lemma} \textnormal{\cite[Theorem 2]{BKP}} \label{lemma:bakshi}
Let $G$ be a finite metabelian group with 
$A$ and $T_{G}$ as defined above. Then 
\[ \rho\ind_{K_{D}}^{G} {\rm for\ } \rho\in R_{C}(D){\rm\ and\ } D\in T_{G} \]
are precisely the inequivalent irreducible complex representations 
of $G$. 
Each of these representations 
is faithful if and only if $D$ is core-free.
\end{lemma}
	
\begin{lemma}\label{lemma:abelian}
Let $A \cong C_{p^{r_{1}}} \times C_{p^{r_{2}}} \times 
\cdots \times C_{p^{r_{t}}}$ 
where $p$ is a prime and 
$r_{1} > r_{2} \geq r_{3} \geq \cdots \geq r_{t}$. 
Suppose $\lambda \in \lin(A)$ satisfies 
$A/\ker(\lambda) \cong C_{p^{r_{1}}}$. If $a\in A$ 
has order $p^{r_{1}}$, 
then $\langle a \rangle \cap \ker(\lambda) = 1$.
\end{lemma}
\begin{proof}
Suppose, to the contrary, that $a^{p^{r_{1}-1}} \in \ker(\lambda)$. 
We choose $a_{2}, \ldots, a_{t} \in A$ such that 
$A = \langle a, a_{2}, \ldots, a_{t} \rangle 
\cong C_{p^{r_{1}}} \times C_{p^{r_{2}}} \times \cdots \times C_{p^{r_{t}}}$. 
Then $A/\ker(\lambda) = \langle a \ker(\lambda), a_{2} \ker(\lambda), 
\ldots, a_{t} \ker(\lambda) \rangle$. 
Since $a^{p^{r_{1}-1}} \in \ker(\lambda)$, 
$o(a \ker(\lambda)) \leq p^{r_{1}-1}$. 
Moreover, $o(a_{j} \ker(\lambda)) \leq p^{r_{2}} \leq p^{r_{1}-1}$
for $2\leq j \leq t$. Hence 
$\exp(A/ \ker(\lambda)) < p^{r_{1}}$, a contradiction. 
Therefore, $\langle a \rangle \cap \ker(\lambda) = 1$.
\end{proof}

\begin{remark} \label{remark:abeliangroups}
{\rm Let $G$ be a metabelian group of order 
$p^n$ where $p$ is a prime and $d(Z(G)) = 2$. 
Suppose $G$ has a maximal abelian subgroup 
$A \cong C_{p^{r_{1}}} \times C_{p^{r_{2}}}$ (where $r_{1} > r_{2}$) 
containing $G'$. 
Suppose $X_{G} = \{ \chi_{1}, \chi_{2} \} \subset \Irr(G)$. 
Since $G$ is metabelian, from Lemma \ref{lemma:bakshi}, 
$\chi_{i} = \rho_{i}\ind_{A}^{G}$ for some $\rho_{i} \in \lin(A)$ 
and for $i=1,2$. Suppose $A / \ker(\rho_{i}) \cong C_{p^{r_{2}}}$
for some $i$; without loss of generality, assume  $i = 1$.  
Now $|\ker(\rho_{1})| = p^{r_{1}}$. 
On the other hand, $\ker(\rho_{1}) \leq A$, 
so $d(\ker(\rho_{1})) \leq 2$. 
If $d(\ker(\rho_{1})) = 2$, then 
$\ker(\rho_{1}) \cong C_{p^{m_{1}}} \times C_{p^{m_{2}}}$
for some $m_{1}, m_{2} \geq 1$. Hence 
$\ker(\rho_{1})$ 
contains all the elements of order $p$ in $A$.
But $A$ is a maximal abelian subgroup of $G$, so $Z(G) \subseteq A$. 
Since $d(Z(G)) = 2$, $Z(G)$ contains all the elements 
of order $p$ in $A$. 
Hence $\ker(\rho_{1}\ind_{A}^{G}) \cap \ker(\chi_{2}) \neq 1$, 
a contradiction. Therefore, $\ker(\rho_{1}) \cong C_{p^{r_{1}}}$.
We claim that if $G$ has an element $b$ of order $p^{r_{2}}$, 
then $\langle b \rangle \cap \ker(\rho_{i}) = 1$.
Suppose, to the contrary, that $b^{p^{r_{2}-1}} \in \ker(\rho_{1})$. 
Choose $x\in A$ such that $o(x) = p^{r_{1}}$ and 
$A = \langle x \rangle \times \langle b \rangle 
\cong C_{p^{r_{1}}} \times C_{p^{r_{2}}}$. 
Now $A/\ker(\rho_{1}) = \langle x \ker(\rho_{1}), 
b \ker(\rho_{2}) \rangle \cong C_{p^{r_{2}}}$. 
Thus $o(x \ker(\rho_{1})) \leq p^{r_{2}}$. 
But $r_{1} > r_{2}$, so $x^{p^{r_{1}-1}} \in \ker(\rho_{1})$. 
Thus $\langle x^{p^{r_{1}-1}}, b^{p^{r_{2}-1}} 
\rangle \cong C_{p} \times C_{p} \subseteq \ker(\rho_{1})$. 
So $\ker(\rho_{1})$ contains all the elements of 
order $p$ in $A$. Hence 
$\ker(\rho_{1}\ind_{A}^{G}) \cap \ker(\chi_{2}) \neq 1$, 
a contradiction. This establishes the claim. 
}
\end{remark}

We now revert to our convention that $p \geq 5$.
\begin{lemma} \label{LemmaNotInX_G}
Let $G$ be a non-abelian group satisfying one of the following:
\begin{enumerate}
\item[(i).] $|G|=p^{n}$ and $d(Z(G)) \leq 2$.
\item[(ii).] $|G| = p^{6}$, 
$G$ is not a direct product of an abelian and a non-abelian subgroup,
$Z(G)\cong C_{p}^{3}$, and $G' \subset Z(G)$ has order $p^2$. 
\end{enumerate}
Suppose $X \subset \Irr(G)$ has cardinality $|X_{G}|$. 
If $\psi\in X$ has degree one
and $Z(G) \subseteq \ker(\psi)$, 
then $X$ does not satisfy Equation \eqref{eq:X_G}. 
\end{lemma}

\begin{proof}
\mbox{ }
\begin{enumerate}
\item[(i).]
If $d(Z(G)) = 1$, then $X\cap \lin(G) = \emptyset$, and the result follows. 
Let $d(Z(G)) = 2$. Lemma \ref{L1} implies that $|X_G|=2$. 
Let $X = \{ \chi, \psi \}\subset \Irr(G)$ be such that $\psi(1) = 1$ 
and $Z(G) \subseteq \ker(\psi)$. Since $G$ is a $p$-group, 
$\ker(\psi) \cap \ker(\chi)\neq 1$, so 
$X$ does not satisfy Equation \eqref{eq:X_G}. 

\item[(ii).]
Since $|G/Z(G)|=p^3$, $\cd(G)= \{ 1, p \}$. 
We deduce from \cite{RJ, Arxiv} that $G \in \Phi_{4}$, 
so $\exp(G) = p^2$ and $\exp(G/Z(G)) = p$. 
Lemma \ref{L1} implies that $|X_G|=3$. 
Let 
$X = \{ \psi_{i} \, : \, 1 \leq i \leq 3\} \subset \Irr(G)$
satisfy Equation \eqref{eq:X_G}, where $\psi_{1}\in \lin(G)$ is such 
that $Z(G) \subseteq \ker(\psi_{1})$. 
Then 
\begin{equation} \label{eq:3}
\bigcap_{i=1}^{3}\ker(\psi_{i}) = 1 \Rightarrow 
\left(\bigcap_{i=1}^{3}\ker(\psi_{i})\right)\cap Z(G) = 1 
\Rightarrow Z(G) \cap 
\left( \bigcap_{i=2}^{3}\ker(\psi_{i})\cap Z(G) \right) = 1.
\end{equation}

\noindent 
\medskip
We claim that $Z(G) \nsubseteq \ker(\psi_{i})$ for $i=2,3$.
Clearly, $Z(G) \nsubseteq \bigcap_{i=1}^{3}\ker(\psi_{i})$. 
If $Z(G)$ is contained in the kernel of 
say $\psi_{2}$, then Equation \eqref{eq:3} implies that 
$\ker(\psi_{3})\cap Z(G)=1$, a contradiction since 
$\ker(\psi_{3})$ is a normal subgroup of $G$. 

\noindent 
\medskip
We now show that $|\ker(\psi_{i})\cap Z(G)|=p^2$ for $i=2,3$.
If either $\psi_{i}$ is linear, 
then $G'$, which has order $p^2$, is contained in $\ker(\psi_{i})$, and 
thus our claim holds for $\psi_{i}$. 
Suppose $\psi_{2}, \psi_{3}\in \nl(G)$. 
Then $Z(G/\ker(\psi_{2}))$ is cyclic.
This implies that $\theta(Z(G))$ is cyclic,
where $\theta: G \mapsto G/\ker(\psi_{2})$ is the natural homomorphism.
Since $Z(G)$ is elementary abelian, $I := Z(G) / (Z(G) 
   	\cap \ker(\psi_{2}))$ is cyclic, so it has order 1 or $p$.
If $|I| = 1$, then $G' \subset Z(G) \subseteq \ker(\psi_{2})$; since 
      $\psi_{2} \in \nl(G)$, this is a contradiction.
       Hence $|I|=p$
        which implies that $|\ker(\psi_{2}) \cap Z(G)| = p^{2}$. 
        A similar argument applies to $\psi_{3}$.  
\end{enumerate}

Hence 
$Z(G) \cap \left( \bigcap_{i=2}^{3}\ker(\psi_{i})\cap Z(G) \right) \neq 1$, 
which contradicts Equation \eqref{eq:3}. 
\end{proof}

\begin{remark} \label{remark:NotInX_G}
\textnormal{In Lemma \ref{LemmaNotInX_G} (i), if $d(Z(G)) = 2$ 
and $\psi\in \lin(G)$ is such that  $Z(G) \nsubseteq \ker(\psi)$, 
then $\psi$ may belong to $X_{G}$. 
For example, let $G = G_{(7,14)}  \in  
\Phi_{7}$.
Consider 
$\psi =\psi_{\langle \alpha_{5} G' \rangle} \cdot 1_{\langle \alpha_{4}G' 
\rangle} \cdot 1_{\langle \alpha_{3} G' \rangle } $
where $\psi_{\langle \alpha_{5} G' \rangle}$ is a faithful linear 
character of $\langle \alpha_{5} G' \rangle$: 
observe that $\psi\in\lin(G/G') \cap X_{G}$ (see Table \ref{t:6}).}
\end{remark}
		
\begin{remark}\label{remark:linearcharacter}
{\rm If $G \in \Phi_i$ for $11 \leq i \leq 43$,
then $d(Z(G) \cap G') = d(Z(G))$ (see \cite{Arxiv}).
 Lemma \ref{lemmalinearcharacter} implies that 
  $X_{G} \cap \lin(G) = \emptyset$.}
\end{remark}

\subsection{Groups with center of order $p^{3}$} \label{SS3}
We now consider those groups of order $p^6$ 
which have center of order $p^3$. 
These groups are in $\Phi_{i}$
where $i\in \{ 3,4,6,11 \}$. From \cite{RJ}, we deduce that
their set of character degrees is $\{ 1, p \}$; 
all are metabelian, as their commutator subgroups are abelian. 
			
\begin{lemma}
\label{prop:centercyclicp3}
If $G$ is a group of order $p^6$ with cyclic
center of order $p^3$, then $c(G) = p^4$ or $p^5$.
\end{lemma}
\begin{proof}
As $\cd(G) = \{ 1, p \}$, the result follows from 
Lemma \ref{Z(G)isCyclic}. \end{proof}
			
\begin{lemma}
\label{prop:phi3}
Let $G$ be a group of order $p^6$ 
in $\Phi_{3}$ with non-cyclic center $Z(G)$.
\begin{enumerate}
\item[(A.)] If $Z(G) \cong C_{p^2} \times C_{p}$, then  
$c(G) = 2p^2$, $p^3+p$, $ p^3+p^2$, $2p^3$, $p^4+p$, or $p^4+p^2$.
\item[(B.)] If $Z(G) \cong C_{p}^3$, then  
$c(G) = p^2+2p$, $2p^2+p$, $3p^2$, $p^3+2p$, or $p^3+p^2+p$.
\end{enumerate}
\end{lemma}

\begin{proof}
Note that $\cd(G) = \{ 1, p \}$, $\exp(G/Z(G)) = p$, 
$G' \cong C_{p}\times C_{p}$, and $Z(G) \cap G' \cong C_{p}$
(see \cite{RJ, Arxiv}).

\begin{enumerate}
\item[(A.)]
If $G \in \Phi_{3}$ and $Z(G) \cong C_{p^2}\times C_{p}$, then 
$\exp(G) = p^{r}$ where $r\leq 3$ (see \cite{Arxiv}).
Lemma \ref{L1} implies that $|X_{G}|=2$. 

\medskip
\noindent 
{\bf Sub-case 1}: $G$ is a direct product of an abelian and a 
non-abelian subgroup.
\begin{itemize}
\item $G = H \times K$ where $H$ is non-abelian, $Z(H) \cong C_{p^2}$, 
and $K\cong C_{p}$. 
Then $c(G) = c(H)+p$. Lemma \ref{Z(G)isCyclic} implies that 
$c(H) = p^3$, or $p^4$. Hence $c(G) = p^3+p$, or $p^4+p$.

\item $G = H \times K$ where $H$ is 
non-abelian, $Z(H) \cong C_{p}$, 
and $K\cong C_{p^2}$. Then $c(G) = c(H)+p^2$. 
Lemma \ref{Z(G)isCyclic} implies that $c(H) = p^2$, or $p^3$. 
Hence $c(G) = 2p^2$, or $p^3+p^2$.
\end{itemize}
  
\medskip
\noindent 
{\bf Sub-case 2}: $G$ is not a direct product of an abelian and 
a non-abelian subgroup.

\noindent 
From \cite{Arxiv}, we 
deduce that if $G \in \Phi_{3}$, then $G$ has an abelian normal subgroup, 
say $A$, of index $p$. Since $\exp(G) = p^r$ where $r\leq 3$
and $A$ contains $Z(G)$, we deduce that
$A$ is isomorphic to $C_{p^3} \times C_{p^2}$,  
or $C_{p^3} \times C_{p} \times C_{p}$, 
or $C_{p^2} \times C_{p^2} \times C_{p}$, or $C_{p^2} \times C_{p}^3$. 
If $A \cong C_{p^3} \times C_{p^2}$, then 
$G' \subseteq Z(G)$, since $G' \cong C_{p}\times C_{p}$ 
and $Z(G) \subseteq A$. This is a contradiction, since 
$Z(G) \cap G' \cong C_{p}$ for 
all $G\in \Phi_{3}$. 
Let $X = \{ \chi_{1}, \chi_{2} \} \subset \Irr(G)$ 
satisfy Equation \eqref{eq:X_G}. 
From Lemma \ref{lemma:bakshi}, for $i = 1,2$, there 
exists $\rho_{i} \in \lin(K_{D_{i}})$, where $D_{i} \leq A$ 
with $A/D_{i}$  cyclic and $\ker(\rho_{i}\restr_{A}) = D_{i}$, such that 
\begin{equation} \label{eq:phi3}
\chi_{i} = \rho_{i}\ind_{K_{D_{i}}}^{G} \in \Irr(G)
\text{ and } \bigcap_{i=1}^{2} \ker\left(\rho_{i}\ind_{K_{D_{i}}}^{G}\right)  = 1.
\end{equation}
If $A$ is isomorphic to $C_{p^2} \times C_{p^2} \times C_{p}$ 
or $C_{p^2} \times C_{p}^3$, then, from Remark \ref{remark:CpXCp}, 
$c(G) = 2p^2$, $p^3+p^2$, or $2p^3$.
			
From \cite{Arxiv}, observe that 
every relevant group in $\Phi_{3} \setminus 
\{ G_{(3,10r)}, G_{(3, 12)}, 
G_{(3, 17)} \}$ has an abelian normal subgroup 
isomorphic to either $C_{p^2}\times C_{p^2} \times C_{p}$, 
or $C_{p^2}\times C_{p}^{3}$. 
From Table \ref{t:3}, {$c(G_{(3, 12)}) = c(G_{(3, 17)})  = p^3+p^2$.}
Consider 
\begin{align*}
G := G_{3,10r} = \langle \alpha_{1}, \alpha_{2}, \alpha_{3}, \alpha_{4}, 
\beta_{1}, \beta_{2} ~|~ & [\alpha_{3}, \alpha_{4}] = \alpha_{2}, 
[\alpha_{2}, \alpha_{4}] = 
\alpha_{1} = \beta_{1}^{p}, \alpha_{3}^{p} = \beta_{1}^{r},\\
& \alpha_{4}^{p} = \beta_{2}, \alpha_{2}^{p} = \beta_{1}^{p^2} = \beta_{2}^{p}  = 1 \rangle \quad (r =1 \text{ or } \nu).
\end{align*}
Here $Z(G) = \langle \alpha_{4}^{p}, \alpha_{3}^{p} \rangle \cong C_{p} 
\times C_{p^2}$, $G' = \langle \alpha_{3}^{p^2}, \alpha_{2} \rangle 
\cong C_{p} \times C_{p}$, and $G/G' = \langle \alpha_{4} G', 
\alpha_{3} G' \rangle \cong C_{p^2} \times C_{p^2}$. 
Observe that $A = \langle \alpha_{4}^{p}, \alpha_{3}, \alpha_{2} \rangle 
\cong C_{p} \times C_{p^3} \times C_{p}$ is a maximal abelian subgroup of 
$G$ containing $G'$. For $i=1,2$, let $\chi_{i}$ and $D_{i}$ be 
defined as in Equation \eqref{eq:phi3}. 
Here $K_{D_{i}} = G$ or $A$. Since both 
irreducible characters in $X$ cannot be linear, 
$K_{D_{1}}$ and $K_{D_{2}}$ cannot 
both be $G$. 

\medskip
\noindent 
{\bf Claim 1:} If $K_{D_{i}} = A$ for some $i \in \{1, 2\}$,
then $A/\ker(\rho_{i}) \ncong C_{p^2}$. 

\noindent 
Suppose, to the contrary, that 
$A/\ker(\rho_{1}) = \langle x \ker(\rho_{1}) \rangle \cong C_{p^2}$
for some $x\in A$. 
Then $$(\alpha_{3}\ker(\rho_{1}))^{p^2} = \ker(\rho_{1}) 
\Rightarrow \alpha_{3}^{p^2}\in \ker(\rho_{1}).$$ 
Since $A/\ker(\rho_{1})$ is cyclic, elementary 
computations show that $\alpha_{4}^{p}$ and 
$\alpha_{2}$ are in $\ker(\rho_{1})$. 
Thus $\langle \alpha_{4}^{p}, \alpha_{3}^{p^2}, 
\alpha_{2} \rangle \subset \ker(\rho_{1})$. 
Hence $Z(G) \cap \ker(\rho_{1}) \cong C_{p}\times C_{p}$. 
Since $Z(G) \cap \ker(\rho_{1}) \subseteq \ker(\rho_{1}\ind_{A}^{G})$, 
we deduce that $\ker(\rho_{1}\ind_{A}^{G}) \cap \ker(\chi) \neq 1$
for every $\chi\in \Irr(G)$. This is a contradiction since 
$\rho_{1}\ind_{K_{D_{1}}}^{G} \in X$. This establishes Claim 1. 

\medskip
\noindent 
Consider first the situation where $K_{D_{i}} = A$ for $i=1,2$. 
Claim 1 shows that $A/\ker(\rho_{i}) \cong C_{p}$, or $C_{p^3}$, 
for $i=1,2$, since $\exp(G) = p^3$. 

\medskip
\noindent 
{\bf Claim 2:}  $A/\ker(\rho_{i}) \cong C_{p}$
for at most one $i$ where $1\leq i \leq 2$.

\noindent 
Suppose, to the contrary, 
that $A/\ker(\rho_{i}) \cong C_{p}$ for $i=1,2$. 
Then $|\ker(\rho_{i})|=p^4$. 
But $\ker(\rho_{1}) \neq \ker(\rho_{2})$, so 
$\ker(\rho_{1}) \cdot \ker(\rho_{2}) = A$ and 
$|\ker(\rho_{1}) \cap \ker(\rho_{2})|=p^3$. 
Since $Z(G) \subset A$, 
$$\left| \left( \bigcap_{i=1}^{2}\ker(\rho_{i}) \right) 
\cap Z(G) \right| \geq p.$$ 
Thus 
$$\bigcap_{i=1}^{2} \left( \ker(\rho_{i}) \cap Z(G) \right) \neq 1 
\Rightarrow  \bigcap_{i=1}^{2} \left( \ker(\rho_{i}\ind_{A}^{G}) 
\cap Z(G) \right) \neq 1 
\Rightarrow \bigcap_{i=1}^{2} 
\left( \ker(\rho_{i}\ind_{A}^{G}) \right) \neq 1.$$ 
This contradicts Equation \eqref{eq:phi3}, and so establishes Claim 2.

\medskip
\noindent 
Without loss of generality, assume that 
$A/\ker(\rho_{1}) \cong C_{p^3}$, and 
$A/\ker(\rho_{2}) \cong C_{p}$, or $C_{p^3}$.
Lemma \ref{lemma:abelian} implies that 
$$\frac{A}{\ker(\rho_{1})} = \langle \alpha_{3} \ker(\rho_{1}) \rangle 
\cong C_{p^3}.$$ 
Thus $\rho_{1}(\alpha_{3}) = \omega$ where 
$\omega$ is a primitive $p^3$-th root of unity. 
The structure of $\ker(\rho_{1})$ shows that 
$\rho_{1}(\alpha_{2}) = \omega^{p^2 m}$ for some $m\in \mathbb{Z}$. 
Since $G/A = \langle \alpha_{4} A\rangle \cong C_{p}$, 
\[ \rho_{1}\ind_{A}^{G}(\alpha_{3}) = \sum_{k=0}^{p-1} \rho(\alpha_{4}^{-k} 
\alpha_{3} \alpha_{4}^{k}) 
= \sum_{k=0}^{p-1} \rho( \alpha_{3}^{1 + p^2k(k-1)/2} \alpha_{2}^{k}) 
= \omega \sum_{k=0}^{p-1} \omega^{p^2[mk + \frac{k(k-1)}{2}]} = c\omega, \]
where $0 \neq c \in \mathbb{Q}(\omega^{p^2}) \subset 
\mathbb{Q}(\rho_{1}\ind_{A}^{G})$. 
Thus $\mathbb{Q}(\rho_{1}) = \mathbb{Q}(\omega) 
\subset \mathbb{Q}(\rho_{1}\ind_{A}^{G})$. 
Therefore, $$d(\rho_{1}\ind_{A}^{G}) = p\phi(p^3) = p^3(p-1).$$ 
On the other hand, $d(\rho_{2}\ind_{A}^{G}) \geq p(p-1)$.
Hence 
$\sum_{i=1}^{2} d(\rho_{i}\ind_{K_{D_{i}}}^{G}) \geq p^3(p-1) + p(p-1)$.

\medskip
\noindent 
Now consider the situation where just one of $K_{D_{1}}$ and $K_{D_{2}}$ is $A$.
Without loss of generality, let $K_{D_{1}} = A$, and $K_{D_{2}} = G$. 
If $A/\ker(\rho_{1}) \cong C_{p}$, 
then $\alpha_{3}^{p} \in \ker(\rho_{1})$. 
But $\rho_{2} \in \lin(G)$, so 
$\alpha_{3}^{p^2} \in G' \subseteq \ker(\rho_{2})$. 
Therefore, $X$ does not satisfy Equation \eqref{eq:X_G}, 
a contradiction. 
Hence $A/\ker(\rho_{1}) \cong C_{p^3}$. 
As calculated above, $d(\rho_{1}\ind_{A}^{G}) = p^3(p-1)$. 
But $\chi_{2} = \rho_{2} \in X$, 
so $\frac{G/G'}{\ker(\rho_{2})/G'} \cong C_{p^2}$. 
Therefore, 
$\sum_{i=1}^{2} d(\rho_{i}\ind_{K_{D_{i}}}^{G}) \geq p^3(p-1) + p(p-1).$

\medskip
\noindent 
Lemma \ref{L1} implies that $c(G) \geq p^4+p^2$. 
But 
$\Core_{G}(\langle \alpha_{4} \rangle) \cap \Core_{G} (\langle \alpha_{3}, 
\alpha_{2} \rangle) = 1$, so $c(G) = \mu(G) \leq p^4+p^2$. 
Therefore, $c(G) = p^4+p^2$. 
			
\medskip
\item[(B.)]
If $G \in \Phi_{3}$ and $Z(G) \cong C_{p}^3$, then 
$\exp(G) = p^{r}$ where $r\leq 2$ (see \cite{Arxiv}).

\noindent 
{\bf Sub-case 1}: $G$ is a direct product of an abelian and 
non-abelian subgroup.

\begin{itemize}
\item 
$G = H \times K$ where $H$ is non-abelian with center of order $p$, 
and $K\cong C_{p}\times C_{p}$. Now $c(G) = c(H)+2p$. 
Lemma \ref{Z(G)isCyclic} implies that $c(H) = p^2$, or $p^3$. 
Hence $c(G) = p^2+2p$, or $p^3+2p$.

\item 
$G = H \times K$ where $H$ is non-abelian 
and not a direct product of an abelian 
and a non-abelian subgroup, 
$Z(H) \cong C_{p}\times C_{p}$, and $K\cong C_{p}$.
These groups are 
$G_{(3, 20)}$, $G_{(3, 22)}$, 
$G_{(3, 23)}$, and 
$G_{(3, 24r)}$ ($r= 1$ or $\nu$). 
From Table \ref{t:2}, {$c(G_{(3, 20)}) = c(G_{(3, 22)}) = c(G_{(3, 23)})  = 2p^2 +p$.}
Consider
\begin{align*}
{G_{(3, 24r)}} = \langle \alpha_{1}, \alpha_{2}, \alpha_{3}, \alpha_{4}, 
\beta_{1}, \beta_{2}, \beta_{3} ~|~ &[\alpha_{3}, 
\alpha_{4}] = \alpha_{2}, [\alpha_{2}, 
\alpha_{4}] = \alpha_{1} = \beta_{1}, 
\alpha_{3}^{p} = \beta_{1}^{r}, \\
& 
\alpha_{4}^{p} = \beta_{2},
\alpha_{2}^{p} =  \beta_{1}^{p} = \beta_{2}^{p} = \beta_{3}^{p} = 1 \rangle.
\end{align*}
Let $H=  \langle \alpha_{1}, \alpha_{2}, \alpha_{3}, 
\alpha_{4}, \beta_{1}, \beta_{2} \rangle$. Now 
$$\Core_{H} (\langle \alpha_{4} \rangle) \cap \Core_{H} (\langle 
\alpha_{3}, \alpha_{2} \rangle) = 1 \Rightarrow c(H) = \mu(H) \leq p^3+p^2.$$ 
Lemmas \ref{P2} and \ref{remark:newrange} imply that 
$c(H) = 2p^2$, or $p^3+p^2$. Therefore, ${c(G_{(3, 24r)})} = 2p^2+p$, or $p^3+p^2+p$.
\end{itemize}

\medskip
\noindent
{\bf Sub-case 2}: $G$ is not a direct product of an abelian and 
a non-abelian subgroup.

\noindent 
The only such group is $G_{(3, 25)}$. From Table \ref{t:3}, 
$\mu(G_{(3, 25)}) = c(G_{(3, 25)}) = 3p^2$.  \qedhere
\end{enumerate}
\end{proof}

\begin{remark} \label{remark:notadirectproduct}
\textnormal{
Let $G$ be a non-abelian $p$-group. If $d(Z(G)\cap G') = d(Z(G))$, 
then $G$ is not a direct product of an abelian and a non-abelian subgroup.
Suppose, to the contrary, that $G = H\times K$
where $H'\neq 1$ and $K'=1$. Then $Z(G) = Z(H) \times Z(K)$, 
and $G' = H'$. But $Z(G) \cap H' \subseteq Z(H) \cap H'$, so 
$d(Z(G) \cap G') = d(Z(G) \cap H') \leq d(Z(H) \cap H') 
\leq d(Z(H)) < d(Z(G))$, a contradiction.
}
\end{remark}

\begin{lemma}
\label{prop:phi4}
Let $G$ be a group of order $p^6$ in $\Phi_{4}$. 
\begin{enumerate}
\item[(A.)] If $Z(G) \cong C_{p^2} \times C_{p}$, then  
$c(G) = p^3+p^2$, $2p^3$, $p^4+p^2$, or $p^4+p^3$.
\item[(B.)] If $Z(G) \cong C_{p}^3$, then 
$c(G) = 2p^2+p$, $3p^2$, $p^3+p^2+p$, $2p^3+p$, $p^3+2p^2$, or $2p^3+p^2$.
\end{enumerate}
\end{lemma}
			
\begin{proof}
Note that $\cd(G) = \{ 1, p \}$, 
$\exp(G/Z(G)) = p$, $G' \cong C_{p}\times C_{p}$, and 
$G' \subset Z(G)$ (see \cite{RJ, Arxiv}).

\begin{enumerate}
\item[(A.)]
If $G \in \Phi_{4}$ and $Z(G) \cong C_{p^2}\times C_{p}$, then 
$\exp(G) = p^{r}$ where $r\leq 3$ (see \cite{Arxiv}).
Since $d(Z(G)\cap G') = d(Z(G))$, 
Remark \ref{remark:notadirectproduct} implies that 
$G$ is not a direct product of an abelian and a non-abelian subgroup. 

\medskip
\noindent 
{\bf Claim 1:} $c(G) \geq p^3+p^2$.

\noindent 
Lemma \ref{L1} implies that $|X_{G}| = 2$. 
Since $Z(G) \cap G'= G' \cong C_{p}\times C_{p}$, 
observe that $\lin(G) \cap X_{G} = \emptyset$. 
Let $X = \{ \chi_{1}, \chi_{2} \}\subset \nl(G)$
satisfy Equation \eqref{eq:X_G}. 
But $\chi_{i}(1) = p$ for $i = 1,2$, so 
$G$ has subgroups $H_i$ of index $p$
such that $\chi_{i} = \lambda_{i}\ind_{H_{i}}^{G}$ for some 
$\lambda_{i} \in \lin(H_{i})$, and 
$\mathbb{Q}(\chi_{i}) = \mathbb{Q}(\lambda_{i})$ (see Lemma \ref{thm:ford}).

We first prove that $|\ker(\lambda_{i})|= p^4$ for at 
most one $i \in \{1, 2\}$.
Suppose, to the contrary, that $|\ker(\lambda_{i})|=p^4$ for $i=1,2$. 
Since $Z(G) \subseteq H_{i}$ and $Z(G) \nsubseteq \ker(\lambda_{i})$, 
we deduce that $|Z(G) \cap \ker(\lambda_{i})| = p^2$. 
But $|Z(G)| = p^3$, so
$\bigcap_{i=1}^{2} Z(G) \cap \ker(\lambda_{i}) \neq 1$. 
Hence 
$\bigcap_{i=1}^{2} \ker(\lambda_{i}\ind_{H_{i}}^{G}) \neq 1$, 
a contradiction.

Without loss of generality, 
assume $|\ker(\lambda_{2})| \leq p^3$. 
Then $\frac{H_{1}/H_{1}'}{\ker(\lambda_{2})/ H_{1}'}$ 
is isomorphic to $C_{p^2}$, 
or $C_{p^3}$, so $d(\chi_{2}) \geq p \phi(p^2) = p^2(p-1)$. 
But $Z(G) \cong C_{p^2} \times C_{p}$, so 
$d(\chi_{1}) \geq p \phi(p) = p(p-1)$. 
Hence $$\sum_{i=1}^{2} d(\chi_{i}) \geq p(p-1) + p^2(p-1).$$ 
Lemma \ref{L1} implies that $c(G) \geq p^3+p^2$. This establishes Claim 1.

From \cite{Arxiv}, observe that if 
$G \in \Phi_{4}$, then it has an abelian normal subgroup 
$A$ of index $p$. 
But $Z(G) \subset A$, so $A$ must be isomorphic to 
$C_{p^3} \times C_{p^2}$, or $C_{p^3} \times C_{p} \times C_{p}$, 
or $C_{p^2} \times C_{p^2} \times C_{p}$, 
or $C_{p^2} \times C_{p}^3$. 
If $A\cong C_{p^2} \times C_{p^2} \times C_{p}$, 
or $C_{p^2} \times C_{p}^3$, then, from the above discussion, 
$c(G) \geq p^3+p^2$. Remark \ref{remark:CpXCp} now implies
$c(G) = p^3+p^2$, or $2p^3$.

Suppose $A  \cong C_{p^3} \times C_{p^2}$. 
From \cite{Arxiv}, 
\begin{align*}
G := G_{(4, 9r)} =	
\langle \alpha_{1}, \alpha_{2}, \alpha_{3}, 
\alpha_{4}, \alpha_{5}, \beta_{1}, 
\beta_{2} ~|~ & [\alpha_{4}, \alpha_{5}] 
= \alpha_{2} = \beta_{2}, [\alpha_{3}, \alpha_{5}] 
= \alpha_{1} = \beta_{1}^{p}, \alpha_{3}^{p} = \beta_{2}^{r}, \\ 
& \alpha_{4}^{p} = \beta_{1},
  \alpha_{5}^{p}= \beta_{1}^{p^2} = \beta_{2}^{p} =1 \rangle, 
\text{ for } r =1 \text{ or } \nu.
\end{align*}
Here $Z(G) = \langle \alpha_{3}^{p}, \alpha_{4}^{p} \rangle 
\cong C_{p} \times C_{p^2}$, 
and 
$G' = \langle \alpha_{3}^{p}, \alpha_{4}^{p^2} \rangle \cong C_{p} \times C_{p}$. 
Let $A = \langle \alpha_{3}, 
\alpha_{4} \rangle \cong C_{p^2} \times C_{p^3}$. 
From Lemma \ref{lemma:bakshi}, for $i = 1,2$, 
there exists $\rho_{i} \in \lin(K_{D_{i}})$, 
where $D_{i} \leq A$ with $A/D_{i}$  cyclic and 
$\ker(\rho_{i}\restr_{A}) = D_{i}$, such that 
\begin{equation} \label{eq:phi4}
\chi_{i} = \rho_{i}\ind_{K_{D_{i}}}^{G} \in \Irr(G),
\end{equation}
and $\bigcap_{i=1}^{2} \ker\left(\rho_{i}\ind_{K_{D_{i}}}^{G}\right)  = 1$.
Since $\chi_{i}(1) = p$, $K_{D_{i}} = A$. 
Suppose $A/\ker(\rho_{i})\cong C_{p}$ for some $i$. 
Hence $|\ker(\rho_{i})|=p^4$; since it is abelian,
$\ker(\rho_{i}) \cong C_{p^3} \times C_{p}$, 
or $C_{p^2} \times C_{p^2}$ and so 
contains 
all the elements of order $p$ in $Z(G)$. But 
$X$ does not satisfy Equation \eqref{eq:X_G}, 
a contradiction. Hence 
$A / \ker(\rho_{i})\ncong C_{p}$ for any $i$. 
But $\exp(A) = p^3$, so 
$A/\ker(\rho_{i})\cong C_{p^2}$, or $C_{p^3}$, for $i=1,2$. 

\medskip
\noindent
We claim that $A/\ker(\rho_{i})\cong C_{p^2}$ for at 
most one $i$ where $1\leq i \leq 2$.
Suppose, to the contrary, that 
$A/\ker(\rho_{i})\cong C_{p^2}$ for $i=1,2$. 
Then $|\ker(\rho_{i})| = p^3$ for $i=1,2$. 
Observe that $|\bigcap_{i=1}^{2} \ker(\rho_{i})| \geq p$. 
Since $Z(G)$ contains all the elements of order $p$ in $A$, 
$$Z(G) \bigcap \left( \bigcap_{i=1}^{2} \ker(\rho_{i}) \right) \neq 1,$$ 
a contradiction. This establishes the claim.

\medskip
\noindent
Without loss of generality, suppose that 
$A / \ker(\rho_{1})\cong C_{p^3}$.
Lemma \ref{lemma:abelian} implies that 
$$\frac{A}{\ker(\rho_{1})} = 
\langle \alpha_{4} \ker(\rho_{1}) \rangle \cong C_{p^3}$$ where 
$A = \langle \alpha_{3}, \alpha_{4} \rangle \cong C_{p^2} \times C_{p^3}$. 
The structure of $G$ shows 
$\rho_{1}\ind_{A}^{G}(\alpha_{3}^{r}\alpha_{4}^{s}) = c\omega^{s}$
for some $0 \neq c \in \mathbb{C}$, 
$r\in \mathbb{Z}$ and $s\in \mathbb{Z}_{p^3}^{\times}$, 
where $\omega$ is a primitive $p^3$-th root of unity. 
Hence $\mathbb{Q}(\rho_{1}\ind_{A}^{G}) = 
\mathbb{Q}(\omega)$, so 
$$d(\rho_{1}\ind_{A}^{G}) = p \phi(p^3) = p^3(p-1).$$

\medskip
\noindent
Now suppose $A/\ker(\rho_{2}) \cong C_{p^2}$. 
From Remark \ref{remark:abeliangroups}, 
$$\frac{A}{\ker(\rho_{2})} = 
\langle \alpha_{3} \ker(\rho_{2}) \rangle \cong C_{p^2}.$$ 
By a procedure similar to above, 
$d(\rho_{2}\ind_{A}^{G}) = p^2(p-1)$.
Therefore, $\sum_{i=1}^{2} d(\chi_{i}) \geq p^3(p-1) + p^2(p-1)$. 
Lemma \ref{L1} implies that $c(G) \geq p^4+p^3$. 
On the other hand, 
$\Core_{G} (\langle \alpha_{3} \rangle) \cap 
\Core_{G} (\langle \alpha_{4} \rangle) = 1$, 
so $c(G) = \mu(G) \leq p^4+p^3$. Therefore, $c(G) = p^4+p^3$.

From \cite{Arxiv}, 
observe that every relevant groups
 in $\Phi_{4} \setminus 
\{ G_{(4, 2)}, G_{(4, 4)}, G_{(4, 10)}, 
G_{(4, 12)} \}$ has an abelian normal subgroup 
isomorphic to $C_{p^3}\times C_{p^2}$, or $C_{p^2}\times C_{p^2} \times C_{p}$, 
or $C_{p^2}\times C_{p}^{3}$. 
From Table \ref{t:3}, {$c(G_{(4,2)}) =c(G_{(4, 10)}) = p^3+p^2$.}
From \cite{Arxiv}, 
$G \in \{ G_{(4,4)}, G_{(4,12)} \}$ has the following form:
\begin{align*}
\langle \alpha_{1}, \ldots, \alpha_{5}, \beta_{1}, \beta_{2} 
~|~ & [\alpha_{4}, \alpha_{5}] 
= \alpha_{2} = \beta_{2}, [\alpha_{3}, \alpha_{5}] 
= \alpha_{1} = \beta_{1}^{p}, \alpha_{4}^{p} = \beta_{1},\\
&\alpha_{3}^{p}= \beta_{1}^{p^2} = \beta_{2}^{p} =1,  
\text{ some other relations between the generators} \rangle.
\end{align*}
Here $Z(G) = \langle  \alpha_{2}, \alpha_{4}^{p} \rangle 
\cong C_{p^2}\times C_{p}$, and 
$G' = \langle \alpha_{1}, \alpha_{2} \rangle \cong C_{p} \times C_{p}$. 
Now $A = \langle \alpha_{2}, \alpha_{3}, \alpha_{4} 
\rangle \cong C_{p} \times C_{p} \times C_{p^3}$
is a maximal abelian subgroup of $G$. 
Since $\exp(G) = p^3$, $G' \subset Z(G)$, 
and $\exp(G/Z(G)) = p$, we deduce that 
$\exp(G/G') = p^2$. For $i=1,2$, 
let $\chi_{i}$ and $D_{i} \leq A$ be as defined in Equation \eqref{eq:phi4}. 
Then $\chi_{i} = \rho_{i}\ind_{K_{D_{i}}}^{G} \in \nl(G)$. 
But $\chi_{i}(1) = p$, so $K_{D_{i}} = A$.
As in Claims 1 and 2 of 
the proof of Lemma \ref{prop:phi3} (A), 
we deduce that $A/\ker(\rho_{1}) \cong C_{p^3}$, and 
$A/\ker(\rho_{2}) \cong C_{p}$, or $C_{p^3}$. 
The group structure implies that 
$\mathbb{Q}(\chi_{1}) = \mathbb{Q}(\rho_{1})$. 
Hence $d(\chi_{1}) = p\phi(p^3) = p^3(p-1)$. 
Since $d(\chi_{2}) \geq p\phi(p) = p(p-1)$, 
$$\sum_{i=1}^{2} d(\chi_{i}) \geq p^3(p-1) + p(p-1).$$ 
Lemma \ref{L1} implies that $c(G) \geq p^4+p^2$. 
On the other hand, 
$\Core_{G}(\langle \alpha_{3}, \alpha_{4}) \rangle \cap 
\Core_{G}(\langle \alpha_{3}, \alpha_{2} \rangle) = 1.$
Hence $c(G) = \mu(G) \leq p^4+p^2$. 
Therefore, $c(G) = p^4+p^2$.

\medskip 
\item[(B.)]
If $G\in \Phi_{4}$ and $Z(G) \cong C_{p}^3$, 
then $\exp(G) = p^r$ for $r\leq 2$ (see \cite{Arxiv}). 
Lemma \ref{L1} implies that $|X_{G}|=3$. 
Observe $G' \cong C_{p} \times C_{p}$ and $G' \subset Z(G)$.

\begin{itemize}
\item 
$G$ is a direct product of an abelian and a non-abelian group.

\noindent 
Let $G = H \times K$ where $H$ is non-abelian and 
not a direct product of an abelian and a non-abelian subgroup, 
$Z(H) \cong C_{p}\times C_{p}$, and $K \cong C_{p}$. 
From Lemmas \ref{P2} and \ref{remark:newrange}, 
$c(H) = 2p^2$, $p^3+p^2$, or $2p^3$. 
Hence $c(G) = 2p^2+p$, $p^3+p^2+p$, or $2p^3+p$.

\item 
$G$ is not a direct product of an abelian and a non-abelian subgroup. 

\noindent 
Let $X = \{ \chi_{i} \; | \; 1 \leq i \leq 3\} \subset \Irr(G)$ 
satisfy Equation \eqref{eq:X_G}. 
Suppose $\chi_{i}(1) = 1$ for some fixed $i$ where $1\leq i \leq 3$. 
Then $\frac{G/G'}{\ker(\chi_{i})/G'}\cong C_{p}$, or $C_{p^2}$. 
If the former holds,
then $|\ker(\chi_{i})| = p^5$. Since $G$ is not a direct product 
of an abelian and a non-abelian subgroup, 
$Z(G) \subset \ker(\chi_{i})$. By Lemma \ref{LemmaNotInX_G} (ii), 
$X$ does not satisfy Equation \eqref{eq:X_G}, a contradiction. 
Hence $\frac{G/G'}{\ker(\chi_{i})/G'}\cong C_{p^2}$
for $\chi_{i} \in X \cap \lin(G)$. On the other hand, 
if $\chi_{i}(1) > 1$, 
then $d(\chi_{i}) \geq p\phi(p) = p(p-1)$. 
Thus $d(\chi_{i}) \geq p(p-1)$ for $1\leq i \leq 3$. 
Lemma \ref{L1} implies that $c(G) \geq 3p^2$.
 
If $G$ is a relevant group in 
$\Phi_{4} \setminus \{ G_{(4,28)}, G_{(4,32)},
G_{(4,34r)}\, (r = \omega, \omega^{3}, \ldots, \omega^{{p-2}}) \}$,
then, from Table \ref{t:3}, $c(G) = \mu(G) = 3p^2$.
Consider 
\begin{eqnarray*}
\hspace*{0.8cm} G_{(4,32)} = \langle \alpha_{1}, \ldots, \alpha_{5},  \beta_{1}, \beta_{2}, \beta_{3} & | & [\alpha_{4}, \alpha_{5}] = \alpha_{2} = \beta_{2}, 
[\alpha_{3}, \alpha_{5}] = \alpha_{1} = \beta_{1}, \\
& & \alpha_{3}^{p} = \beta_{2}, 
 \alpha_{4}^{p}= \beta_{1}^{\nu},  
 \alpha_{5}^{p} = \beta_{3}, \beta_{1}^{p} = \beta_{2}^{p} = \beta_{3}^{p} = 1 \rangle.
\end{eqnarray*}
Observe 
$Z(G) = \langle  \alpha_{3}^{p}, 
\alpha_{4}^{p}, \alpha_{5}^{p} \rangle \cong C_{p}^{3}$, and 
$G' = \langle \alpha_{3}^{p}, \alpha_{4}^{p} \rangle 
\cong C_{p} \times C_{p}$, and 
$$\Core_G (\langle \alpha_{5}, \alpha_{3}^{p}) \rangle \cap 
\Core_{G} (\langle \alpha_{5}, \alpha_{4}^{p}) \rangle \cap 
\Core_{G} (\langle \alpha_{3}, \alpha_{4} \rangle ) = 1.$$ 
Hence $c(G) = \mu(G) \leq 2p^3 + p^2$.

Similarly, $c(G_{(4,28)}) \leq p^3+2p^2$, and $c(G_{(4,34r)}) \leq 2p^3+p^2$.

Hence, if $G=G_{(4,28)}$, $G_{(4,32)}$, or 
$G_{(4,34r)}$, then 
$c(G) = 3p^2$, $p^3+2p^2$, or $2p^3+p^2$. \qedhere
\end{itemize} 
\end{enumerate} 
\end{proof}

\begin{lemma}
\label{prop:phi6}
Let $G$ be a group of order $p^6$ in $\Phi_{6}$. 
\begin{enumerate}
\item[(A.)] If $Z(G) \cong C_{p^2} \times C_{p}$, then  
$c(G) = p^3+p^2$, $2p^3$, $p^4+p^2$, or $p^4+p^3$.
\item[(B.)] If $Z(G) \cong C_{p}^3$, then 
$c(G) = 2p^2+p$, $3p^2$, $p^3+p^2+p$, or $2p^3+p$.
\end{enumerate}
\end{lemma}
			
\begin{proof}
Note that $\cd(G) = \{ 1, p \}$, $\exp(G/Z(G)) = p$, 
$G' \cong C_{p}^{3}$, and 
$G' \cap Z(G) \cong C_{p} \times C_{p}$ (see \cite{RJ, Arxiv}). 

\begin{enumerate}
\item[(A.)] 
If $Z(G) \cong C_{p^2}\times C_{p}$, then 
$\exp(G) = p^{r}$ where $r\leq 3$ (see \cite{Arxiv}).
Since $d(Z(G)) = d(G' \cap Z(G))$, 
Remark \ref{remark:notadirectproduct} shows that 
$G$ is not a direct product of an abelian and a non-abelian subgroup.
Lemma \ref{L1} implies that $|X_{G}| = 2$. 
Since $Z(G) \cap G'$ contains all the elements of order $p$ in $Z(G)$, 
$\lin(G) \cap X_{G} = \emptyset$. 
As in Claim 1 of the proof of Lemma 
\ref{prop:phi4} (A), we deduce that $c(G) \geq p^3 + p^2$.

From \cite{Arxiv}, observe that 
if $G\in \Phi_{6}$, then $G$ has no abelian subgroup of index $p$.
Now $$A = Z(G)\cdot G' \cong C_{p^2} \times C_{p} \times C_{p}$$ 
is a maximal abelian subgroup of $G$ containing $G'$. 
Let $X = \{ \chi_{1}, \chi_{2} \} \subset \nl(G)$ 
satisfy Equation \eqref{eq:X_G}. From Lemma \ref{lemma:bakshi}, 
for $i = 1,2$ there exists $\rho_{i} \in \lin(K_{D_{i}})$, 
where $D_{i} \leq A$ with $A/D_{i}$  cyclic and 
$\ker(\rho_{i}\restr_{A}) = D_{i}$, such that 
\begin{equation} \label{eq:phi6}
\chi_{i} = \rho_{i}\ind_{K_{D_{i}}}^{G} \in \Irr(G),
\end{equation}
and $\bigcap_{i=1}^{2} \ker\left(\rho_{i}\ind_{K_{D_{i}}}^{G}\right)  = 1$.
Observe that $|K_{D_{i}}|=p^5$ for $i=1,2$. 

If $\exp(G) \leq p^r$ where $r\leq 2$, then, from the above 
discussion and Lemma \ref{prop:CpXCp}, we conclude that 
$c(G) = p^3+p^2$, or $2p^3$. 

If $\exp(G) = p^3$, then $G$ is one of 
{$G_{(6, 2r)}~ (r=1 \text{ or } \nu), 
G_{(6, 4)}, G_{(6, 6rs)}~ (r,s=1 \text{ or } \nu)$,   
or $G_{(6, 7r)}~ (r=1,\ldots,p-1)$}
(see \cite{Arxiv}).
From Table \ref{t:3}, {$c(G_{(6,4)}) = 
c(G_{(6, 7r)}) = p^3+p^2$}. 
If $G$ is $G_{(6, 2r)}$, or $G_{(6, 6rs)}$,
then $G$ has the following form:
\begin{eqnarray*}
\langle \alpha_{1}, \ldots, 
\alpha_{5}, \beta_{1}, \beta_{2} 
& | & [\alpha_{4}, \alpha_{5}] = \alpha_{3}, [\alpha_{3}, \alpha_{5}] = \alpha_{2} = \beta_{2}, 
 [\alpha_{3}, \alpha_{4}] = \alpha_{1} = \beta_{1}^{p}, 
\alpha_{5}^{p} = \beta_{1}^t,\\
& & \alpha_{3}^{p} =   \beta_{1}^{p^2} = \beta_{2}^{p} = 1, \text{ some other relations between the generators} \rangle
\end{eqnarray*}
where $t=1$, or $\nu$.
Either $o(\alpha_{4}) = p$, or 
$\langle \alpha_{4}^{p} \rangle = 
\langle \alpha_{2} \rangle \cong C_{p}$ (see \cite{Arxiv}).
Observe $\Core_{G} (\langle \alpha_{4}, \alpha_{2} \rangle) 
\cap \Core_{G} (\langle \alpha_{5} \rangle) = 1$. 
Hence $c(G) = \mu(G) \leq p^4+p^3$. 
From Lemmas \ref{P2} and \ref{remark:newrange}, 
for every $G \in \Phi_{6}$ with $Z(G) \cong C_{p^2}\times C_{p}$, 
we deduce that $c(G) = p^3+p^2$, $2p^3$, $p^4+p^2$, or $p^4+p^3$.

\item[(B.)] 
Here $\exp(G) = p^2$ (see \cite{Arxiv}).
\begin{itemize}
\item Sub-case 1: 
$G=H\times K$ where $H$ is non-abelian and $K$ abelian.

\noindent 
We claim that $K \cong C_{p}$.
Suppose, to the contrary, that $Z(H) \cong C_{p}$, and 
$K\cong C_{p} \times C_{p}$. 
Now $G' = H'$ and 
$Z(G) \cap G' \cong C_{p} \times C_{p}$, 
and $Z(G) \cap H' \cong C_{p} \times C_{p}$, so 
$Z(H) \cap H' \cong C_{p} \times C_{p}$. 
This is a contradiction since $Z(H) \cong C_{p}$. 
Therefore, if $G=H \times K$, then $Z(H) \cong C_{p}\times C_{p}$ 
and $K \cong C_{p}$. Without loss of generality, assume $H$ is not 
a direct product of an abelian and a non-abelian subgroup. 
From Lemmas \ref{P2} and \ref{remark:newrange}, 
we conclude that $c(H) = 2p^2$, $p^3+p^2$, or $2p^3$. 
Hence $c(G) = 2p^2+p$, $p^3+p^2+p$, or $2p^3+p$.

\item Sub-case 2: $G$ is not a direct product of an 
abelian and a non-abelian subgroup.

\noindent 
There are four such groups $G$ in $\Phi_{6}$: 
namely $G_{(6,9)}$, $G_{(6,12r)}$ where $r=1$ or $\nu$, 
and $G_{(6,13)}$. 
From Table \ref{t:3}, $c(G) = \mu(G) = 3p^2$. \qedhere
\end{itemize}
\end{enumerate}
\end{proof}

\begin{lemma} \label{lemma:phi11}
	Let $G$ be a group of order $p^6$ in $\Phi_{11}$.
	Suppose $\SS$ and $\TT$ are maximal subgroups of $G$,
	with $\SS' \neq \TT'$ and $\SS/\SS' \ncong \TT/\TT'$. 
	If $\TT/\TT'$ is isomorphic to $C_{p^2} \times C_{p^2}$, 
	then $c(G) \leq 2p^3+p^2$.
\end{lemma}

\begin{proof}
	From \cite{Arxiv}, $G$ has the following form:
	\begin{align*}
\langle \alpha_{1}, \ldots, \alpha_{5}, \alpha_{6} ~|~ & [\alpha_{5}, \alpha_{6}] = \alpha_{3}, [\alpha_{4}, \alpha_{6}] = \alpha_{2}, [\alpha_{4}, \alpha_{5}] = \alpha_{1},\\
& \alpha_{1}^{p} = \alpha_{2}^{p} = \alpha_{3}^{p} = 1, 
\text{ some other relations between the generators} \rangle.
	\end{align*}
	Note that $\cd(G) = \{ 1, p \}$, $\exp(G) = p^r$ where $r\leq 2$, 
	$Z(G) = G' \cong C_{p}^3$, and $G/Z(G) \cong C_{p}^3$ (see \cite{RJ, Arxiv}).
	Since $Z(G) = G'$, by Remark \ref{remark:notadirectproduct}, 
	$G$ is not a direct product of an abelian and a non-abelian subgroup. 
	From the structure of $G$, observe that if $\SS\leq G$ has index $p$,
	then $\SS' \cong C_{p}$. Lemma \ref{L1} implies that 
	$|X_{G}|=3$, and, 
	by Lemma \ref{lemmalinearcharacter},
	$X_{G} \cap \lin(G) = \emptyset$. 
	Suppose there exist maximal subgroups $\SS$ and $\TT$ of $G$
	with $\SS' \neq \TT'$ and $\SS/\SS' \ncong \TT/\TT'$. 
	If $\TT/\TT'$ is isomorphic to $C_{p^2} \times C_{p^2}$,
	then $\SS/\SS' \cong C_{p^2} \times C_{p} \times C_{p}$, or $C_{p}^4$.
	Now $\exp(G) = p^2$, and $Z(G)$ is contained in both $\SS$ and $\TT$. 

Let $\lambda_{1} \in \lin(\TT/\TT')$ be such that 
$Z(G) \nsubseteq \ker(\lambda_{1})$. 
Hence $\frac{\TT/\TT'}{\ker(\lambda_{1})/ \TT'}\cong C_{p}$, or $C_{p^2}$. 
If the former holds,
then $|\ker(\lambda_{1})/\TT'|=p^3$, so 
$\ker(\lambda_{1})/\TT' \cong C_{p^2} \times C_{p}$. 
Then $\ker(\lambda_{1})/\TT'$ contains all the elements of order
$p$ in $\TT/\TT'$. 
But $\TT/\TT' \geq Z(G)/\TT' \cong C_{p}\times C_{p}$, so 
$Z(G)/ \TT' \subseteq \ker(\lambda_{1})/\TT'$, a contradiction.
Hence $\frac{\TT/\TT'}{\ker(\lambda_{1})/\TT'} \cong C_{p^2}$.  So 
$$\frac{\ker(\lambda_{1})}{\TT'} \cong C_{p^2} 
\Rightarrow \frac{\ker(\lambda_{1})}{\TT'} 
= \langle t \TT' \rangle \cong C_{p^2}$$ 
for some $t\in \ker(\lambda_{1})$. 
But $\exp(G) = p^2$, so $t^{p^2} = 1$. 
Thus $\langle t \rangle \cap \TT' = 1$, 
and $\ker(\lambda_{1}) = \langle t \rangle \cdot \TT'$. 
Since $\langle t \rangle$ and $\TT'$ 
are both normal, 
$$\ker(\lambda_{1}) = \langle t \rangle \times \TT' \cong 
C_{p^2} \times C_{p}.$$
Note that $t^{p} \in Z(G)$, as $\exp(G/Z(G)) = p$. 
Choose $y,z \in G$ such that 
$\TT' = \langle y \rangle \cong C_{p}$ and 
$\SS' = \langle z \rangle \cong C_{p}$. 
Observe that $\langle t^{p}, y \rangle \subset Z(G) \cap \ker(\lambda_{1})$. 
Choose $x\in \TT$ such that 
$\ker(\lambda_{1}) = \langle x, y \rangle \cong C_{p^2} \times C_{p}$, 
and $Z(G) = \langle x^{p}, y, z \rangle \cong C_{p}^3$.
Define $B= \langle x, y, z \rangle \cong 
C_{p^2} \times C_{p} \times C_{p}$, and 
$D = \langle x, y \rangle$. Then $B$ is a maximal 
abelian subgroup of $G$ containing $G'$, and 
$B/D$ is cyclic. {Since $\TT$ is a maximal element of 
$\{ \MTG \leq G ~|~ B \leq \MTG \text{ and } \MTG' \leq D \}$, we take $K_{D} = \TT$.} 
As $\ker(\lambda_{1}\restr_{B}) = D$, 
$\lambda_{1}\ind_{\TT}^{G} \in \nl(G)$ (see Lemma \ref{lemma:bakshi}). 
We consider two cases according to the exponent of $\SS$.

\medskip
\noindent {\bf Case 1: $\exp(\SS) = p^2$}.

\noindent 
There exists $a\in \SS$ with $o(a) = p^2$. 
Set $A = \langle a, x^{p}, y, z \rangle$. 
Since $\exp(G/Z(G)) = p$, $a^{p} \in Z(G)$. 
Then $A \cong C_{p^2} \times C_{p} \times C_{p}$ with $Z(G) \subset A$.
Now $a^{p} = x^{pi'}y^{j'}z^{k'}$ for some $0\leq i',j',k' \leq p-1$, 
where at least one of $i'$, $j'$ and $k'$ belongs to $\mathbb{Z}_{p}^{\times}$. 
Then $A$ is one of the following.
\begin{enumerate}
	\item[Form 1:]
	if $i' \in \mathbb{Z}_{p}^{\times}$, then 
	$A = \langle a, y, z \rangle \cong C_{p^2} \times C_{p} \times C_{p}$. 
	\item[Form 2:] 
	if $j' \in \mathbb{Z}_{p}^{\times}$, 
	then $A = \langle x^p, a, z \rangle \cong C_{p} \times C_{p^2} \times C_{p}$.
	\item[Form 3:] 
	If $k' \in \mathbb{Z}_{p}^{\times}$, then 
	$A = \langle x^p, y, a \rangle \cong C_{p} \times C_{p} \times C_{p^2}$. 
\end{enumerate}

Consider the first of these: 
$A = \langle a, y, z \rangle \cong C_{p^2} \times C_{p} \times C_{p}$ 
is a maximal abelian subgroup of $G$ containing $G'$. 
Let $D_{2} = \langle y, z \rangle$ and 
$D_{3} = \langle a, z \rangle$.
{For $i=2,3$, $A/D_{i}$ is cyclic and $\SS$ is a maximal element 
of $\{ \MTG \leq G ~|~ A \leq \MTG \text{ and } \MTG' \leq D_{i} \}$.}
Thus we can take $K_{D_{i}} = \SS$ for both $i$. 
Choose $\lambda_{i} \in \lin(\SS)$ such that 
$\ker(\lambda_{i}\restr_{A}) = D_{i}$.
From Lemma \ref{lemma:bakshi}, we conclude that 
$\lambda_{i}\ind_{\SS}^{G} \in \nl(G)$.

We claim that 
$$ \ker(\lambda_{1}) \bigcap 
\left( \bigcap_{i=2}^{3} \Core_{G}(\ker(\lambda_{i})) \right) = 1.$$
Take $c = x^{i_{1}}y^{j_{1}} \in \ker(\lambda_{1})$
for some $0\leq i_{1} \leq p^2-1$, $0 \leq j_{1} \leq p-1$. 
If $x^{i_{1}}y^{j_{1}} \in 
\bigcap_{i=2}^{3}\Core_{G}(\ker(\lambda_{i}))$, 
then $x^{i_{1}}y^{j_{1}} \in \Core_{G}(\ker(\lambda_{2}))$. 
Then $x^{i_{1}} \in \Core_{G}(\ker(\lambda_{2}))$, so 
$x^{pi_{1}} \in \Core_{G}(\ker(\lambda_{2}))$. 
But this implies that $Z(G) \subset \Core_{G}(\ker(\lambda_{2}))$, 
a contradiction unless $i_{1}=0$. Hence $c = y^{j_{1}}$. 
But $y^{j_{1}} \in \bigcap_{i=2}^{3}\Core_{G}(\ker(\lambda_{i}))$ 
so $Z(G) = \langle a^{p}, y^{j_{1}}, z \rangle 
\subset \Core_{G}(\ker(\lambda_{3}))$, a contradiction. 
We conclude that $j_{1}=0$. This establishes the claim. 

If $\SS/\SS' \cong C_{p}^{4}$, then 
$d(\lambda_{i})= \phi(p) = p-1$, for $i=2,3$. 
On the other hand, $d(\lambda_{1}) = \phi(p^2) = p(p-1)$.
Thus $$d(\lambda_{1}\ind_{\TT}^{G}) + 
\sum_{i=2}^{3} d(\lambda_{i}\ind_{\SS}^{G}) \leq  pd(\lambda_{1}) 
+  \sum_{i=2}^{3} pd(\lambda_{i}) = p^2(p-1) + 2p(p-1),$$ 
which implies that $c(G) \leq p^3+2p^2$.

If $\SS/\SS' \cong C_{p^2} \times C_{p} \times C_{p}$,
then $o(a\SS') = p^2$. If $\frac{\SS/\SS'}{\ker(\eta)/\SS'} \cong C_{p^2}$
for some $\eta\in \lin(\SS/\SS')$, then, from Lemma \ref{lemma:abelian}, 
$$\frac{\SS/\SS'}{\ker(\eta)/\SS'} = \left\langle a\SS' 
\frac{\ker(\eta)}{\SS'} \right\rangle \cong C_{p^2}.$$ 
But $\ker(\lambda_{3}) \cap Z(G) = \langle a^p, z \rangle$, 
so $\frac{\SS/\SS'}{\ker(\lambda_{3})/\SS'} \cong C_{p}$. 
Hence $d(\lambda_{3}) = \phi(p) = p-1$, which implies 
that $d(\lambda_{3}\ind_{\SS}^{G}) \leq p(p-1)$. 
On the other hand,  $d(\lambda_{2}\ind_{\SS}^{G}) \leq p\phi(p^2)= p^2(p-1)$ 
and $d(\lambda_{1}\ind_{\TT}^{G}) \leq p\phi(p^2)= p^2(p-1)$.
Thus $$d(\lambda_{1}\ind_{\TT}^{G}) + \sum_{i=2}^{3} d(\lambda_{i}\ind_{\SS}^{G}) 
\leq  pd(\lambda_{1}) +  \sum_{i=2}^{3} pd(\lambda_{i})
= p^2(p-1)+ p^2(p-1) + p(p-1),$$ which implies that $c(G) \leq 2p^3+p^2$. 

If $A$ has form 2, then, by a procedure similar to above, 
$c(G) \leq 2p^3+p^2$. 

Now let $A$ have form 3. Suppose there exists $b\in \SS$ 
of order $p^2$ such that $A_{1} = \langle b, x^p, y, z \rangle$ has 
either form 1 or form 2. Note that $A_{1}$ is a maximal abelian subgroup 
of $G$ containing $G'$. Proceeding as in the proof for 
form 1, we deduce that $c(G) \leq 2p^3+p^2$. If no such $b$ exists, 
then $\exp(\SS/\SS')=p$, as $\SS' = \langle z \rangle\cong C_{p}$. 
Again, 
we deduce that $c(G) \leq p^3+2p^2$.

\medskip
\noindent {\bf Case 2: $\exp(\SS) = p$}. 

\noindent 
{Now $\SS/\SS' \cong C_{p}^{4}$.}
Suppose $a\in \SS$ but $a \notin Z(G)$. 
Set $A = \langle a, x^p, y, z \rangle \cong C_{p}^{4}$. 
Let $D_{2} = \langle a, y, z \rangle$ and $D_{3} = \langle a, x^p, z \rangle$. 
For $i=2,3$, $D_{i} \leq A$ with $A/D_{i} \cong C_{p}$, 
{and 
$\SS$ is a maximal element of 
$\{ \MTG \leq G ~|~ A \leq \MTG \text{ and } \MTG' \leq D_{i} \}$.} Thus we can take 
$K_{D_{i}} = \SS$ for both $i$. 
Choose $\lambda_{i} \in \lin(\SS)$ such that 
$\ker(\lambda_{i}\restr_{A}) = D_{i}$. 
From Lemma \ref{lemma:bakshi}, $\lambda_{i}\ind_{\SS}^{G} \in \nl(G)$. 
Proceeding as above, we deduce that $c(G) \leq p^3+2p^2$. 
\end{proof}
				
\begin{lemma}
\label{prop:phi11}
If $G$ is a group of order $p^6$ in $\Phi_{11}$, 
then $c(G) = 3p^2$, $p^3+2p^2$, or $2p^3+p^2$.
\end{lemma}

\begin{proof}
Note that $\exp(G) = p^r$ where $r\leq 2$ (see \cite{Arxiv}). 
Lemma \ref{L1} implies that $|X_{G}| = 3$. But $X_{G} \cap \lin(G) = \emptyset$ 
(from Lemma \ref{lemmalinearcharacter}), so 
$d(\chi) \geq p\phi(p) = p(p-1)$ for all $\chi \in X_{G}$. 
Lemma \ref{L1} implies that $c(G)\geq 3p^2$.

From Lemma \ref{P2}, if $\exp(G) = p$, then $c(G) = 3p^2$. 
We now consider groups of exponent $p^2$.
Let $$G \in \Phi_{11}\setminus {\{ G_{(11, 2)}, G_{(11, 5)}, 
G_{(11, 8)}, G_{(11, 16r)}\ (r=2,3,\ldots,(p-1)/2), 
G_{(11, 17r)}\ (r=1,2,\ldots,(p-1)/2) \}}.$$ 
From \cite{Arxiv}, observe that 
$\SS = \langle \alpha_{5}, \alpha_{4}, \alpha_{3}, \alpha_{2}, 
\alpha_{1} \rangle$ and $\TT = \langle \alpha_{6}, \alpha_{5}, 
\alpha_{3}, \alpha_{2}, \alpha_{1} \rangle$
are maximal subgroups of $G$, and
exactly one of $\SS/\SS'$ and $\TT/\TT'$ is 
isomorphic to $C_{p^2} \times C_{p^2}$.
Lemma \ref{lemma:phi11} implies that $c(G) \leq 2p^3+p^2$. 
From Table \ref{t:3}, 
if {$G\in \{  G_{(11, 5)}, G_{(11, 8)}, G_{(11, 16r)}, G_{(11, 17r)} \}$},
then $c(G) = \mu(G) = 3p^2$. 
Consider
\begin{align*}
	{G_{(11, 2)}} = \langle \alpha_{1}, \ldots, \alpha_{5}, \alpha_{6} ~|~ & [\alpha_{5}, \alpha_{6}] = \alpha_{3}, [\alpha_{4}, \alpha_{6}] = \alpha_{2}, [\alpha_{4}, \alpha_{5}] = \alpha_{1}, \alpha_{6}^{p} = \alpha_{1}, \alpha_{1}^{p} = \alpha_{2}^{p} =\alpha_{3}^{p} =  \alpha_{4}^{p} = \alpha_{5}^{p}=  1 \rangle.
\end{align*}
{If 
$H_{1} = \langle \alpha_{5}, \alpha_{4}, \alpha_{3}, \alpha_{1} \rangle$, 
$H_{2} = \langle \alpha_{6}, \alpha_{4}, \alpha_{1} \rangle$, 
and $H_{3} = \langle \alpha_{5}, \alpha_{3}, \alpha_{2} \rangle$,} 
then {$\bigcap_{i=1}^{3}\Core_{G_{(11, 2)}}(H_{i}) = 1$.} 
Hence {$c(G_{(11, 2)}) = \mu(G_{(11, 2)}) \leq 2p^2+p^3$.}
From Lemma \ref{P2}, 
we conclude that $c(G) = 3p^2$, $p^3+2p^2$, or $2p^3+p^2$ for every $G\in \Phi_{11}$. 
\end{proof}
					
We summarize the outcome of this section. 
\begin{theorem} \label{thm:p^3}
If $G$ is a group of order $p^{6}$ with a center
of order $p^3$, then $c(G)$ is one of the following: 
$p^5$, $p^4 + p^3$, $p^4+p^2$, $p^4+p$, $p^4$, $2p^3+p^2$, 
$2p^3+p$, $2p^3$, $p^3+2p^2$, $p^3+p^2+p$, $p^3 +p^2$, $p^3+2p$, 
$p^3+p$, $3p^2$, $2p^2+p$, $2p^2$ or $p^2+2p$.
\end{theorem}
					
\subsection{Groups with center of order $p^{2}$} \label{SS4}
We now consider those groups of order $p^6$ which have
center of order $p^2$. These groups are in $\Phi_{i}$ where 
$i \in \{ 5, 7,8, 9, 10, 12, 13,\ldots, 21, 23 \}$.
All are metabelian, as their 
commutator subgroups are abelian (see \cite{RJ}).
{
If $G\in \{\Phi_{5}, \Phi_{15}\}$, 
then $G$ is a VZ group; see Table \ref{t:VZ} for $\mu(G)$. 
}
If $G \in \{\Phi_{9}, \Phi_{16}\}$, then $\cd(G) = \{ 1, p \}$; 
if $G \in \Phi_{i}$
where {$i \in \{ 7, 8, 10, 12, 13, 14, 17, \ldots, 21, 23 \}$}, 
then $\cd(G) = \{ 1, p, p^2 \}$ (see \cite{RJ}). 
					
\begin{lemma}
\label{prop:directproduct}
Let $G$ be a group of order $p^6$ 
with center of order $p^2$
and let $G$ be a direct product of an 
abelian and a non-abelian subgroup. Then $c(G) = p^2 + p$ or $p^3 + p$.
\end{lemma}
\begin{proof}
From Remark \ref{remark:notadirectproduct}, $Z(G) \cong C_{p} \times C_{p}$.
Suppose $G = H \times K$ where $H$ is non-abelian
and $K$ is abelian. Since $Z(H)$ is a proper subgroup of $Z(G)$,
we deduce that $|H| = p^5$, $Z(H) \cong C_{p}$, and $K\cong C_{p}$. 
From \cite{RJ}, observe that 
$H \in  \Phi_{i}$ for $i \in \{ 5,7,8,9,10 \}$. 
If $H\in \Phi_{5}$, then, from \cite[Table 2]{SA}, $c(H) = p^3$. 
If $H\in \Phi_{i}$, for $7\leq i \leq 10$, has order $p^5$, 
then it is tabulated in Table \ref{t:2} and 
from there we deduce that $c(H) = \mu(H) = p^2$, or $p^3$.
Hence $c(G) = p^2+p$, or $p^3+p$. 
\end{proof}
					
\begin{lemma}
\label{prop:cdG=1,p}
Let $G$ be a group of order $p^6$ 
such that $\cd(G) = \{ 1, p \}$, $|Z(G)| = p^2$, 
and $G$ is not a direct product of an abelian and a non-abelian subgroup. 
Then $c(G) = 2p^2$, $p^3$, $p^3 + p^2$, $2p^3$, or $p^4$.
\end{lemma}
\begin{proof}
Groups satisfying the above hypothesis 
belong to $\Phi_{9}$, or $\Phi_{16}$. So $\exp(G) = p^r$ where $r\leq 3$
(see \cite{Arxiv}). If $Z(G) \cong C_{p^2}$, then, from 
Lemma \ref{Z(G)isCyclic}, $c(G) = p^3$ or $p^4$. 
Let $Z(G) \cong C_{p} \times C_{p}$. 
Now $\exp(G) = p^r$ where $r\leq 2$ (see \cite{Arxiv}).
From Lemmas \ref{P2} and \ref{remark:newrange}, 
we conclude that $c(G) = 2p^2$, $p^3 + p^2$, or $2p^3$. 
\end{proof}
				
\begin{lemma}
\label{prop:phi7}
Let $G$ be a group of order $p^6$ in $\Phi_{7}$ such that 
$G$ is not a direct product of an abelian 
and a non-abelian subgroup. Then $c(G) = p^3+p^2$, or $p^4$.
\end{lemma}
\begin{proof}
Note that $\cd(G) = \{ 1,p,p^2 \}$, $G' \cong C_{p} \times C_{p}$, 
$Z(G) \cong C_{p} \times C_{p}$ or $C_{p^2}$, and  
$Z(G) \cap G' \cong C_{p}$
(see \cite{RJ, Arxiv}).
For every group in $\Phi_{7}$, except $G_{(7, 14)}$, 
$c(G) = \mu(G) = p^3+p^2$, or $p^4$ (see Table \ref{t:3}). Consider 
\begin{align*}
G := G_{(7, 14)} = \langle \alpha_{1}, \ldots, \alpha_{5}, \beta_{1}, \beta_{2} ~|~ &[\alpha_{4}, \alpha_{5}] = \alpha_{2}, [\alpha_{2}, \alpha_{5}] =  [\alpha_{3}, \alpha_{4}] = \alpha_{1} = \beta_{1}, \alpha_{3}^{p}= \beta_{1}, \alpha_{5}^{p} = \beta_{2},\\
& \alpha_{2}^{p} = \alpha_{4}^{p} = 
\beta_{1}^{p} = \beta_{2}^{p} = 1 \rangle.
\end{align*}
Here $Z(G) = \langle \alpha_{5}^{p}, \alpha_{3}^{p} \rangle 
\cong C_{p} \times C_{p}$, 
$G' = \langle \alpha_{2}, \alpha_{3}^{p} \rangle \cong C_p \times C_p$, 
and $G/G' = \langle \alpha_{5} G', \alpha_{4}G', 
\alpha_{3} G' \rangle \cong C_{p^2} \times C_{p} \times C_{p}$. 
Lemma \ref{L1} implies that $|X_{G}|=2$. 
Let $X = \{ \chi_{1}, \chi_{2} \} \subset \Irr(G)$ 
satisfy Equation \eqref{eq:X_G}. Let $\chi_{1}(1) = p$. 
Then $G$ has a subgroup $\SS$ of index $p$ such that 
$\chi_{1} = \lambda\ind_{\SS}^{G}$ for some $\lambda \in \lin(\SS)$, 
and $\SS' \lhd G'$. 
Thus either $\SS' = 1$ or $\SS' \cap Z(G) \neq 1$. 
From Lemma \ref{thm:cdG1}, $G$ has no abelian subgroup of index $p$, so 
$\SS'\neq 1$. Thus 
$\langle \alpha_{3}^p \rangle \subseteq \SS' \subseteq \ker(\lambda)$, 
so $\langle \alpha_{3}^p \rangle \subseteq \ker(\lambda\ind_{\SS}^{G})$. 
Hence $|X \cap \{ \chi\in \nl(G)~|~ \chi(1) = p \}|\leq 1$. 
Since $\alpha_{3}^{p} \in G'$, 
one of the following holds:
$\chi_{1}(1)=1$ and $\chi_{2}(1) = p^2$; 
or $\chi_{1}(1)=p$ and $\chi_{2}(1) = p^2$; 
or $\chi_{i}(1)=p^2$ for $i=1,2$. 
As $G$ is not a direct product of an abelian and a non-abelian subgroup, 
$G / \ker(\eta) \cong C_{p^2}$
for every $\eta \in \lin(G)\cap X$, so $d(\eta) = p(p-1)$. 
Moreover, $d(\chi) \geq \chi(1)(p-1)$ for $\chi \in \nl(G)$.
If $\chi_{1}(1)=1$ and $\chi_{2}(1) = p^2$, then 
$$\sum_{i=1}^{2}d(\chi_{i}) \geq \phi(p^2)+p^2\phi(p) 
= p(p-1) + p^2(p-1).$$ If $\chi_{1}(1)=p$ and $\chi_{2}(1) = p^2$, 
then 
$$\sum_{i=1}^{2}d(\chi_{i}) \geq p\phi(p)+p^2\phi(p) = p(p-1) + p^2(p-1).$$ 
If $\chi_{i}(1)=p^2$ for $i=1,2$, then 
$\sum_{i=1}^{2}d(\chi_{i}) \geq 2p^2\phi(p) = 2p^2(p-1)$.

Therefore, in all three cases, 
$\sum_{i=1}^{2}d(\chi_{i}) \geq p(p-1) + p^2(p-1)$. 
Lemma \ref{L1} implies that $c(G) \geq p^3 + p^2$.
On the other hand, 
$\Core_{G}(\langle \alpha_{2}, \alpha_{3}, \alpha_{4} \rangle) \cap 
\Core_{G} (\langle  \alpha_{2}, \alpha_{4}, \alpha_{5}^{p} \rangle) = 1$, 
so $c(G) = \mu(G) \leq p^3+p^2$. Therefore, $c(G) = p^3+p^2$. 
\end{proof}
					
\begin{lemma}
\label{prop:phi8}
Let $G$ be a group of order $p^6$ in $\Phi_{8}$
such that $G$ is not a direct product of an abelian and a 
non-abelian subgroup. Then $c(G) = p^3 + p^2$, $2p^3$, or $p^4$.
\end{lemma}
\begin{proof}
Note that $\cd(G) = \{ 1, p, p^2 \}$, $G' \cong C_{p^2}$, 
$Z(G) 
\cong C_{p} \times C_{p}$ or $C_{p^2}$,
and $Z(G) \cap G' \cong C_{p}$ 
(see \cite{RJ, Arxiv}).

\medskip
\noindent 
{\bf Case I}: $ Z(G) \cong C_{p}\times C_{p}$. 

\noindent 
Lemma \ref{L1} implies that $|X_{G}|=2$. 
From Lemma \ref{thm:cdG1}, $G$ has no abelian subgroup of index $p$.
But $Z(G)\cap G'\cong C_{p}$ where $G'$ is cyclic, 
so $\SS'\cap Z(G) \neq 1$ for every subgroup $\SS$ of index $p$ in $G$. 
As in the case of $G_{(7,14)}$ in the proof 
of Lemma \ref{prop:phi7}, we deduce  that $c(G) \geq p^3+p^2$. 
The groups in $\Phi_{8}$ which 
satisfy the above hypothesis are $G_{(8,6)}$ and $G_{(8, 7)}$ 
(see \cite{Arxiv}). Table \ref{t:3} shows
$c(G_{(8,6)}) = p^3+p^2$.
Consider 
\begin{align*}
G := G_{(8, 7)} = \langle \alpha_{1}, \ldots, \alpha_{5}, \beta_{1}, \beta_{2} ~|~ & [\alpha_{4}, \alpha_{5}] = \alpha_{2}, [\alpha_{2}, \alpha_{5}] =  [\alpha_{3}, \alpha_{4}] = \alpha_{1} = \beta_{1}, \alpha_{4}^{p} = \alpha_{2},\\
 &\alpha_{5}^{p} = \alpha_{3}^{-1}, \alpha_{5}^{p^2} = \beta_{2},   \beta_{1}^{p} = \beta_{2}^{p} = 1 \rangle
\end{align*}
Here $Z(G) = \langle \alpha_{5}^{p^2}, \alpha_{2}^{p} \rangle 
\cong C_{p} \times C_{p}$, 
$G' = \langle \alpha_{2} \rangle \cong C_{p^2}$, and 
$Z(G) \cap G' = \langle \alpha_{2}^{p} \rangle \cong C_{p}$. 
Observe that $G$ is metacyclic with non-cyclic center: 
from \cite[Lemma 3.10]{BehraveshMetacyclic}, we conclude that $c(G) = 2p^3$. 

\medskip
\noindent 
{\bf Case II}: $Z(G) \cong C_{p^2}$.

\noindent 
Lemma \ref{L1} implies that $|X_{G}|=1$. There are $p+2$ groups 
in $\Phi_{8}$ which satisfy the above hypothesis: 
$G_{(8,1)}$, $G_{(8,2r)}$ 
($r=0, \ldots, p-2$), $G_{(8,3)}$, and $G_{(8,4)}$.
Table \ref{t:3} shows $c(G_{(8, 4)}) = p^4$.
Let {$G \in \{ G_{(8,1)}, G_{(8,2r)}, G_{(8,3)} \}$} and take  
$A = \langle \alpha_{4}, \alpha_{2} \rangle$. 
From \cite{Arxiv}, observe 
that $A$ is an abelian subgroup of order $p^4$,
$\exp(A) = p^2= |Z(G)|$,  and 
$G' \subseteq A$. Since $\cd(G) = \{ 1, p, p^2 \}$, $A$ is a 
normal abelian subgroup of maximal order among all abelian subgroups 
of $G$ (from Lemma \ref{thm:cdG1}). 
Lemma \ref{cor:normallymonomial} implies that $c(G) = p^4$. 
\end{proof}
						
\begin{lemma}
\label{prop:phi10}
Let $G$ be a group of order $p^6$ in $\Phi_{10}$
such that $G$ is not a direct product of an abelian and a non-abelian subgroup. 
Then $c(G) = p^3+p^2$, or $p^4$.
\end{lemma}
\begin{proof}
Note that $\cd(G) = \{ 1, p, p^2 \}$, 
$\exp(G/Z(G))=p$, $G' \cong C_{p}^3$, 
$Z(G) \cap G' \cong C_{p}$, and 
$Z(G) \cong C_{p} \times C_{p}$ or $C_{p^2}$ (see \cite{RJ, Arxiv}).

\medskip
\noindent 
{\bf Case I}: $Z(G) \cong C_{p} \times C_{p}$.

\noindent 
Lemma \ref{L1} implies that $|X_{G}|=2$. The groups 
in $\Phi_{10}$ which satisfy the above hypothesis are
{$G_{(10, 6)}$, $G_{(10, 8)}$, $G_{(10, 9r)}$ 
(where $r=1,\omega,\omega^2$ when $p\equiv 1\bmod 3$, and $r=1$ when 
$p\equiv 2\bmod 3$) and $G_{(10, 10r)}$} 
(where $r=1,\omega,\omega^2, \omega^3$ when $p\equiv 1\bmod 4$, and 
$r=1, \omega$ when $p\equiv 3\bmod 4$). 
From Table \ref{t:3}, we conclude that  
$c(G_{(10, 8)}) = c(G_{(10, 10r)}) = p^3+p^2$.

If {$G \in \{ G_{(10,6)}, G_{(10, 9r)} \}$,} then $G$ has the following form: 
\begin{align*}
\langle \alpha_{1}, \ldots,  \alpha_{5}, \beta_{1}, \beta_{2} ~|~ & [\alpha_{4}, \alpha_{5}] = \alpha_{3}, [\alpha_{3}, \alpha_{5}] = \alpha_{2}, [\alpha_{2}, \alpha_{5}] = [\alpha_{3}, \alpha_{4}] = \alpha_{1} = \beta_{1}, {\alpha_{5}^{p}= \beta_{2}},\\
& \alpha_{2}^{p} = \alpha_{3}^{p} = \beta_{1}^{p} = \beta_{2}^{p} =  1, \text{ some other relations between the generators} \rangle.
\end{align*}
Moreover, {either $\alpha_{4}^{p} = 1$, or 
$\langle \alpha_{4}^{p} \rangle = \langle \alpha_{1}\rangle$}. 
Here $Z(G) = \langle \alpha_{5}^{p}, \alpha_{1} \rangle 
\cong C_{p} \times C_{p}$, and 
$G' = \langle \alpha_{3}, \alpha_{2}, \alpha_{1} \rangle \cong C_p^3$. 
From Lemma \ref{thm:cdG1}, $G$ has no abelian subgroup of index $p$.
Hence $1 \neq Z(G) \cap G' \subseteq \SS'$ for all maximal
subgroups $\SS$.
As in the case of $G_{(7, 14)}$ in the proof 
of Lemma \ref{prop:phi7}, we deduce that $c(G) \geq p^3+p^2$. 
On the other hand, 
$\langle[\alpha_{4}, \alpha_{3}]\rangle = 
\langle[\alpha_{5}, \alpha_{2}]\rangle = 
\langle \alpha_{1} \rangle$, 
so $\Core_{G}\langle \alpha_{5}^{p}, \alpha_{3}, \alpha_{2} \rangle 
= \langle \alpha_{5}^{p} \rangle$.
Thus $$\Core_{G}(\langle \alpha_{5}^{p}, \alpha_{3}, \alpha_{2} \rangle) 
\cap \Core_{G}(\langle \alpha_{4}, \alpha_{3}, \alpha_{2}, 
\alpha_{1} \rangle) = 1,$$ where 
$|\langle \alpha_{4}, \alpha_{3}, \alpha_{2}, \alpha_{1}\rangle| = p^4$. 
Hence $c(G) = \mu(G) \leq p^3+p^2$. 
We conclude that $c(G) =p^3+p^2$.

\medskip 
\noindent 
{\bf Case II}: $Z(G) \cong C_{p^2}$. 

\noindent Lemma \ref{L1} implies that $|X_{G}|=1$. 
Let $\chi \in \Irr(G)$ be such that $\ker(\chi) = 1$. 
Lemma \ref{thm:cdG3} implies that $\chi(1) = p^2$. 
Let $A = Z(G) \cdot G' \cong C_{p^2}\times C_{p} \times C_{p}$. 
From Lemma \ref{thm:cdG1}, we deduce that
$A$ is a normal abelian subgroup of maximal 
order among all abelian subgroups of $G$. 
But $\exp(A) = p^2 = |Z(G)|$, so, from 
Lemma \ref{cor:normallymonomial}, we conclude that $c(G) = p^4$. 
\end{proof}
						
\begin{remark}
{\rm If $G \in  \Phi_{i}$ where 
$i \in \{ 12, 13, 17, 18, 19, 20, 21, 23 \}$, 
then $G'$ is elementary abelian and $Z(G) \leq G'$ (see \cite{RJ, Arxiv}), 
so $Z(G) \cong C_{p}\times C_{p}$. 
If $G\in \Phi_{14}$, then $G' =Z(G) \cong C_{p^2}$. 
From Remark \ref{remark:notadirectproduct}, 
none of these groups is a direct product of an 
abelian and a non-abelian subgroup.
}
\end{remark}

\begin{lemma}
\label{prop:phi12}
If $G$ is a group of order $p^6$ in $\Phi_{12}$, 
then $c(G) = 2p^2$, $p^3+p^2$, or $2p^3$.
\end{lemma}
\begin{proof}
From \cite{Arxiv}, $G$ has the following form:
\begin{align*}
\langle \alpha_{1}, \ldots, \alpha_{6} ~|~ &[\alpha_{3}, \alpha_{4}] = \alpha_{1}, 
{[\alpha_{5}, \alpha_{6}] = \alpha_{2}}, \alpha_{1}^{p}= \alpha_{2}^{p} =  1, \text{ some other relations between the generators} \rangle.
\end{align*}
Note that $\cd(G) = \{ 1, p, p^2 \}$, $\exp(G) = p^r$ where $r\leq 2$, and 
$Z(G) = G' = \langle \alpha_{1}, \alpha_{2} \rangle \cong 
C_{p} \times C_{p}$ (see \cite{RJ, Arxiv}).
Let $G\in 
\Phi_{12} 
\setminus \{  G_{(12, 7)}, 
 G_{(12, 13)}, G_{(12, 14)}, 
G_{(12, 15)},  G_{(12, 16)} \}$. 
From \cite{Arxiv}, observe that $G$ has 
an elementary abelian normal subgroup of order $p^4$. 
Lemma \ref{prop:CpXCp} implies that $c(G) = 2p^2$, $p^3+p^2$, or $2p^3$. 
From Table \ref{t:3}, $c(G_{(12,7)}) = \mu(G_{(12,7)}) = p^3+p^2$.
Consider
\begin{align*}
G := G_{(12, 14)} = \langle \alpha_{1}, \ldots, \alpha_{6} ~|~ & [\alpha_{3}, \alpha_{4}] = \alpha_{1}, [\alpha_{5}, \alpha_{6}] = \alpha_{2}, \alpha_{4}^{p} = \alpha_{1}, \alpha_{3}^{p} = \alpha_{5}^{p}= \alpha_{2},   \alpha_{1}^{p} =   \alpha_{2}^{p} =   \alpha_{6}^{p} =  1 \rangle.
\end{align*}
Observe $Z(G) = G' = \langle \alpha_{3}^{p}, \alpha_{4}^{p} \rangle 
\cong C_{p}\times C_{p}$. 
Lemma \ref{L1} implies that $|X_{G}|=2$.  
Let $X = \{ \chi_{1}, \chi_{2} \} \subset \Irr(G)$ 
satisfy Equation \eqref{eq:X_G}.
We claim that $\sum_{i=1}^{2}d(\chi_{i}) \geq p^2(p-1) + p(p-1)$.
From Remark \ref{remark:linearcharacter}, $X\cap \lin(G) = \emptyset$. 
Suppose $\chi_{i}(1) = p^2$ for at least one $i$ where $1\leq i \leq 2$. 
But $Z(G) \cong C_{p} \times C_{p}$, so 
$\sum_{i=1}^{2}d(\chi_{i}) \geq p^2\phi(p) + p\phi(p)$.
Suppose $\chi_{i}(1) = p$, for $i=1,2$. Then 
$G$ has a subgroup $H_{i}$ of index $p$ such that 
$\chi_{i} = \lambda_{i}\ind_{H_{i}}^{G}$ for some 
$\lambda_{i} \in \lin(H_{i})$. Since 
$H_{i}' \cap Z(G) \subseteq \ker(\chi_{i})$, 
\begin{equation} \label{eq:phi12}
\bigcap_{i=1}^{2}(H_{i}'\cap Z(G)) = 1.
\end{equation}
Consider the following sets of maximal subgroups of $G$:
\begin{eqnarray*}
V & := & \{ 
 \langle 
\alpha_{5}\alpha_{6}^{j}, 
\alpha_{4}, 
\alpha_{3}, \alpha_{2}, \alpha_{1} \rangle 
\; |\;  0\leq j\leq p-1 \} \cup 
\{\langle 
\alpha_{6}, 
\alpha_{4}, 
\alpha_{3}, \alpha_{2}, \alpha_{1} \rangle \}
\\
W & := & 
\{ 
 \langle 
\alpha_{6}, \alpha_{5}, 
\alpha_{4} \alpha_{3}^{l}, 
\alpha_{2}, \alpha_{1} \rangle \; | \; 0\leq l\leq p-1 \} \cup \{
\langle \alpha_{6}, \alpha_{5}, \alpha_{3}, \alpha_{2}, \alpha_{1} \rangle
 \}. 
\end{eqnarray*}
If $M \in V $ then $M' = \langle \alpha_{1} \rangle$; 
if $M \in W $ then $M' = \langle \alpha_{2} \rangle$. 
In neither case is $Z(G)$ contained in $M'$.
If $M$ is any other maximal subgroup of $G$, then $M'$ contains $Z(G)$.

By Equation \eqref{eq:phi12}, we can choose $H_{1}$ from $V$
and $H_{2}$ from $W$ without loss of generality. 
Since $H_{1}' = \langle \alpha_{1} \rangle$, $\alpha_{4}^{p} = \alpha_{1}$ 
and $o(\alpha_{3}H_{1}') = p^2$, we deduce that 
$H_1 / H_{1}' \cong C_{p^2} \times C_{p} \times C_{p}$.
Then $\frac{H_{1}/H_{1}'}{\ker(\lambda_{1})/H_{1}'} \cong C_{p}$, 
or $C_{p^2}$. 
If the former holds, 
then $\alpha_{3}^{p} \in \ker(\lambda_{1})$. 
But $\alpha_{3}^{p} = \alpha_{2}$, so 
$Z(G) \subset \ker(\lambda_{1})$. 
Now $X$ does not 
satisfy Equation \eqref{eq:X_G}, since $\lambda_{1}\ind_{H_{1}}^{G} \in X$,
a contradiction. Therefore, 
$\frac{H_{1}/H_{1}'}{\ker(\lambda_{1})/H_{1}'} \cong C_{p^2}$, 
so $d(\lambda_{1}) = p^2(p-1)$.
Lemma \ref{lemma:abelian} implies that 
$H_{1}/H_{1}' = \langle \alpha_{3} H_{1}' \rangle \times 
\frac{\ker(\lambda_{1})}{H_{1}'}$. 
Since $G/H_{1} = \langle \alpha_{5} H_{1} \rangle\cong C_{p}$, 
$$\lambda_{1}\ind_{H_{1}}^{G}(\alpha_{3}) 
= \sum_{k=0}^{p-1} \lambda_{1}(\alpha_{5}^{-k} \alpha_{3} \alpha_{5}^{k}) 
= \sum_{k=0}^{p-1} \lambda_{1}(\alpha_{3}) = p \omega,$$ 
where $\omega$ is a primitive $p^2$-th root of unity. 
Hence $\mathbb{Q}(\lambda_{1}\ind_{H_{1}}^{G}) = \mathbb{Q}(\omega)$, 
so $d(\lambda_{1}\ind_{H_{1}}^{G}) = p^2(p-1)$.
Since $d(\lambda_{2}\ind_{H_{2}}^{G}) \geq p(p-1)$, 
$$\sum_{i=1}^{2} d(\lambda_{i}\ind_{H_{i}}^{G}) \geq p^2(p-1)+p(p-1).$$ 
This establishes the claim. 
Lemma \ref{L1} now implies that $c(G) \geq p^3+p^2$. 
Similar arguments show that 
$c(G) \geq p^3+p^2$ for the remaining cases.

From \cite{Arxiv}, observe that:
\begin{eqnarray*}
\Core_{G_{(12,i)}} (\langle \alpha_{4}, \alpha_{6} \rangle)
& \cap & \Core_{G_{(12,i)}} (\langle \alpha_{3}, \alpha_{5} \rangle)  \ =  
1 \text{\ for\ } 13 \leq i \leq 15; \\
 \Core_{G_{(12,16)}} (\langle \alpha_{3}, \alpha_{5} \rangle)
& \cap & \Core_{G_{(12,16)}} (\langle \alpha_{4}, \alpha_{6} \rangle)  =  1. 
\end{eqnarray*}
Therefore, $c(G_{(12,i)}) = \mu(G_{(12,i)}) \leq 2p^3$ for $13\leq i \leq 16$.

From Lemmas \ref{P2} and \ref{remark:newrange}, 
{we conclude that $c(G) = 2p^2$, $p^3+p^2$, or $2p^3$ for every $G\in \Phi_{12}$.}
\end{proof}

\begin{lemma}
\label{prop:phi13}
If $G$ is a group of order $p^6$ in $\Phi_{13}$,
then $c(G) = p^3+p^2$, or $2p^3$.
\end{lemma}
\begin{proof}
From \cite{Arxiv}, $G$ has the following form:
\begin{align*}
\langle \alpha_{1}, \ldots, \alpha_{6}  ~|~ [\alpha_{3}, & \, \alpha_{6}] = [\alpha_{4}, \alpha_{5}] = \alpha_{2},  [\alpha_{3}, \alpha_{5}] = \alpha_{1}, \alpha_{1}^{p} = \alpha_{2}^{p} = 1, \\
&  \text{ some other relations between the generators} \rangle.
\end{align*}
Note that $\cd(G) = \{ 1, p, p^2 \}$, $\exp(G) = p^r$
where $r\leq 2$, $\exp(G/Z(G)) = p$, and 
$Z(G) = G' = \langle \beta_{1}, \beta_{2} \rangle \cong C_{p} \times C_{p}$
 (see \cite{RJ, Arxiv}). 
Lemma \ref{L1} implies that $|X_{G}|=2$. 
From Remark \ref{remark:linearcharacter}, $X_{G}\cap \lin(G) = \emptyset$.  
Let $X = \{ \chi_{1}, \chi_{2} \} \subset \nl(G)$ 
satisfy Equation \eqref{eq:X_G}.
Suppose $\chi_{i}(1) = p$ for $i=1,2$. Then 
$G$ has a subgroup $H_{i}$ of index $p$ such that
$\chi_{i}= \lambda_{i}\ind_{H_{i}}^{G}$ for some 
$\lambda_{i}\in \lin(H_{i})$. From the structure of $G$, 
observe that if $\SS$ has index $p$ in $G$, 
then $\alpha_{2} \in \SS'$. Consequently, 
$\bigcap_{i=1}^{2} \ker(\lambda_{i}\ind_{H_{i}}^{G})\neq 1$, 
a contradiction. Thus $\chi_{i}(1) = p$ for at most 
one $i$ where $1\leq i \leq 2$. 
But $Z(G) \cong C_{p} \times C_{p}$, 
so $d(\chi_{i}) \geq \chi_{i}(1)\phi(p)$ for $i=1,2$. Therefore:
\begin{enumerate}
\item if $\chi_{1}(1) = p$ and $\chi_{2}(1) = p^2$, then 
$\sum_{i=1}^{2}d(\chi_{i}) \geq p(p-1) + p^2(p-1)$;
\item if $\chi_{i}(1) = p^2$ for $i= 1, 2$, 
then $\sum_{i=1}^{2}d(\chi_{i}) \geq 2p^2(p-1)$.
\end{enumerate}
Hence $\sum_{i=1}^{2}d(\chi_{i}) \geq p(p-1) + p^2(p-1)$.
Lemma \ref{L1} implies that $c(G)\geq p^3+p^2$.

Let $G\in 
\Phi_{13}
\setminus \{ G_{(13, 9)}, 
G_{(13, 10)}, G_{(13, 11)} \}$. 
From \cite{Arxiv}, observe that 
$G$ has an elementary abelian normal subgroup of order $p^4$. 
From Lemma \ref{prop:CpXCp}, {$c(G) = p^3+p^2$, or $2p^3$.} 

{From \cite{Arxiv}, observe that: 
\begin{eqnarray*}
	\Core_{G_{(13,9)}} (\langle \alpha_{4}, \alpha_{5}, \alpha_{6} \rangle)
	& \cap & \Core_{G_{(13,9)}} (\langle \alpha_{3}, \alpha_{4} \rangle)  \ =  
	1; \\
	\Core_{G_{(13,10)}} (\langle \alpha_{4}, \alpha_{5} \rangle)
	& \cap & \Core_{G_{(13,10)}} (\langle \alpha_{4}, \alpha_{6} \rangle)  =  1; \\
		\Core_{G_{(13,11)}} (\langle \alpha_{3}, \alpha_{4} \rangle)
	& \cap & \Core_{G_{(13,10)}} (\langle \alpha_{5}, \alpha_{6} \rangle)  =  1. 
\end{eqnarray*}
Therefore, $c(G_{(13,i)}) = \mu(G_{(13,i)}) \leq 2p^3$ for $9\leq i \leq 11$.}

From Lemmas \ref{P2} and \ref{remark:newrange}, 
$c(G) = p^3+p^2$, or $2p^3$ for every $G\in \Phi_{13}$. 
\end{proof}

\begin{lemma}
\label{prop:phi14}
If $G$ is a group of order $p^6$ in $\Phi_{14}$, then 
$c(G) = p^4$.
\end{lemma}
\begin{proof}
Note that $G$ is nilpotent of class 2 and $Z(G) \cong C_{p^2}$ (see \cite{RJ, Arxiv}). 
From \cite[Theorem 4.12]{HB}, $c(G) = p^4$. 
\end{proof}
						
\begin{lemma}
\label{prop:phi17}
If $G$ is a group of order $p^6$ in $\Phi_{17}$,
then $c(G) = 2p^2$, $p^3+p^2$, or $2p^3$.
\end{lemma}
\begin{proof}
From \cite{Arxiv}, $G$ has the following form:
\begin{align*}
\langle \alpha_{1}, \ldots, \alpha_{6} ~|~ [\alpha_{5}, \alpha_{6}] &= \alpha_{3}, [\alpha_{4}, \alpha_{5}] = \alpha_{2},  [\alpha_{3}, \alpha_{6}] =  \alpha_{1}, \alpha_{1}^{p} = \alpha_{2}^{p} = \alpha_{3}^{p}  =  1 \rangle\\
&	 \text{ some other relations between the generators} \rangle.
\end{align*}
Note that 
$\cd(G) = \{ 1, p, p^2 \}$, $\exp(G) = p^r$ where $r\leq 2$, 
$Z(G) = \langle \alpha_{1}, \alpha_{2} \rangle \cong C_{p} \times C_{p}$, 
and $G' = \langle \alpha_{3}, \alpha_{1}, \alpha_{2} \rangle \cong C_{p}^{3}$
(see \cite{RJ, Arxiv}). 


Let $G\in \Phi_{17} 
 \setminus \{ G_{(17,23)}, G_{(17,24r)},
 G_{(17,25r)}, G_{(17,26r)}, G_{(17,27r)}, G_{(17,28r)} \}$.
From \cite{Arxiv}, observe that $G$ has 
an elementary abelian normal subgroup of order $p^4$. 
From Lemma \ref{prop:CpXCp}, $c(G) = 2p^2$, $p^3+p^2$, or $2p^3$. 
\noindent From Table \ref{t:3}, {$c(G_{(17,23)}) = p^3+p^2$ 
and $c(G_{(17,jr)}) = p^3+p^2$ for $j=24,25,27,28$.} 
Consider
\begin{align*}
G_{(17, 26r)} = \langle \alpha_{1}, \ldots, \alpha_{6} ~|~ &[\alpha_{5}, \alpha_{6}] = \alpha_{3}, [\alpha_{4}, \alpha_{5}] = \alpha_{2},  [\alpha_{3}, \alpha_{6}] =  \alpha_{1}, \alpha_{4}^{p} = \alpha_{1}^{r}, \\
& \alpha_{5}^{p}= \alpha_{6}^{p} = \alpha_{2}, \alpha_{1}^{p} = \alpha_{2}^{p} = \alpha_{3}^{p} = 1 \rangle.
\end{align*}
Observe ${\Core_{G_{(17, 26r)}} (\langle \alpha_{5}, \alpha_{3} \rangle) \cap 
\Core_{G_{(17, 26r)}} (\langle \alpha_{4}, \alpha_{3} \rangle) = 1}$,
so {$c(G_{(17, 26r)}) = \mu(G_{(17, 26r)}) \leq 2p^3$.} 
From Lemmas \ref{P2} and  \ref{remark:newrange}, 
$c(G) = 2p^2$, $p^3+p^2$, or $2p^3$ 
for all $G\in \Phi_{17}$.
\end{proof}
						
\begin{lemma}
\label{prop:phi18}
If $G$ is a group of order $p^6$ in $\Phi_{18}$,
then $c(G) = p^3+p^2$, or $2p^3$.
\end{lemma}
\begin{proof}
From \cite{Arxiv}, $G$ has the following form:
\begin{align*}
\langle \alpha_{1}, \ldots, \alpha_{6} ~|~ & [\alpha_{5}, \alpha_{6}] = \alpha_{3}, [\alpha_{4}, \alpha_{6}] = \alpha_{2},  [\alpha_{3}, \alpha_{6}] = [\alpha_{4}, \alpha_{5}] =  \alpha_{1}, \alpha_{1}^{p} = \alpha_{2}^{p} = \alpha_{3}^{p} = 1, \\
&\text{ some other relations between the generators} \rangle.
\end{align*}
Note that $\cd(G) = \{ 1, p, p^2 \}$, $\exp(G) = p^r$ where 
$r\leq 2$, $Z(G) = \langle \alpha_{1}, \alpha_{2} \rangle 
\cong C_{p} \times C_{p}$, and 
$G' = \langle \alpha_{3}, \alpha_{1}, \alpha_{2} \rangle 
\cong C_{p}^{3}$ (see \cite{RJ, Arxiv}). 
The structure of $G$ shows that 
if it has a subgroup $\SS$ of index $p$, then $\alpha_{1} \in \SS'$.
Lemma \ref{L1} implies that $|X_{G}|=2$, and from 
Remark \ref{remark:linearcharacter} $X_{G}\cap \lin(G) = \emptyset$. 
Let $X = \{ \chi_{1}, \chi_{2} \} \subset \nl(G)$ satisfy 
Equation \eqref{eq:X_G}. Using similar arguments as in the 
proof of Lemma \ref{prop:phi13}, we deduce that $c(G)\geq p^3+p^2$.
Let $$G\in \Phi_{18}
\setminus \{G_{(18,12r)} ~ (\text{} 
r=1,\omega,\omega^2 \text{ when } p\equiv 1\bmod 3, \text{ and } 
r=1 \text{ otherwise}),
G_{(18,13r)} ~(r= 1, \ldots, p-1)\}.$$
From \cite{Arxiv}, observe that $G$ has 
an elementary abelian normal subgroup of order $p^4$. 
From Lemma \ref{prop:CpXCp}, $c(G) = p^3+p^2$, or $2p^3$. 
From Table \ref{t:3}, {$c(G_{(18, 13r)}) = p^3+p^2$.}
Consider
\begin{align*}
{G_{(18,12r)}} = \langle \alpha_{1}, \ldots, \alpha_{6} ~|~ &[\alpha_{5}, \alpha_{6}] = \alpha_{3}, [\alpha_{4}, \alpha_{6}] = \alpha_{2},  [\alpha_{3}, \alpha_{6}] = [\alpha_{4}, \alpha_{5}] =  \alpha_{1}, \alpha_{4}^{p} = \alpha_{1}, \\
& \alpha_{5}^{p}= \alpha_{2}^{r},  \alpha_{1}^{p} = \alpha_{2}^{p} = \alpha_{3}^{p} = \alpha_{6}^{p} =  1 \rangle.
\end{align*}
Observe ${\Core_{G_{(18,12r)}} (\langle \alpha_{6}, \alpha_{3} \rangle) \cap 
\Core_{G_{(18,12r)}} (\langle \alpha_{5}, \alpha_{3} \rangle) = 1}$,
so {$c(G_{(18,12r)}) = \mu(G_{(18,12r)}) \leq 2p^3$}. 
From Lemmas \ref{P2} and \ref{remark:newrange}, 
$c(G) = p^3+p^2$, or $2p^3$ for all $G\in \Phi_{18}$. 
\end{proof}
					
\begin{lemma}
\label{prop:phi19}
If $G$ is a group of order $p^6$ in $\Phi_{19}$, 
then $c(G) = 2p^2$, $p^3+p^2$, or $2p^3$.
\end{lemma}
\begin{proof}
From \cite{Arxiv}, $G$ has the following form:
\begin{align*}
\langle  \alpha, \alpha_{1}, \alpha_{2}, \beta, \beta_{1}, \beta_{2} &~|~ [\alpha_{1}, \alpha_{2}] = \beta, [ \beta, \alpha_{1}] = \beta_{1}, [\beta, \alpha_{2}] = \beta_{2}, [\alpha, \alpha_{1}] = \beta_{1},  \\
& \beta^{p} = \beta_{1}^{p} = \beta_{2}^{p} = 1, \text{ some other relations between the generators} \rangle.
\end{align*}
Note that 
$\cd(G) = \{ 1, p, p^2 \}$, $\exp(G) = p^r$ where $r\leq 2$, 
$\exp(G/Z(G)) = p$, $Z(G) = \langle\beta_{1}, \beta_{2} \rangle 
\cong C_{p} \times C_{p}$, and 
$G' = \langle \beta, \beta_{1}, \beta_{2} \rangle 
\cong C_{p}^{3}$ (see \cite{RJ, Arxiv}). 
Since $\exp(G/Z(G)) = p$, 
the abelian subgroup 
$A = \langle \alpha,  \beta, \beta_{1}, \beta_{2} \rangle$ 
has order $p^4$. 
From Lemma \ref{thm:cdG1}, $A$ is a maximal abelian subgroup 
of $G$ containing $G'$. 

If $o(\alpha) = p$ then $A$ is elementary abelian. 
Lemma \ref{prop:CpXCp} now implies that $c(G) = 2p^2$, $p^3+p^2$, or $2p^3$.

Otherwise $o(\alpha) = p^2$.
From \cite{Arxiv}, the relevant groups are:
\begin{eqnarray*}
 &\phi_{19}(2211)h_{r} & \mbox{\rm\ for } r=1, \ldots, p-1 \\
 &\phi_{19}(2211)i & \\ 
 &\phi_{19}(2211)j_{r} & \; \mbox{\rm for\ } r=1, \ldots, (p-1)/2 \\
 &\phi_{19}(2211)k_{r,s} & {\rm\ for\ } r=1, \ldots, p-1 {\rm\ and\ }  
s=0,1,\ldots,(p-1)/2 \mbox{\rm\ where\ } s-r \mbox{\rm\ and\ } 2r-s 
\mbox{\rm\ are not divisible by\ } p \\
&\phi_{19}(2211)l_{r} & \;  \mbox{\rm for\ } r=1, \ldots, p-1 \\
&\phi_{19}(2211)m_{r} & \; \mbox{\rm for\ } r=1 \mbox{\rm\ or\ } \nu  \\
&\phi_{19}(21^{4})g & \\
&\phi_{19}(21^{4})h. & 
\end{eqnarray*}
Let {$G \in \{ \phi_{19}(2211)h_{r},\phi_{19}(2211)i, 
\phi_{19}(2211)m_{r}, \phi_{19}(21^{4})g  \}$.} 
Observe $\alpha^{p} = \beta_{1}$ (see \cite{Arxiv}).
Lemma \ref{L1} implies that $|X_{G}|=2$, and, from 
Remark \ref{remark:linearcharacter}, $X_{G} \cap \lin(G) = \emptyset$. 
Let $X = \{ \chi_{1}, \chi_{2} \} \subset \nl(G)$ 
satisfy Equation \eqref{eq:X_G}. 
Let $\SS$ be a subgroup of index $p$ in $G$. From the structure of $G$, 
observe that $Z(G) \subseteq \SS'$, except when 
$\SS$ is 
$H_{1} = \langle \alpha, \alpha_{1}, \beta, \beta_{1}, \beta_{2} \rangle$, 
or $H_{2} = \langle \alpha, \alpha_{2}, \beta, \beta_{1}, \beta_{2} \rangle$. 
Here $H_{1}' = \langle \beta_{1} \rangle$, and 
$H_{2}' = \langle \beta_{2} \rangle$. Hence 
$o(\alpha H_{2}') = p^2$ in $H_{2}/ H_{2}'$, 
so $\exp(H_{2}/ H_{2}') = p^2$. 
Proceeding as in the proof 
of Lemma \ref{prop:phi12}, we deduce that 
$\sum_{i=1}^{2}d(\chi_{i}) \geq p^2(p-1) + p(p-1)$. 
Lemma \ref{L1} implies that $c(G) \geq p^3+p^2$. 
Now $\Core_{G} (\langle \alpha, \beta \rangle)
 \subseteq \langle \alpha \rangle$ 
since $[\beta, \alpha_{i}] = \beta_{i}$ for $i=1,2$.  
Hence $\Core_{G} (\langle \alpha, \beta \rangle) \cap 
\Core_{G}\langle  (\alpha_{2}, \beta, \beta_{2} \rangle) = 1$. 
But $|\langle \alpha_{2}, \beta, \beta_{2} \rangle| \geq p^3$, 
so $c(G) = \mu(G) \leq 2p^3$. Therefore, $c(G) = p^3+p^2$, or $2p^3$.

Let {$G \in \{ \phi_{19}(2211)j_{r}, \phi_{19}(2211)k_{r,s},
 \phi_{19}(2211)l_{r}, \phi_{19}(21^{4})h \}$.} 
Observe $\alpha^{p} = \beta_{1}\beta_{2}$ (see \cite{Arxiv}).
By similar arguments to above, we deduce that $c(G) = p^3+p^2$, or $2p^3$. 

From Lemmas \ref{P2} and \ref{remark:newrange}, 
$c(G) = 2p^2$, $p^3+p^2$, or $2p^3$ for all $G\in \Phi_{19}$. 
\end{proof}

\begin{lemma}
\label{prop:phi20}
If $G$ is a group  of order $p^6$ in $\Phi_{20}$, 
then $c(G) = p^3+p^2$, or $2p^3$.
\end{lemma}
\begin{proof}
From \cite{Arxiv}, $G$ has the following form:
\begin{align*}
\langle \alpha_{1}, \ldots, \alpha_{6} ~|~ &[\alpha_{5}, \alpha_{6}] = \alpha_{3}, [\alpha_{4}, \alpha_{6}] = \alpha_{1}^{-1},  [\alpha_{3}, \alpha_{6}] = \alpha_{2}, [\alpha_{3}, \alpha_{5}] =  \alpha_{1},  \\
& \alpha_{1}^{p} = \alpha_{2}^{p} = \alpha_{3}^{p} = 1, \text{ some other relations between the generators} \rangle.
\end{align*}
Note that $\cd(G) = \{ 1, p, p^2 \}$, $\exp(G) = p^r$
where $r\leq 2$, $Z(G) = \langle \alpha_{1}, \alpha_{2} \rangle 
\cong C_{p} \times C_{p}$, and 
$G' = \langle \alpha_{1}, \alpha_{2}, \alpha_{3} \rangle \cong C_{p}^{3}$ 
(see \cite{RJ, Arxiv}).
From the structure of $G$, observe 
that if $\SS$ is a subgroup of index $p$ in $G$, then 
$\alpha_{1} \in \SS'$.
Lemma \ref{L1} implies that $|X_{G}|=2$. From 
Remark \ref{remark:linearcharacter}, $X_{G}\cap \lin(G) = \emptyset$. 
Let $X = \{ \chi_{1}, \chi_{2} \} \subset \nl(G)$ 
satisfy Equation \eqref{eq:X_G}. By arguments 
as in the proof of Lemma \ref{prop:phi13}, we deduce that $c(G)\geq p^3+p^2$.


Let $G\in \Phi_{20} 
\setminus \{ G_{(20, 7)}, G_{(20, 8)}, G_{(20, 14)},
  G_{(20, 15r)}, G_{(20, 16r)}, G_{(20, 17r)}, G_{(20, 18r)} \}$.
From \cite{Arxiv}, observe that $G$ has an 
elementary abelian normal subgroup of order $p^4$. From 
Lemma \ref{prop:CpXCp}, {$c(G) = p^3+p^2$, or $2p^3$.} 
From Table \ref{t:3}, $c(G_{(20, 14)}) = p^3+p^2$. 
Let {$G \in \{ G_{(20, 8)}, G_{(20, 15r)}, G_{(20, 17r)} \}$.} 
From \cite{Arxiv}, 
$\alpha_{4}^{p} = \alpha_{2}$, and either $o(\alpha_{5}) = p$, 
or $\langle \alpha_{5}^{p} \rangle = \langle \alpha_{1} \rangle$.
Hence $\Core_{G} (\langle \alpha_{4}, \alpha_{3} \rangle) \cap 
\Core_{G} (\langle \alpha_{5}, \alpha_{3}, \alpha_{1} \rangle) = 1$.
Thus $c(G) = \mu(G) \leq 2p^3$.  
Let {$G \in \{ G_{(20, 7)} , G_{(20, 16r)} \}$.}
 Observe  $\alpha_{4}^{p} = \alpha_{1}$, and $o(\alpha_{6}) = p$ in $G$.
  Also $\Core_{G} (\langle \alpha_{4}, \alpha_{3} \rangle) \cap 
\Core_{G} (\langle \alpha_{6}, \alpha_{3}, \alpha_{2} \rangle) = 1$. 
Thus $c(G) = \mu(G) \leq 2p^3$. 
Finally, {$\Core_{G_{(20, 18r)}}
(\langle \alpha_{4}, \alpha_{3} \rangle) \cap 
\Core_{G_{(20, 18r)}} 
(\langle \alpha_{5}, \alpha_{3}, \alpha_{1} \rangle) = 1$, 
so $c(G_{(20, 18r)}) = \mu(G_{(20, 18r)}) \leq 2p^3$.} 

From Lemmas \ref{P2} and \ref{remark:newrange}, 
$c(G) = p^3+p^2$, or $2p^3$ for all $G\in \Phi_{20}$. 
\end{proof}
						
\begin{lemma}
\label{prop:phi21}
If $G$ is a group of order $p^6$ in $\Phi_{21}$, 
then $c(G) = 2p^3$.
\end{lemma}
\begin{proof}
From \cite{Arxiv}, $G$ has the following form:
\begin{align*}
\langle \alpha_{1}, \ldots, \alpha_{6} ~|~ & [\alpha_{5}, \alpha_{6}] = \alpha_{3}, [\alpha_{4}, \alpha_{6}] = \alpha_{1}^{\nu},  [\alpha_{3}, \alpha_{6}] = [\alpha_{4}, \alpha_{5}] = \alpha_{2}, [\alpha_{3}, \alpha_{5}] =  \alpha_{1},\\
& \alpha_{1}^{p} = \alpha_{2}^{p} = \alpha_{3}^{p} = 1, \text{ some other relations between the generators} \rangle.
\end{align*}
Note that 
$\cd(G) = \{ 1, p, p^2 \}$, $\exp(G) = p^r$ where $r\leq 2$, 
$\exp(G/Z(G)) = p$, 
$Z(G) = \langle \alpha_{1}, \alpha_{2} \rangle \cong C_{p} \times C_{p}$, 
and $G' = \langle \alpha_{3}, \alpha_{1}, \alpha_{2} \rangle 
\cong C_{p}^{3}$ (see \cite{RJ, Arxiv}).
From the structure of $G$, observe that 
if $\SS$ is a subgroup of index $p$ in $G$, then
$Z(G) \subseteq \SS'$.
Lemma \ref{L1} implies that $|X_{G}|=2$. From 
Remark \ref{remark:linearcharacter}, $X_{G}\cap \lin(G) = \emptyset$.
Let $X = \{ \chi_{1}, \chi_{2} \} \subset \nl(G)$ 
satisfy Equation \eqref{eq:X_G}. 
Suppose $\chi_{i}(1) = p$ for some fixed $i$ where $1\leq i \leq 2$. 
Then $G$ has a subgroup $H_{i}$ of index $p$ such that 
$\chi_{i}= \lambda_{i}\ind_{H_{i}}^{G}$
for some $\lambda_{i}\in \lin(H_{i})$. 
Since $Z(G) \subseteq H_{i}'$, 
$\bigcap_{i=1}^{2} \ker(\lambda_{i}\ind_{H_{i}}^{G})\neq 1$, 
a contradiction. Therefore, $\chi_{i}(1) = p^2$ for $i=1,2$. 
Thus $d(\chi_{i}) \geq p^2 \phi(p) = p^2(p-1)$, so 
$\sum_{i=1}^{2}d(\chi_{i}) \geq 2p^2(p-1)$.
Lemma \ref{L1} implies that $c(G) \geq 2p^3$.

Let $G \in \Phi_{21}
\setminus \{ G_{(21, 2)}, 
G_{(21,6r)}\, (r= 1, \ldots, p-1), 
G_{(21, 7rs)}\, (r= 0, \ldots, p-1 \text{ and } 
s= 1, \ldots, (p-1)/2)\}.$
From \cite{Arxiv}, we deduce that 
$G$ has an elementary abelian normal subgroup of order $p^4$. 
Lemma \ref{prop:CpXCp} implies that $c(G) = 2p^3$. 
From \cite{Arxiv}, observe that: 
\begin{enumerate}
\item ${\Core_{G_{(21, 2)}} (\langle \alpha_{4}, \alpha_{3} \rangle) \cap 
\Core_{G_{(21, 2)}} (\langle \alpha_{6}, \alpha_{3}, \alpha_{2} \rangle) = 1 
\Rightarrow \mu(G_{(21, 2)}) \leq 2p^3}$;
\item ${\Core_{G_{(21,6r)}} (\langle \alpha_{4}, \alpha_{3} \rangle ) 
\cap \Core_{G_{(21,6r)}} (\langle \alpha_{5}, \alpha_{3}, \alpha_{1} \rangle) = 1 
\Rightarrow \mu(G_{(21,6r)}) \leq 2p^3}$;
\item {$\Core_{G_{(21, 7rs)}}\langle \alpha_{4}, \alpha_{3} \rangle \cap 
\Core_{G_{(21, 7rs)}}\langle \alpha_{5}, \alpha_{3}, \alpha_{1} \rangle = 1 
\Rightarrow \mu(G_{(21, 7rs)}) \leq 2p^3$.}
\end{enumerate}
Hence,  if $G \in \{ G_{(21, 2)}, G_{(21,6r)}, G_{(21, 7rs)} \}$, then 
$c(G) = \mu(G) = 2p^3$.

Therefore, $c(G) = 2p^3$ for all $G\in \Phi_{21}$. 
\end{proof}
						
\begin{lemma}
\label{prop:phi23}
If $G$ is a group of order $p^6$ in $\Phi_{23}$, 
then $c(G) = 2p^2$, $p^3+p^2$, or $2p^3$.
\end{lemma}
\begin{proof}
If $G\in \Phi_{23}$, then  
$\cd(G) = \{ 1, p, p^2 \}$, $\exp(G) = p^r$ where $r\leq 2$, 
$Z(G) \cong C_{p} \times C_{p}$, 
and $G' \cong C_{p}^{4}$ (see \cite{RJ, Arxiv}). 
From Lemma \ref{thm:cdG1}, $G'$ is a maximal abelian subgroup of $G$. 
From Lemma \ref{prop:CpXCp}, we conclude that
 $c(G) = 2p^2$, $p^3+p^2$, or $2p^3$. 
\end{proof}

We summarize the outcome of this section.
\begin{theorem} \label{thm:p^2}
If $G$ is a group of order $p^{6}$ with a center of order $p^2$,
then $c(G) =p^4$, $2p^3$, $p^3 +p^2$, $p^3+p$, $p^3$, $2p^2$ or $p^2+p$.
\end{theorem}

\subsection{Groups with center of order $p$} \label{SS5}
We now consider those groups of order 
$p^6$ which have center of order $p$. These 
groups are in 
$\Phi_{i}$ where $i \in I := \{ 22, 24, 25, \ldots, 43 \}$. 
If $G$ is  in $\Phi_{35}$,  then 
$\cd(G) = \{ 1, p \}$, otherwise 
$\cd(G) = \{ 1, p, p^2 \}$ (see \cite{RJ}). 
Moreover, all groups in $\Phi_{i}$ where 
$i \in I \setminus \{ 37, 39 \}$ are metabelian, 
since their commutator subgroups are abelian (see \cite{RJ}).

\begin{theorem} \label{T6}
If $G$ is a group of order $p^{6}$ with 
a center of order $p$, then $c(G) = p^2$, $p^3$, or $p^4$.
\end{theorem}
\begin{proof}
From \cite{Arxiv}, $\exp(G) = p^r$ where $r\leq 3$, 
and $Z(G) \subset G'$. 

Assume ${\exp(G) = p^{r}}$ where $r\leq 2$.
If $\cd(G) = \{ 1, p\}$, then, from Lemma \ref{Z(G)isCyclic}, 
$c(G) = p^2$, or $p^3$. If $\cd(G) = \{ 1, p, p^2\}$, then,
from Lemma \ref{Z(G)isCyclic}, $c(G) = p^3$, or $p^4$. 

The (remaining) groups of exponent $p^3$ are: 
$G_{(25, 3)}$; $G_{(26, 3)}$; $G_{(28, 2r)}$ and 
$G_{(29, 2r)}$ for $r=1, \ldots, p-2$; and 
$G_{(34, 1)}$ and 
$G_{(34, 2s)}$ for $s = 1, \nu$. 
If $G$ is one of these groups, then 
$\cd(G) = \{ 1, p, p^2 \}$. Let $\chi$ be a faithful irreducible 
character of $G$. From Remark \ref{remark:linearcharacter}, 
$\chi(1) = p^2$. But $G$ is metabelian, and so normally monomial
since $G' \cong C_{p^2} \times C_{p}$ (see \cite{RJ}).
Observe 
$A = \langle \alpha_{4}, \alpha_{3} \rangle$
is abelian, has order $p^4$, and exponent $p^2$,
and is a normal abelian
subgroup of maximal order among all abelian subgroups of
$G_{(25, 3)}$ and $G_{(26, 3)}$.
Lemma \ref{lemma:normallymonomial} now implies that
$c(G) = p^3$, or $p^4$ for each of these groups.
The same arguments applied to 
$\langle \alpha_{6}^{p}, \alpha_{3}, \alpha_{2} \rangle$
as a subgroup of $G_{(28,2r)}$ and $G_{(29,2r)}$, 
and 
{$\langle \alpha_{4}, \alpha_{3}, \alpha_{2} \rangle$ 
as a subgroup of $G_{(34,1)}$ and $G_{(34,2s)}$, 
resolve the remaining cases.}
\end{proof}

\subsection{Tables} \label{subsection:tables}
Tables \ref{t:VZ}--\ref{t:3} list $\mu(G)$ for some 
groups $G$ of order $p^6$. 
We present 
$c(G)$ for the remaining groups in Tables \ref{t:4}--\ref{t:7}. 

We partition the groups into two types. 
\begin{enumerate}
\item Type 1: a group that is a direct product of an abelian and 
a non-abelian subgroup; these are listed in  Table \ref{t:4}. 
\item Type 2: all others; these are listed in Tables \ref{t:5}--\ref{t:7}. 
\end{enumerate}

{The column labelled 
``Group $G$" identifies a group $G$ of order $p^6$ in all tables.
In Table \ref{t:4} the column labelled 
``Group $\NAB$" lists a {non-abelian direct factor} $\NAB$ of $G$. 
Suppose 
$$ X_\NAB =  \{ {\eta_{i}}_{\NAB/\NAB'} \}_{i=1}^{s} 
\cup \{ \psi_{j} \}_{j=1}^{t} \text{ where } s+t = d(Z(\NAB)),
\text{ and } \{ {\eta_{i}}_{\NAB/\NAB'} \}_{i=1}^{s}
\subset \lin(\NAB/\NAB'), $$
and $\psi_{j} = 
(\chi_{\SG_{j}})\ind_{\SG_{j}}^\NAB 
\text{ where } \SG_{j} \leq \NAB \text{ and } 
\chi_{\SG_{j}} \in \lin(\SG_{j})$ for $1\leq j \leq t$,
determines $c(\NAB)$.
The column labelled ``$X_\NAB$" lists 
the chosen $\SG_{j} \leq \NAB$, $\chi_{\SG_{j}} \in \lin(\SG_{j})$, and  
${\eta_{i}}_{\NAB/\NAB'}\in \lin(\NAB)$.
If $ \langle a \rangle$ is a subgroup of $\SG_{j} \leq \NAB$, 
then $\chi_{\langle a \rangle}$ 
and $1_{\langle a \rangle}$ 
denote a faithful and trivial 
character of $\langle a \rangle$ respectively.
(For clarity of exposition in the tables, 
we use $\SG, \SGP$ and $\SGM$ 
to denote subgroups of $\NAB$ rather than $\SG_{1}, \SG_{2}$ and $\SG_{3}$.) 
We use ${\eta}_{\NAB/\NAB'}$ when $\eta$ is a linear character of $\NAB$; 
otherwise, we use $\chi_{\SG}$ for some $\SG \leq \NAB$
where $\chi_\SG \in \lin(\SG)$ and $(\chi_{\SG})\ind_\SG^\NAB  \in \nl(\NAB)$.}
Finally, the values of $c(\NAB)$ and $c(G)$ are listed.
For groups of Type 2, these values are 
recorded for $G$, not $\NAB$. 
For all groups $G$ in Table \ref{t:4} and 
Table \ref{t:5}, $\cd(G) = \{ 1, p \}$;
in Tables \ref{t:6} and \ref{t:7}, the column labelled 
``$\cd(G)$" lists the set of character degrees of $G$. 
The remaining columns are as in Table \ref{t:4}.

\begin{example}
{\rm 
{Consider the group of Type 1
\[ G := G_{(3, 24r)} = 
\langle \alpha_{1}, \alpha_{2}, \alpha_{3}, 
\alpha_{4}, \beta_{1}, \beta_{2}, \beta_{3} \rangle = \langle \alpha_{1}, 
\alpha_{2}, \alpha_{3}, \alpha_{4}, \beta_{1}, \beta_{2} \rangle 
\times \langle \beta_{3} \rangle = \NAB \times K. \]
The column labelled ``$X_\NAB$" lists 
${\SG = \langle \alpha_{4}^{p}, \alpha_{3}, \alpha_{2} \rangle}$;
and 
${\eta_{\NAB/\NAB'} = \eta_{\langle \alpha_{4} \NAB' \rangle } 
	\cdot 1_{\langle \alpha_{3} \NAB' \rangle }} \in \lin(\NAB)$
where $\eta_{\langle \alpha_{4} \NAB' \rangle }$ is a faithful 
linear character of $\langle \alpha_{4} \NAB' \rangle $; 
and 
$ {\chi_{\SG} = 1_{\langle \alpha_{4}^{p} \rangle} \cdot 
	\chi_{\langle \alpha_{3} \rangle} \cdot 1_{\langle \alpha_{2} \rangle}} 
\in \lin(\SG)$
where $\chi_{\langle \alpha_{3} \rangle}$ is a faithful linear character 
of $\langle \alpha_{3} \rangle$. 
Since $c(\NAB)= p^3 + p^2$, from Lemma \ref{thm:nilpotent} and 
\ref{thm:pgroup}, $c(G) = p^3 + p^2 + p$.} 

{We now illustrate how to obtain $X_G$ from $X_\NAB$. 
Observe $G = \NAB \times K$ where $K = \langle \beta_{3} \rangle \cong C_p$,
and $d(Z(G)) = d(Z(\NAB)) + d(K) = 3$, 
so $|X_G| = 3$. Since $K$ is abelian, $G' = \NAB'$
and $\beta_{3}G' \neq G'$. Hence $G/G' = \langle \alpha_{4}G',
\alpha_{3} G', \beta_{3}G' \rangle$. 
Let $\eta_{G/G'} = \eta_{\langle \alpha_{4} G' \rangle } 
\cdot 1_{\langle \alpha_{3} G' \rangle } 
\cdot 1_{\langle \beta_{3} G' \rangle } \in \lin(G)$. Since $G' = \NAB'$,
observe that $\eta_{G/G'}$ is an extension of $\eta_{\NAB/\NAB'}$
to $G/G'$. Take $\overline{\SG} = \langle \alpha_{4}^{p},
\alpha_{3}, \alpha_{2}, \beta_{3} \rangle \geq \SG$ and
$ \chi_{\overline{\SG}} = 1_{\langle \alpha_{4}^{p} \rangle} \cdot 
\chi_{\langle \alpha_{3} \rangle} \cdot 1_{\langle \alpha_{2}
	\rangle} \cdot 1_{\langle \beta_{3} \rangle}
\in \lin(\overline{\SG})$. 
Here 
$\chi_{\overline{\SG}}$ is an extension of $\chi_{\SG}$ to $\overline{\SG}$.
Since $\inertiagroup_{G}(\chi_\SG) = \SG$, 
we deduce that 
$\inertiagroup_{G}(\chi_{\overline{\SG}}) = \overline{\SG}$. 
Hence $\chi_{\overline{\SG}}\ind_{\overline{\SG}}^G \in \nl(G)$. Lastly, 
consider 
$\lambda_{G/G'} = 1_{\langle \alpha_{4} G' \rangle } 
\cdot 1_{\langle \alpha_{3} G' \rangle } 	\cdot
\lambda_{\langle \beta_{3} G' \rangle } 
\in \lin(G)$. 
Observe that 
$d(\eta_{G/G'}) = d(\eta_{\NAB/\NAB'}) = \phi(p^2)$, 
$d(\chi_{\SG}\ind_{\SG}^N) = d(\chi_{\overline{\SG}}\ind_{\overline{\SG}}^G)
= p \phi(p^2)$, $d(\lambda_{G/G'}) = \phi(p)$ and 
$\ker(\eta_{G/G'}) \cap \ker(\lambda_{G/G'}) 
\cap \ker(\chi_{\overline{\SG}}\ind_{\overline{\SG}}^G) = 1$.
Suppose \[ \xi = \left[ \sum_{\sigma \in \Gamma(\eta_{G/G'})}
\left( \eta_{G/G'} \right)^{\sigma}  \right] + 
\left[ \sum_{\sigma \in \Gamma(\lambda_{G/G'})}
\left(\lambda_{G/G'} \right)^{\sigma}  \right] + 
\left[ \sum_{\sigma \in \Gamma(\chi_{\overline{\SG}}\ind_{\overline{\SG}}^G)}
\left(\chi_{\overline{\SG}}\ind_{\overline{\SG}}^G \right)^{\sigma}  \right]. \]
Hence $\xi(1) =  p\phi(p^2) + \phi(p^2) + \phi(p)$, 
and \cite[Lemma 4.5]{HB} implies that $m(\xi) = p^2 + p + 1$. 
Thus $\xi(1) + m(\xi) = p^3 + p^2 + p = c(G)$ 
and $\{ \eta_{G/G'}, \lambda_{G/G'}, 
\chi_{\overline{\SG}}\ind_{\overline{\SG}}^G \}$ is a 
minimal degree faithful quasi-permutation representation of $G$.}
}
\end{example}

\begin{example}\label{example:G4,28}
{\rm	
Consider the group of Type 2
	\begin{align*}
		G := G_{(4,28)} =  \langle \alpha_{1}, \ldots, \alpha_{5}, \beta_{1}, \beta_{2}, \beta_{3} ~|~ &[\alpha_{4}, \alpha_{5}] = \alpha_{2} = \beta_{2}, [\alpha_{3}, \alpha_{5}] = \alpha_{1} = \beta_{1}, \alpha_{4}^{p}= \beta_{1},\\
		 & \alpha_{5}^{p} = \beta_{3}, \alpha_{3}^{p} =  \beta_{1}^{p} = \beta_{2}^{p} = \beta_{3}^{p} = 1 \rangle.
	\end{align*} 
	The column labelled ``$X_G$" lists 
	 $\SG = \langle \alpha_{5}^{p}, \alpha_{3}, \alpha_{4}, \alpha_{2} \rangle $; and $\eta_{\small G/G'} = \eta_{\langle \alpha_{5} G' \rangle} \cdot 1_{\langle \alpha_{3} G' \rangle } \cdot  1_{\langle \alpha_{4}G' \rangle } \in \lin(G) $ where $\eta_{\langle \alpha_{5} G' \rangle}$ is a faithful 
	 linear character of $\langle \alpha_{5} G' \rangle$; and 
	 $\lambda_\SG =  1_{\langle \alpha_{5}^{p} \rangle } \cdot 1_{\langle \alpha_{3} \rangle} \cdot 1_{\langle \alpha_{4} \rangle } \cdot \lambda_{\langle \alpha_{2} \rangle} \in \lin(\SG) $ where $\lambda_{\langle \alpha_{2} \rangle}$ is a faithful 
	 linear character of $\langle \alpha_{2} \rangle$; and 
	   $\chi_\SG =  1_{\langle \alpha_{5}^{p} \rangle } \cdot  1_{\langle \alpha_{3} \rangle} \cdot \chi_{\langle \alpha_{4} \rangle} \cdot 1_{\langle \alpha_{2} \rangle } \in \lin(\SG)$ where $\chi_{\langle \alpha_{4} \rangle}$ is a faithful 
	   linear character of $\langle \alpha_{4} \rangle$. 
{Hence $\{ \eta_{\small G/G'}, {\lambda_\SG}\ind_{\SG}^{G}, {\chi_\SG}\ind_{\SG}^{G} \}$ is a minimal degree faithful quasi-permutation representation of $G$, and so $c(G) = p^3 + 2p^2$.}}
\end{example}

\begin{remark} \label{remark:QuasiPermRepAsPermRep}
\textnormal{Let $G$ be a $p$-group of odd order with 
$d(Z(G)) = m$. Suppose $c(G) = \sum_{i=1}^{m}p^{a_{i}}$
and $X_G = \{ \chi_{i} \}_{i=1}^{m}$ is a minimal degree faithful
quasi-permutation representation of $G$. From Lemma
\ref{thm:ford}, for $1\leq i \leq m$, there exists
$H_{i}\leq G$ with $|G/H_{i}| = \chi_{i}(1)$ such
that $\chi_{i} = \lambda_{i}\ind_{H_{i}}^{G}$, for
some $\lambda_{i} \in \lin(H_{i})$, and
$\mathbb{Q}(\chi_{i}) = \mathbb{Q}(\lambda_{i})$.
Let $\xi = \sum_{\chi \in X_G}
\left[ \sum_{\sigma \in \Gamma(\chi)}
\chi^{\sigma}  \right]$. 
Observe that 
\begin{align*}
\sum_{i=1}^{m}p^{a_{i}} =& c(G) = \xi(1) + m(\xi)
= \sum_{i=1}^{m} \chi_{i}(1) |H_i:\ker(\lambda_{i})| 
= \sum_{i=1}^{m} |G:H_{i}| |H_i:\ker(\lambda_{i})|
= \sum_{i=1}^{m} |G:\ker(\lambda_{i})|.
\end{align*}
Since $1 = \bigcap_{i=1}^{m} \ker(\chi_{i}) =
\bigcap_{i=1}^{m} \Core_{G}(\ker(\lambda_{i}))$,
and $\mu(G) = c(G) = \sum_{i=1}^{m}p^{a_{i}}$, we deduce 
that $\{ \ker(\lambda_{i}) \}_{i=1}^{m}$ is a minimal 
degree faithful permutation representation of $G$.	
Hence, for each group $G$ in Tables \ref{t:4}--\ref{t:7}, 
by identifying kernels of appropriate characters, 
we also obtain a minimal degree faithful permutation representation of $G$.
}
\end{remark}

\begin{tiny}
\begin{longtable}[c]{|l|l|l|l|l|}
\caption{$c(G)$ for groups of Type 1 \label{t:4}}\\
							
\hline
\multicolumn{5}{| c |}{Character degree set = $\{ 1, p \}$}\\
\hline
Group $G$ & Group $\NAB$ &  $X_\NAB$ &  $c(\NAB)$ &  $c(G)$ \\ 
\hline
\endfirsthead
							
\hline
\multicolumn{5}{|c|}{Continuation of Table \ref{t:4}}\\
\hline
Group $G$ & Group $\NAB$ &  $X_\NAB$ &  $c(\NAB)$ &  $c(G)$ \\ 
\hline
\endhead


							
\vtop{\hbox{\strut $G_{(3, 24r)}$, }
	\hbox{\strut for $r$ = 1 or $\nu$   }} & \vtop{\hbox{\strut $ \langle \alpha_{1}, \alpha_{2}, \alpha_{3}, \alpha_{4}, \beta_{1}, \beta_{2} \rangle$ }} &  \vtop{\hbox{\strut $\SG = \langle \alpha_{4}^{p}, \alpha_{3}, \alpha_{2}\rangle $}  \hbox{\strut $\eta_{\small \NAB/\NAB'} = \eta_{\langle \alpha_{4} \NAB' \rangle} \cdot 1_{\langle \alpha_{3}\NAB' \rangle }$ } \hbox{\strut $\chi_{\SG} =  1_{\langle \alpha_{4}^{p} \rangle } \cdot \chi_{\langle \alpha_{3} \rangle} \cdot 1_{\langle \alpha_{2} \rangle} $ }}  &  $p^{3}+p^{2}$ & $p^{3} + p^{2}+p$\\
\hline
\vtop{\hbox{\strut $G_{(4, 18)}$ }  }  & \vtop{\hbox{\strut $ \langle \alpha_{1}, \ldots, \alpha_{5}, \beta_{1}, \beta_{2} \rangle$ }} &  \vtop{ \hbox{\strut $\SG = \langle \alpha_{5}^{p}, \alpha_{3}, \alpha_{4}\rangle $} \hbox{\strut $\eta_\SG = \eta_{\langle \alpha_{5}^{p} \rangle} \cdot 1_{\langle \alpha_{3} \rangle } \cdot 1_{\langle \alpha_{4} \rangle } $ } \hbox{\strut $\chi_\SG =  1_{\langle \alpha_{5}^{p} \rangle } \cdot  1_{\langle \alpha_{3} \rangle} \cdot \chi_{\langle \alpha_{4} \rangle} $ }} & $p^{3}+p^{2}$ & $p^{3}+p^{2}+p$\\
\hline

\vtop{\hbox{\strut $G_{(4, 23)}$ }    }  & \vtop{\hbox{\strut $ \langle \alpha_{1}, \ldots, \alpha_{5}, \beta_{1}, \beta_{2} \rangle$ }} &
\vtop{ \hbox{\strut $\SG = \langle \alpha_{3}, \alpha_{4}\rangle $} \hbox{\strut $\eta_\SG =  1_{\langle \alpha_{3} \rangle } \cdot \eta_{\langle \alpha_{4} \rangle } $ } \hbox{\strut $\chi_\SG =  \chi_{\langle \alpha_{3} \rangle} \cdot 1_{\langle \alpha_{4} \rangle} $ }}
&
$2p^{3}$ & $2p^{3}+p$ \\
\hline

\end{longtable}
\end{tiny}

\begin{tiny}
						
\begin{longtable}[c]{|l|l|l|}
\caption{$c(G)$ for groups of Type 2 with center of order $p^{3}$\label{t:5}}\\

\hline
\multicolumn{3}{| c |}{Character degree set = $\{ 1, p \}$}\\
\hline
Group $G$ & $X_{G}$ &  $c(G)$\\
\hline
\endfirsthead
							
							\hline
							\multicolumn{3}{|c|}{Continuation of Table \ref{t:5}}\\
							\hline
							Group $G$ & $X_{G}$  & $c(G)$\\
							\hline
							\endhead
							
							

								\vtop{\hbox{\strut $G_{(3,10r)}$, } \hbox{\strut for $r = 1$ or $\nu$ } } &  \vtop{\hbox{\strut $\SG = \langle \alpha_{4}^{p}, \alpha_{3}, \alpha_{2} \rangle $ } \hbox{\strut $\eta_{\small G/G'} = \eta_{\langle \alpha_{4} G' \rangle} \cdot 1_{\langle \alpha_{3}G' \rangle }$ } \hbox{\strut $\chi_\SG =  1_{\langle \alpha_{4}^{p} \rangle } \cdot \chi_{\langle \alpha_{3} \rangle} \cdot 1_{\langle \alpha_{2} \rangle} $ }}  & $p^{4}+p^{2}$\\
						\hline
						\vtop{\hbox{\strut $G_{(3,16r)}$, } \hbox{\strut for $r$ = 1 or $\nu$ } } &  \vtop{\hbox{\strut $\SG = \langle \alpha_{4}^{p}, \alpha_{3}, \alpha_{2} \rangle $ } \hbox{\strut $\eta_{\small G/G'} = \eta_{\langle \alpha_{4} G' \rangle} \cdot 1_{\langle \alpha_{3}G' \rangle }$ } \hbox{\strut $\chi_\SG =  1_{\langle \alpha_{4}^{p} \rangle } \cdot \chi_{\langle \alpha_{3} \rangle} \cdot 1_{\langle \alpha_{2} \rangle} $ }}  & $2p^{3}$\\
						\hline
						$G_{(4, 3)}$ &  \vtop{\hbox{\strut $\SG = \langle \alpha_{4}, \alpha_{3}, \beta_{1} \rangle $ } \hbox{\strut $\eta_\SG =  1_{\langle \alpha_{4} \rangle } \cdot \eta_{\langle \alpha_{3} \rangle} \cdot 1_{\langle \beta_{1} \rangle} $ } \hbox{\strut $\chi_\SG =  1_{\langle \alpha_{4} \rangle } \cdot  1_{\langle \alpha_{3} \rangle} \cdot \chi_{\langle \beta_{1} \rangle} $ }}  & $2p^{3}$\\
						\hline
						$G_{(4,4)}$ &  \vtop{\hbox{\strut $\SG = \langle \alpha_{3}, \alpha_{4}, \alpha_{2} \rangle $ } \hbox{\strut $\eta_\SG =  1_{\langle \alpha_{3} \rangle } \cdot \eta_{\langle \alpha_{4} \rangle} \cdot 1_{\langle \alpha_{2} \rangle } $ } \hbox{\strut $\chi_\SG =  1_{\langle \alpha_{3} \rangle } \cdot  1_{\langle \alpha_{4} \rangle} \cdot \chi_{\langle \alpha_{2} \rangle} $ }}  & $p^{4} + p^{2}$\\
\hline
						\end{longtable}
					\end{tiny}
						
						
					\begin{tiny}
						
						\begin{longtable}[c]{|l|l|l|l|}
							\caption{$c(G)$ for groups of Type 2 with center of order $p^{2}$\label{t:6}}\\
							
							\hline
							Group $G$ & $\cd(G)$ & $X_G$ & $c(G)$\\
							\hline
							\endfirsthead
							
							\hline
							\multicolumn{4}{|c|}{Continuation of Table \ref{t:6}}\\
							\hline
							Group $G$ & $\cd(G)$ & $X_G$  & $c(G)$\\
							\hline
							\endhead
							
							\hline
							\endfoot
							
							\hline
							\endlastfoot

					\vtop{\hbox{\strut $G_{(7, 14)}$ } }  & $ \{ 1,p,p^2 \}$ &  \vtop{\hbox{\strut $\SG = \langle \alpha_{5}^{p}, \alpha_{4}, \alpha_{2}, \alpha_{3}^{p} \rangle $} \hbox{\strut $\eta_{G/G{}'} =\eta_{\langle \alpha_{5} G{}' \rangle} \cdot 1_{\langle \alpha_{4}G{}' \rangle} \cdot 1_{\langle \alpha_{3} G{}' \rangle } $ } \hbox{\strut $\chi_{\SG} = 1_{\langle \alpha_{5}^{p} \rangle } \cdot 1_{\langle \alpha_{4} \rangle} \cdot  1_{\langle \alpha_{2} \rangle} \cdot  \chi_{\langle \alpha_{3}^{p} \rangle} $ }} &  $p^{3}+p^{2}$\\
					\hline
					
					\vtop{\hbox{\strut $G_{(8,1)}$,  $G_{(8,3)}$ } }   & $ \{ 1,p,p^2 \}$ &  \vtop{\hbox{\strut $\SG = \langle  \alpha_{4}, \alpha_{2} \rangle $} \hbox{\strut $\chi_{\SG} = 1_{\langle \alpha_{4} \rangle} \cdot \chi_{\langle \alpha_{2} \rangle} $ }} &  $p^{4}$\\
				
					\hline
					\vtop{\hbox{\strut $G_{(8,7)}$ } }  & $ \{ 1,p,p^2 \}$ &  \vtop{\hbox{\strut $\SG = \langle  \alpha_{4}^{p^{2}}, \alpha_{5} \rangle $} \hbox{\strut $\eta_{G/G{}'} = 1_{\langle \alpha_{4}G{}' \rangle} \cdot \eta_{\langle \alpha_{5}G{}' \rangle} $ } \hbox{\strut $\chi_{\SG} = \chi_{\langle \alpha_{4}^{p^2} \rangle} \cdot   1_{\langle \alpha_{5} \rangle} $ }} &  $2p^{3}$\\
					\hline
					\vtop{\hbox{\strut $G_{(9,9r)}$, } \hbox{\strut  where $r=1,\omega,\omega^2$ when $p\equiv 1\bmod 3$, } \hbox{\strut and  $r=1$ when $p\equiv 2 \bmod 3$ } }  & $\{ 1,p \}$ &  \vtop{\hbox{\strut $\SG = \langle \alpha_{5}^{p}, \alpha_{4}, \alpha_{3}, \alpha_{2} \rangle $} \hbox{\strut $\eta_{G/G{}'} =\eta_{\langle \alpha_{5} G{}' \rangle} \cdot 1_{\langle \alpha_{4}G{}' \rangle} $ } \hbox{\strut $\chi_{\SG} = 1_{\langle \alpha_{5}^{p} \rangle} \cdot \chi_{\langle \alpha_{4} \rangle} \cdot   1_{\langle \alpha_{3} \rangle} \cdot 1_{\langle \alpha_{2} \rangle} $ }}  & $p^{3}+p^{2}$\\
\end{longtable}
					\end{tiny}

					\begin{tiny}
						
						\begin{longtable}[c]{|l|l|l|l|}
							\caption{$c(G)$ for groups of Type 2 with center of order $p$\label{t:7}}\\
							
							\hline
							Group $G$ & $\cd(G)$ & $X_G$ &  $c(G)$\\
							\hline
							\endfirsthead
							
							\hline
							\multicolumn{4}{|c|}{Continuation of Table \ref{t:7}}\\
							\hline
							Group $G$ & $\cd(G)$ & $X_G$ &  $c(G)$\\
							\hline
							\endhead
							
							\hline
							\endfoot
							
							\hline
							\endlastfoot

							\vtop{ \hbox{\strut $G_{(25, 1)}$, } \hbox{\strut $G_{(26,1)}$ }}  & $\{ 1,p,p^2 \}$ & \vtop{\hbox{\strut $\SG = \langle \alpha_{3}, \alpha_{2}, \alpha_{5} \rangle $ }  \hbox{\strut $\eta_{\SG} = \eta_{\langle \alpha_{3} \rangle} \cdot 1_{\langle \alpha_{2} \rangle} \cdot 1_{\langle \alpha_{5} \rangle}  $ }} & $p^{4}$\\
							\hline
							\vtop{ \hbox{\strut $G_{(25, 2r)}$, } \hbox{\strut $G_{(26,2r)}$, } \hbox{\strut for $r=1,2,\ldots,(p-1)/2$} }  & $\{ 1,p,p^2 \}$ & \vtop{\hbox{\strut $\SG = \langle \alpha_{3}, \alpha_{2}, \alpha_{5} \rangle $ }  \hbox{\strut $\eta_{\SG} = \eta_{\langle \alpha_{3} \rangle} \cdot 1_{\langle \alpha_{2} \rangle} \cdot 1_{\langle \alpha_{5} \rangle}  $ }} & $p^{4}$\\
							\hline
							\vtop{\hbox{\strut $G_{(25,3)}$, } \hbox{\strut $G_{(26,3)}$ } }  & $\{ 1,p,p^2 \}$ & \vtop{\hbox{\strut $\SG = \langle \alpha_{4}, \alpha_{3} \rangle $ }  \hbox{\strut $\eta_{\SG} = 1_{\langle \alpha_{4} \rangle} \cdot \eta_{\langle \alpha_{3} \rangle}  $ }}  & $p^{4}$\\
							\hline
							\vtop{\hbox{\strut $G_{(27,4r)}$, } \hbox{\strut for $r=1$ or $\nu$} }  & $\{ 1,p,p^2 \}$ & \vtop{\hbox{\strut $\SG = \langle \alpha_{3}, \alpha_{2}, \alpha_{4} \rangle $ }  \hbox{\strut $\eta_{\SG} = 1_{\langle \alpha_{3} \rangle} \cdot 1_{\langle \alpha_{2} \rangle} \cdot \eta_{\langle \alpha_{4} \rangle}  $ }} & $p^{4}$\\
							\hline
							\vtop{\hbox{\strut $G_{(28,1)}$, } \hbox{\strut $G_{(29,1)}$ } }  & $\{ 1,p,p^2 \}$ & \vtop{\hbox{\strut $\SG = \langle \alpha_{6}^{p}, \alpha_{4}^{p}, \alpha_{3} \rangle $ }  \hbox{\strut $\eta_{\SG} = 1_{\langle \alpha_{6}^{p} \rangle} \cdot 1_{\langle \alpha_{4}^{p} \rangle} \cdot \eta_{\langle \alpha_{3} \rangle}  $ }} &  $p^{4}$\\
\end{longtable}
\end{tiny}

\section{$\mu(G)$ for groups of order $3^{6}$} \label{section:3^6}
We used the standard function {\tt MinimalDegreePermutationRepresentation} 
in {\sf GAP} to compute $\mu(G)$ for each of the 504 groups $G$ of 
order $3^6$. We present the results for $\mu(G) = c(G)$ in Table \ref{t:729}.
The group identifiers are with respect to 
the {\sc SmallGroups} library \cite{SmallGroups}.

	\begin{footnotesize}
	
	\begin{longtable}[c]{|l|l|}
\caption{Value of $\mu(G)$ for the 504 groups of order $3^6$ \label{t:729}}\\
		
\hline
$\mu(G)$ & Group identifier\\
\hline
\endfirsthead
		
\hline
\multicolumn{2}{|c|}{Continuation of Table \ref{t:729}}\\
\hline
$\mu(G)$ & Group identifier\\
\hline
\endhead
		
\hline
\endfoot
	 
	 15  &   479 \\  \hline 
	 18  &   149, 150, 167, 168, 276, 290, 302, 303, 325, 339, 341, 344, 354, 358, 453, 455, 458, 498, 499, 504 
	 \\  \hline 
	 21  &   138, 139, 142, 245, 255, 256, 419, 420, 425, 426, 427, 430, 433, 438, 497 
	 \\  \hline 
	 24  &   415, 416, 417 
	 \\  \hline 
	 27  &   9, 10, 11, 12, 26, 34, 35, 36, 56, 57, 65, 66, 67, 68, 69, 76, 80, 102, 105, 106, 107, 108, 110, 114, \\
	 & 115, 116, 117, 
	 118, 119, 120, 121, 122, 123, 125, 126, 127, 131, 132, 133, 134, 135, 307, 309, 312, \\
	 & 314, 315, 321, 322, 388, 397, 410, 411, 
	 412, 490, 494 
	 \\  \hline 
	 30  &   103, 253, 268, 394, 395, 396, 402, 403, 404, 483, 484, 485, 486, 487, 488, 501, 502 
	 \\
\hline 
\end{longtable}

\end{footnotesize}

\section{Access to results}\label{section:access}
{
Much of the material recorded in Tables \ref{t:VZ}--\ref{t:7} 
is publicly available in {\sc Magma} 
via a GitHub repository \cite{github}. 
}
For each group $G$ of order $p^6$ we recorded its presentation from 
\cite{Arxiv}.  For a given prime $p$, we can now use the $p$-quotient 
algorithm to construct an explicit power-conjugate presentation 
for $G$; for details, see \cite[Chapter 9]{HEO}. 
Such a presentation is useful for extensive computations with $G$.

For those $G$ listed in Tables \ref{t:VZ}--\ref{t:3}, we 
recorded either ${\mathcal{H}}_\NAB$ or ${\mathcal{H}}_G$; so  
we can readily construct an explicit faithful permutation representation for 
$G$ of degree $\mu(G)$.

For those $G$ listed in Tables \ref{t:4}--\ref{t:7} 
we recorded details which allow us to reconstruct the required characters. 
We illustrate this by reference to Example \ref{example:G4,28}. 
Recall ${\lambda_\SG}\ind_{\SG}^{G} \in X_{G}$ for $G = G_{(4,28)}$.
We did not record the chosen character $\lambda_\SG$, 
but instead listed generating sets for both $\SG$ and $\ker(\lambda_\SG)$.
Using these, we can first compute a faithful linear character
of $\SG/\ker(\lambda_\SG)$, and then lift this to $\SG$ to recover $\lambda_R$; 
finally we compute ${\lambda_\SG}\ind_{\SG}^{G} \in X_{G}$. 
As discussed in Remark \ref{remark:QuasiPermRepAsPermRep},
we can use the stored descriptions of the subgroups and kernels 
of characters to construct an explicit faithful 
permutation representation for $G$ of degree~$\mu(G)$.

Using this data and related code, we verified readily that the 
description listed
{
for one of ${\mathcal{H}}_\NAB$, ${\mathcal{H}}_G$, $X_H$ or $X_G$ 
in Tables \ref{t:VZ}--\ref{t:7} determines a faithful permutation 
representation for each group $G$ of order $p^6$ for $5 \leq p \leq 97$.
}
We checked that the elements of $X_H$ or $X_G$ are irreducible 
for $p \leq 13$.  For this range we also verified the claimed value 
of $\mu(G)$ using the following:
\begin{itemize}
\item An improved version of the {\sf GAP} 
standard function {\tt MinimalFaithfulPermutationDegree}
developed by Alexander Hulpke.
\item An implementation in {\sc Magma} 
of the algorithm of Elias {\it et al.} \cite{Elias2010}
developed by Neil Saunders.
\end{itemize} 
The latter is included in the repository, as is code illustrating
the examples presented in the paper.

\section*{Acknowledgements}
{O'Brien was supported by the Marsden Fund of New Zealand grant 20-UOA-107.
Prajapati acknowledges the Science and Engineering Research 
Board, Government of India, for financial support through 
grant MTR/2019/000118. 
This work forms part of the Ph.D thesis of Ayush Udeep who acknowledges the University Grants Commission, Government of India 
(File: Nov2017-434175). We thank Alexander Hulpke and Neil Saunders 
for sharing their code with us.
}


\begin{thebibliography}{99}
							
\bibitem{AB} M.H. Abbaspour and H. Behravesh, A note on quasi-permutation representations of finite groups, Int. J. Contemp. Math. Sci., {\bf 4}(27) (2009), 1315-1320.	
							
\bibitem{BKP} G.K. Bakshi,  R.S. Kulkarni, and I.B.S. Passi, The rational group algebra of a finite group, J. Algebra Appl., {\bf 12}(03) (2013), 1250168.

\bibitem{BehraveshMetacyclic} H. Behravesh, Quasi-permutation representations 
of metacyclic $p$-groups with non-cyclic center, 
Southeast Asian Bull. Math., {\bf 24} (2000), 345-353.
							
\bibitem{HB} H. Behravesh, Quasi-permutation representations of $p$-groups of class 2, J. London Math. Soc., {\bf 55}(2) (1997), 251-260.
							
\bibitem{HB1997} H. Behravesh, The minimal degree of a faithful 
quasi-permutation representation of an abelian group, 
Glasgow Math. J., {\bf 39}(1) (1997), 51-57.
							
\bibitem{BD} H. Behravesh and M. Delfani, On faithful quasi-permutation 
representations of groups of order $p^{5}$, 
J. Algebra Appl., {\bf 17}(7) (2018), 1850127 (12 pages).
							
\bibitem{BG} H. Behravesh and G. Ghaffarzadeh, Minimal degree of faithful 
quasi-permutation representations of $p$-groups, Algebra Colloq., 
{\bf 18} (2011), 843--846.
							
\bibitem{BGgroup64} H. Behravesh and G. Ghaffarzadeh, Quasi-permutation representations of groups of order 64, Turkish J. Math., {\bf 31}(5) (2007), 1-6.									
							
\bibitem{SmallGroups} H.U.\ Besche, B.\ Eick and E.A. O'Brien, 
A millennium project: constructing small groups, 
Internat. J. Algebra Comput., {\bf 12} (2002), 623--644.
							
\bibitem{MAGMA} W. Bosma, J. Cannon and C. Playoust, 
The Magma Algebra System I: The User Language, J. Symb. Comput., {\bf 24} (1997), 235-265.
							
\bibitem{BGHS} J.M. Burns, B. Goldsmith, B. Hartley and R. Sandling, 
On quasi-permutation representations of finite groups, 
Glasgow Math. J., {\bf 36}(3) (1994),  301-308.
							
\bibitem{Britnell2017} J.R. Britnell, N. Saunders and T. Skyner, 
On exceptional groups of order $p^5$, J. Pure Appl. Algebra, {\bf 221} (11) (2017), 2647-2665.

\bibitem{EasdownSaunders2016} D. Easdown and N. Saunders, The minimal faithful permutation degree for a direct product obeying an inequality condition, Comm. Algebra, {\bf 44} (8) (2016), 3518-3537.
							
\bibitem{Easterfield} T.E. Easterfield, A classification of groups of order $p^6$, Ph.D. thesis, University of Cambridge (1940).
							
\bibitem{Elias2010} B. Elias, L. Silberman and R. Takloo-Bighash, Minimal permutation representations of nilpotent groups, Experiment. Math., {\bf 19} (1) (2010), 121–128.
							
\bibitem{FAM} G.A. Fern\'andez-Alcober and A. Moret\'o, Groups with two extreme character degrees and their normal subgroups, Trans. Amer. Math. Soc. {\bf 353}(6) (2001), 2171-2192.
							
\bibitem{FORD} C.E. Ford, Characters of $p$-groups, 
Proc.\ Amer.\ Math.\ Soc., {\bf 101}(4), (1987), 595-601.

\bibitem{GAP} The GAP Group, GAP -- Groups, Algorithms, and Programming, Version 4.12.1; 2022, \href{https://www.gap-system.org}{www.gap-system.org}.

\bibitem{GA} G. Ghaffarzadeh and M.H. Abbaspour, Minimal degrees of faithful quasi-permutation representations for direct products of $p$-groups, Proc. Indian Acad. Sci. Math. Sci., {\bf 122}(3) (2012), 329-334.

\bibitem{HS} Marshall Hall, Jr., and James K.\ Senior, 
{\it The Groups of Order $2^n$ ($n \leq 6$)}, Macmillan, New York, 1964.

\bibitem{HEO}
Derek~F.\ Holt, Bettina Eick, and Eamonn~A.\ O'Brien,
\newblock {\em Handbook of computational group theory},
\newblock Chapman and Hall/CRC, London, 2005.
							
\bibitem{I} I.M. Isaacs, Character theory of finite groups, Academic 
Press, New York, 1976.
							
\bibitem{RJ} R. James, The groups of order $p^6$ ($p$ an odd prime), Math. Comp., {\bf 34}(150) (1980), 613-637.
							
\bibitem{DLJ} D.L. Johnson, Minimal permutation representations of finite groups, 
Amer.\ J.\ Math., {\bf 93}(4) (1971), 857-866.

\bibitem{S} S.R. Lemieux, Minimal degree of faithful permutation representations of finite degree, Carleton University, 1997.
							
\bibitem{L1} M.L. Lewis, The vanishing-off subgroup, J. Algebra, {\bf 321}(4) (2009), 1313-1325.
							
\bibitem{AM} A. Mann, Minimal characters of $p$-groups, J. Group Theory,  {\bf 2} (1999), 225-250.
							
\bibitem{NO'bV2004} M.F. Newman, E.A. O'Brien and M.R. Vaughan-Lee, Groups and nilpotent Lie rings whose order is the sixth power of a prime, J. Algebra, {\bf 278} (2004), 383–401.

\bibitem{Arxiv} M.F. Newman, E.A. O'Brien and M.R. Vaughan-Lee, 
Presentations for the groups of order $p^6$ for prime $p \geq 7$, 
\href{http://arxiv.org/abs/2302.02677}{arXiv:2302.02677 [math.GR]}

\bibitem{github}
E.A.\ O'Brien, S.K. Prajapati, and A. Udeep.
\emph{Minimal degree permutation representations for groups of order $p^6$}.
\href{https://github.com/eamonnaobrien/Minimal-Degree}{github.com/eamonnaobrien/Minimal-Degree}

\bibitem{SA} S.K. Prajapati and A. Udeep, On faithful quasi-permutation representations of VZ-groups and Camina $p$-groups, 
Comm. Algebra, {\bf 51} (4) (2023), 1431-1446.
							
\bibitem{SAcyclic} S.K. Prajapati and A. Udeep, Minimal faithful quasi-permutation representation degree of $p$-groups with cyclic center, \emph{accepted for publication in} Proc. Indian Acad. Sci. Math. Sci., \href{https://arxiv.org/abs/2302.06257v2}{arXiv:2302.06257v2 [math.GR]}
							
\bibitem{SaundersThesis} N. Saunders, Minimal Faithful Permutation 
Representations	of Finite Groups, Ph.D. thesis, The University of Sydney, 2011.
							
\bibitem{W} W.J. Wong, Linear groups analogous to permutation groups, J. Austral. Math. Soc., (Sec A),  {\bf 3} (1963), 180-184.		
							
\bibitem{DW} D. Wright, Degrees of minimal embeddings for some direct products, Amer. J. Math., {\bf 97}(4) (1975), 897-903.	
							
\end{thebibliography}
\end{document}